\documentclass[11pt,a4paper]{amsart}

\usepackage{amsthm,amsfonts,amssymb,amsmath}
\usepackage{stmaryrd}
\usepackage{enumitem}
\usepackage[french,english]{babel}
\usepackage[utf8]{inputenc}
\usepackage[T1]{fontenc}
\usepackage{lmodern}
\usepackage[all]{xy}
\usepackage{upref}
\usepackage{array}
\usepackage{color}
\usepackage{tikz-cd}
\usepackage{verbatim}
\usepackage{microtype}
\usepackage[a4paper]{geometry}
\usepackage{mathrsfs}
\usepackage{mathtools}
\usepackage{verbatim}
\usepackage{aliascnt}          

\usepackage{hyperref}
\usepackage[capitalise]{cleveref}

\newtheorem{theorem}{Theorem}[subsection]

\numberwithin{equation}{theorem}

\newaliascnt{prop}{theorem}
\newtheorem{prop}[prop]{Proposition}
\aliascntresetthe{prop}

\newaliascnt{lemma}{theorem}
\newtheorem{lemma}[lemma]{Lemma}
\aliascntresetthe{lemma}

\newaliascnt{coro}{theorem}
\newtheorem{coro}[coro]{Corollary}
\aliascntresetthe{coro}

\theoremstyle{definition}

\newaliascnt{rem}{theorem}
\newtheorem{rem}[rem]{Remark}
\aliascntresetthe{rem}

\newtheorem{secnumber}[theorem]{}

\newaliascnt{definition}{theorem}
\newtheorem{definition}[definition]{Definition}
\aliascntresetthe{definition}

\newaliascnt{conjecture}{theorem}
\newtheorem{conjecture}[conjecture]{Conjecture}
\aliascntresetthe{conjecture}

\newaliascnt{question}{theorem}

\aliascntresetthe{question}

\crefname{theorem}{theorem}{theorems}
\Crefname{theorem}{Theorem}{Theorems}

\crefname{prop}{proposition}{propositions}
\Crefname{prop}{Proposition}{Propositions}

\crefname{lemma}{lemma}{lemmas}
\Crefname{lemma}{Lemma}{Lemmas}

\crefname{coro}{corollary}{corollaries}
\Crefname{coro}{Corollary}{Corollaries}

\crefname{rem}{remark}{remarks}
\Crefname{rem}{Remark}{Remarks}

\crefname{definition}{definition}{definitions}
\Crefname{definition}{Definition}{Definitions}

\crefname{conjecture}{conjecture}{conjectures}
\Crefname{conjecture}{Conjecture}{Conjectures}

\crefname{question}{question}{questions}
\Crefname{question}{Question}{Questions}

\def\gg{\mathfrak{g}}

\def\cg{\check{\mathfrak{g}}}
\def\ct{\check{\mathfrak{t}}}

\def\cD{\mathscr{D}}
\def\PP{\mathscr{P}}
\def\XX{\mathfrak{X}}
\def\UU{\mathfrak{U}}
\def\YY{\mathfrak{Y}}

\def\Mod{\mathbf{Mod}}
\def\Coh{\mathbf{Coh}}

\def\Rep{\mathbf{Rep}}

\def\Ab{\mathbf{Ab}}

\def\Ad{\mathbf{Ad}}

\def\Gm{\mathbb{G}_{m}}
\def\Ga{\mathbb{G}_{a}}
\def\cG{\check{G}}

\def\cB{\check{B}}
\def\cT{\check{T}}
\def\fg{\mathfrak{g}}
\def\fF{\mathfrak{F}}

\def\fz{\mathfrak{z}}

\def\cfg{\check{\mathfrak{g}}}
\def\cft{\check{\mathfrak{t}}}
\def\A1{\mathbb{A}^1}
\def\P1{\mathbb{P}^1}

\def\bP{\mathbf{P}}
\def\bI{\mathbf{I}}
\def\bX{\mathbb{X}}
\def\bW{\mathbb{W}}
\def\bZ{\mathbb{Z}}

\def\sG{\mathsf{G}}
\def\sB{\mathsf{B}}

\def\cE{\mathcal{E}}
\def\HPi{\widehat{\Pi}}
\def\bQQ{\overline{\mathbb{Q}}}
\def\bQl{\overline{\mathbb{Q}}_{\ell}}
\def\Isod{\textnormal{Isoc}^{\dagger}}
\def\FIsod{\textnormal{F-Isoc}^{\dagger}}
\def\EE{\mathsf{E}}
\def\Ddd{\mathscr{D}^{\dagger}}
\def\OK{\mathcal{O}_K}
\def\bQ{\mathbb{Q}}
\def\Pn{\mathbb{P}^n}
\def\cF{\mathfrak{F}}
\def\mm{\mathfrak{m}}

\DeclareMathOperator{\Out}{Out}
\DeclareMathOperator{\Hol}{Hol}
\DeclareMathOperator{\Qcoh}{Qcoh}
\DeclareMathOperator{\Hk}{Hk}
\DeclareMathOperator{\pr}{pr}
\DeclareMathOperator{\Aut}{Aut}
\DeclareMathOperator{\MCF}{MCF}
\DeclareMathOperator{\red}{red}
\DeclareMathOperator{\cor}{coarse}
\DeclareMathOperator{\Pred}{Pro-red}
\DeclareMathOperator{\Pss}{Pro-ss}
\DeclareMathOperator{\rank}{rank}

\DeclareMathOperator{\coh}{coh}

\DeclareMathOperator{\ad}{ad}
\DeclareMathOperator{\Gal}{Gal}
\DeclareMathOperator{\Ker}{Ker}

\DeclareMathOperator{\Hom}{Hom}

\DeclareMathOperator{\Spec}{Spec}

\DeclareMathOperator{\rH}{H}

\DeclareMathOperator{\rD}{D}
\DeclareMathOperator{\rW}{W}

\DeclareMathOperator{\rR}{R}

\DeclareMathOperator{\Gr}{Gr}

\DeclareMathOperator{\dR}{dR}

\DeclareMathOperator{\id}{id}
\DeclareMathOperator{\rig}{rig}

\DeclareMathOperator{\an}{an}
\DeclareMathOperator{\Iso}{Isoc}
\DeclareMathOperator{\Del}{Del}
\DeclareMathOperator{\Be}{Be}
\DeclareMathOperator{\Kl}{Kl}
\DeclareMathOperator{\Ai}{Ai}
\DeclareMathOperator{\GL}{GL}
\DeclareMathOperator{\SL}{SL}
\DeclareMathOperator{\SO}{SO}
\DeclareMathOperator{\Conn}{Conn}
\DeclareMathOperator{\Frob}{Frob}
\DeclareMathOperator{\GR}{GR}
\DeclareMathOperator{\AS}{AS}

\DeclareMathOperator{\Bun}{Bun}

\DeclareMathOperator{\LocS}{LocSysm}
\DeclareMathOperator{\Tr}{Tr}
\DeclareMathOperator{\op}{op}
\DeclareMathOperator{\Sp}{Sp}
\DeclareMathOperator{\geo}{geo}

\DeclareMathOperator{\diff}{diff}
\DeclareMathOperator{\wild}{wild}
\DeclareMathOperator{\st}{st}
\DeclareMathOperator{\ord}{ord}

\DeclareMathOperator{\Frac}{Frac}
\DeclareMathOperator{\Irr}{Irr}
\DeclareMathOperator{\Swan}{Swan}

\DeclareMathOperator{\Fr}{Fr}
\DeclareMathOperator{\MC}{MC}

\DeclareMathOperator{\sss}{ss}
\DeclareMathOperator{\WD}{WD}
\DeclareMathOperator{\Nm}{Nm}

\newcolumntype{L}{>{$}l<{$}} 

\newcommand{\quash}[1]{}

\title{Frobenius structure on rigid connections and arithmetic applications}
\date{\today}
\author{Daxin Xu, Lingfei Yi}

\begin{document}
\selectlanguage{english}
\maketitle

\begin{abstract}
	We construct the natural Frobenius structures on two families of rigid irregular $\cG$-connections on $\Gm$ (or $\A1$) for a split simple group $\cG$: (i) the $\theta$-connections arising from Vinberg's $\theta$-groups introduced by Chen and Yun; (ii) the Airy connection of Jakob--Kamgarpour--Yi generalizing the classical Airy equations. 
	These data form the $p$-adic companions of the $\ell$-adic local systems introduced by Yun and Jakob--Kamgarpour--Yi. 
	Via the Frobenius structures, we study the local monodromy representations of these local systems at the unique wildly ramified point and verify the prediction of Reeder--Yu on epipelagic Langlands parameters in our setting. 
	We calculate the global geometric monodromy group of a special Airy $\cG$-local system via its local monodromy. 
	We show the cohomological rigidity and the physical rigidity of these local systems, as conjectured by Heinloth--Ng\^o--Yun. 
\end{abstract}

\tableofcontents
\section{Introduction}

\subsection{Rigid connections on curves: cohomological rigidity and physical rigidity}
\begin{secnumber}
	Let $j:U=\P1-S\hookrightarrow \P1$ be an open embedding of curves over a field $K$ of characteristic zero, where $S$ is a finite set of closed points. 
	Let $\cG$ be a connected split reductive group over $K$ and $\cfg$ its Lie algebra. 

	A $\cG$-bundle with connection $E$ on $U$ is called \textit{physically rigid} if it is uniquely determined by its local monodromies at $S$, as meromorphic $\cG$-connections over the associated Laurent series fields. 
	On the other hand, $E$ induces the adjoint local system $E(\cfg)$, 
	and the infinitesimal deformations of $E$ with fixed formal types at $S$ 
	are controlled by the cohomology of the middle extension $j_{!*}E(\cfg)$.
	We say $E$ is \textit{cohomologically rigid} if the following holds:
	\[
	\rH^*\bigl(\P1,\,j_{!*}E(\cfg)\bigr)=0,
	\]
	i.e. there are no non-trivial global infinitesimal deformations compatible with the prescribed formal local behaviour. 

	When $\cG=\GL_n$, Katz (in the regular singularity case) \cite{Katz96} and Bloch--Esnault \cite{BE04} showed that for irreducible connections, cohomological rigidity is equivalent to physical rigidity. 
	In general, there exist cohomologically rigid $\cG$-connections which are \textit{not} physically rigid \cite[\S 7.2]{KNP}. 

	On curves over finite fields, one has analogous notions for $\ell$-adic local systems (or for overconvergent isocrystals). 
	In this paper, we show that a $\cG$-overconvergent isocrystal is physically rigid provided it admits a physically rigid underlying algebraic $\cG$-connection. 
	This result will apply to theta/Airy $\cG$-connections once we equip them with Frobenius structures. 
\end{secnumber}

\begin{secnumber}
	\textbf{Classical Bessel/Airy connections}
	Let $x$ be a coordinate at $0\in\mathbb{P}^1$.
	The classical Bessel differential equation and Airy differential equation (of rank $n$) with a parameter $\lambda$ are defined by:
         \begin{equation}\label{Bessel connection intro}
		 \biggl(x\frac{d}{dx}\biggr)^n(f) - \lambda^n x\cdot f = 0,\quad \biggl(\frac{d}{dx}\biggr)^n(f)- \lambda^n x\cdot f=0.
         \end{equation}
	 They can be converted to rigid connections $\Be_n$ and $\Ai_n$ on the rank $n$ trivial bundle on $\mathbb{G}_{m,K}$ and $\mathbb{A}^1_K$ respectively:
	 \begin{equation} \label{eq:Bessel}
		\Be_n=	~d + \begin{pmatrix}
		0                  & 1 & 0 & \dots & 0 \\
		0                  & 0 & 1 & \dots & 0 \\
		\vdots          & \vdots & \ddots & \ddots & \vdots \\
		0                  & 0 & 0 & \dots & 1 \\
		\lambda^nx & 0 & 0 & \dots & 0  \end{pmatrix} \frac{dx}{x},
		\qquad 
			\Ai_n= ~ d + \begin{pmatrix}
		0                  & 1 & 0 & \dots & 0 \\
		0                  & 0 & 1 & \dots & 0 \\
		\vdots          & \vdots & \ddots & \ddots & \vdots \\
		0                  & 0 & 0 & \dots & 1 \\
		\lambda^nx & 0 & 0 & \dots & 0  \end{pmatrix} dx.
	\end{equation}
	Their differential Galois groups were calculated by Katz \cite{Katz87}. 
\end{secnumber}

\begin{secnumber}\textbf{$\theta$-connections.}\label{intro-sss:theta-conn}
	In the following, we assume moreover that $\cG$ is almost simple.  
	
	Via the principal grading on the Lie algebra $\cfg$ of $\cG$, Frenkel and Gross wrote down an explicit $\cG$-connection on $\mathbb{G}_{m}$, called \textit{Bessel $\cG$-connection}, which specializes to $\Be_n$ when $\cG=\SL_n$ \cite{FG09}. 
	In an unpublished work, Yun constructed a generalization of Bessel $\cG$-connection from a stable grading of $\cfg$, called \textit{$\theta$-connection}. In \cite{Chen17}, Chen proved under a condition on the stable grading that these connections are cohomologically rigid. 

	Let $\theta$ be an inner stable grading of $\cfg$ over $K$ corresponding to the regular elliptic number $m$ via \cite[Corollary 14]{RLYG12}. 
	Concretely, we may assume $\theta$ is given by the adjoint action of $\exp(x)$, with $x\in \mathbb{X}_*(\cT)\otimes_{\mathbb{Z}}\mathbb{R}$, such that $\check{\lambda}=mx\in\mathbb{X}_*(\cT) $. 
	Then $\Ad_{\check{\lambda}(\zeta_m)}$(resp. $\check{\lambda}$) defines an order $m$ grading $\cfg=\bigoplus_{l\in \mathbb{Z}/m\mathbb{Z}} \cfg_l$ (resp. a $\mathbb{Z}$-grading $\cfg=\bigoplus_{l\in \mathbb{Z}}\cfg(l)$). 
	We assume that the (normalized) Kac coordinate $s_0:=\alpha_0(x)$ equals $1$. 
	We refer to \cite[(2)]{CY24} and \cite[\S 7]{RLYG12} for a list of regular elliptic numbers $m$ and those satisfying $s_0=1$. 
	We take $X=X_1+X_{1-m}\in\cfg_1=\cfg(1)\oplus\cfg(1-m)$ to be a stable element. 
	The $\theta$-connection associated to $X$ and a parameter $\lambda\in K$ is the $\cG$-connection on $\Gm$, defined by
\begin{equation} \label{intro-eq:theta-connection}
	\Theta(X,\lambda):= d+ \biggl( X_1+\lambda^m X_{1-m} x\biggr)\frac{dx}{x}.
\end{equation}
	When $m=h$ is the Coxeter number, 
	the above $\cG$-connection is the Bessel $\cG$-connection. 
\end{secnumber}

\begin{secnumber}\textbf{Airy connections.}
	Let $m=h$ and $X=X_1+X_{1-h}\in \cfg_1$ be a stable element as in \ref{intro-sss:theta-conn}. 
	In the introduction, we consider a special family of \textit{Airy $\cG$-connections on $\A1$} \eqref{eq:naive Ai-conn}, which specializes to $\Ai_n$ \eqref{eq:Bessel} when $\cG=\GL_n$: 
	\begin{equation} \label{intro-eq:Ai-conn}
		\Ai(X,\lambda)=d+(X_1+\lambda^h X_{1-h} x)dx.
	\end{equation}
\end{secnumber}

\subsection{Frobenius structures}
\begin{secnumber} \textbf{Summary of main results.}
	In the 1970s \cite{Dw74}, Dwork and Sperber showed that there exists a Frobenius structure on the Bessel connection \eqref{eq:Bessel} over the $p$-adic field $K=\mathbb{Q}_p(\zeta_p)$ whose Frobenius traces give Kloosterman sum. 
	By a \textit{Frobenius structure} on a connection, we mean a horizontal isomorphism between the Bessel connection and its pullback by the ``Frobenius endomorphism'' $F:\mathbb{G}_{m,K}\to \mathbb{G}_{m,K}$ over $K$ defined by $x\mapsto x^{p}$. 

	In \cite{XZ22}, Xu--Zhu constructed the Frobenius structure on the Bessel $\cG$-connection, building on the geometric Langlands correspondence of Kloosterman local systems for reductive groups due to Heinloth--Ng\^o--Yun \cite{HNY} and of Zhu \cite{Zhu17}. 
	In the present paper, we extend the argument of \cite{XZ22} to construct the natural Frobenius structures on $\theta$-connections (resp. Airy connections). 
	Equipped with the Frobenius structure, the resulting object forms a $\cG$-overconvergent $F$-isocrystal $\Kl_{\cG}^{\rig}$ on $\mathbb{G}_{m,\mathbb{F}_p}$ (resp. $\Ai_{\cG}^{\rig}$ on $\mathbb{A}^1_{\mathbb{F}_p}$). 
	Moreover, for a prime $\ell\neq p$, $\Kl_{\cG}^{\rig}$ (resp. $\Ai_{\cG}^{\rig}$) is the $p$-adic companion of the $\ell$-adic generalized Kloosterman sheaf $\Kl_{\cG}^{\ell}$ constructed in \cite{Yun16} (resp. Airy sheaves $\Ai_{\cG}^{\ell}$ constructed in \cite{JKY}) in appropriate sense. 

	Via these Frobenius structures, we obtain the following arithmetic applications:
		
	(i) We explicitly study the local monodromy representation of $\Kl_{\cG}^{\rig}$ (resp. $\Ai_{\cG}^{\rig}$)  at the unique wildly ramified point $\infty$ and verify the prediction of Reeder--Yu on epipelagic Langlands parameter in our setting \cite[\S~7.1]{RY14}. 

	(ii) In the Airy case, we apply the local monodromy to calculate the global geometric monodromy group of $\Ai_{\cG}^{\rig}$. 

	(iii) We deduce the physical rigidity of $\Kl_{\cG}^{\rig}$ (resp. $\Ai_{\cG}^{\rig}$)  as $\cG$-overconvergent isocrystals from that of its underlying connection. 
	Via the $p$-$\ell$ companion, we verify a conjecture of Heinloth--Ng\^o--Yun on the physical rigidity of the $\ell$-adic Kloosterman sheaves for reductive groups \cite[Conjecture 7.1]{HNY}
	when the tame monodromy at $0$ is unipotent. 
\end{secnumber}
\begin{secnumber}
	To state the main result, we recall \textit{the ring of $p$-adic analytic functions on $\P1$ overconvergent along $\{\infty\}$} \cite{Ber96}. 
	We set $K=\mathbb{Q}_p(\mu_p)$, equipped with a $p$-adic valuation $|\textnormal{-}|$ normalized by $|p|=p^{-1}$, and denote by $A^{\dagger}$ the ring of $p$-adic analytic functions with a radius of convergence $>1$:
	\begin{equation} \label{def A dagger}
		A^{\dagger}=\bigg\{ \sum_{n=0}^{+\infty} a_nx^n ~|~ a_n\in K, \exists~ \rho>1, \lim_{n\to +\infty} |a_n|\rho^n =0 \bigg\}.  
	\end{equation}
\end{secnumber}

\begin{theorem}[\ref{ss:Frob-theta-connection},\ref{ss:compare-dR-rig-Airy}] \label{intro-t:Frobenius}
	Let $K=\mathbb{Q}_p(\mu_p)$, $\overline{K}$ an algebraic closure of $K$ and set $\lambda=-\pi\in K$ satisfying $\pi^{p-1}=-p$. 
	Assume $(p,m)=1$ (resp. $p>h$) and that the Kac coordinate $s_0=1$ and the same holds for the order $m$ inner stable grading of the Lie algebra $\mathfrak{g}$ of the dual group $G$ of $\cG$. 
	There exists a natural Frobenius structure $\varphi(x)\in \cG(A^{\dagger})$ defining a gauge transform between $\theta$-connection (resp. Airy connection) and its Frobenius pullback: 
	\[
	x \frac{d\varphi}{dx} \varphi^{-1} + \Ad_{\varphi} (X_1+\lambda^m X_{1-m}x) = p(X_1+\lambda^m X_{1-m} x^p), \]
	\[\textnormal{(resp.} ~ 
	\frac{d\varphi}{dx} \varphi^{-1} + \Ad_{\varphi} (X_1+\lambda^h X_{1-h}x) = px^{p-1}(X_1+\lambda^h X_{1-h} x^p). )
\]
\end{theorem}

The trace of the values of Frobenius structure at the Teichmüller lifting of elements of $\mathbb{F}_p$ gives exponential sums, including Kloosterman sums. We refer to \cite[\S 5.2]{XZ22} for the explicit exponential sums of Kloosterman sheaves for classical groups. 
When $\cG=\SL_n$, the Frobenius trace on $\Ai_n$ \eqref{eq:Bessel} gives the Airy sums \cite{Katz87}, for $a\in \mathbb{F}_p$ and $[a]$ its Teichmüller lifting: 
\[
\Tr(\varphi([a]))= \sum_{x\in \mathbb{F}_p} \psi\biggl(\frac{x^{n+1}}{n+1}+a x\biggr),
\]
where $\psi:\mathbb{F}_p\to K^{\times}$ is the additive character associated to $\pi$ (\ref{sss:def-isod}). 

\begin{secnumber}
	The proof of the above theorem consists of three ingredients: 

	(i) Following Yun \cite{Yun16} and Jakob--Kamgarpour--Yi \cite{JKY}, we produce a $\cG$-overconvergent $F$-isocrystal $\Kl_{\cG}^{\rig}$ (resp. $\Ai_{\cG}^{\rig}$) and a $\cG$-connection $\Kl_{\cG}^{\dR}$ (resp. $\Ai_{\cG}^{\dR}$) as the Langlands parameter of a certain automorphic function. 

	(ii) Then we show that the overconvergent isocrystal $\Kl_{\cG}^{\rig}$ (resp. $\Ai_{\cG}^{\rig}$) is isomorphic to the analytification of the $\cG$-connection $\Kl_{\cG}^{\dR}$ (resp. $\Ai_{\cG}^{\dR}$) as in \cite{XZ22} (see \Cref{t:compare-Kl-rig-dr,t:compare-Ai-rig-dr}). 

	(iii) Chen--Yi and Yi \cite{CY24,Yi25} showed that $\Kl_{\cG}^{\dR}$ (resp. $\Ai_{\cG}^{\dR}$) is isomorphic to $\Theta(X,\lambda)$ (resp. $\Ai(X,\lambda)$) as algebraic $\cG$-connections (see \Cref{t:compare-Kl-theta-connection,t:Ai connection}). 

	Consequently, $\Theta(X,\lambda)$ (resp. $\Ai(X,\lambda)$) can be equipped with a natural Frobenius structure and the resulting $\cG$-overconvergent isocrystal is isomorphic to $\Kl_{\cG}^{\rig}$ (resp. $\Ai_{\cG}^{\rig}$). 
\end{secnumber}

\begin{rem}
	(i) For rigid connections with regular singularities over an unramified $p$-adic field, Esnault--Groechenig constructed the natural Frobenius structure under suitable hypotheses \cite{EG20,EG25}. 
	Their proof is based on the Frobenius endomorphism map on the moduli of bundles with connection with prescribed regular singularities. 
	In their setting, Frobenius structures are given by rational functions. 

	(ii) 
	The connections considered above are irregular and the Frobenius structures are $p$-adically analytic functions;
this is a different situation from the regular singular case. 
	It is natural to expect that an irregular rigid connection with suitable coefficients admits a Frobenius structure, 
but the method of Esnault--Groechenig does not apply in the irregular case. 
		In our setting, the construction instead uses an explicit geometric Langlands description.

	We refer to \cite{JY23,HJ24, KS, KNP} for more examples of cohomologically/physically rigid connections on $\mathbb{P}^1-\{0,\infty\}$, including $\theta$-connections and Airy connections considered above. 
\end{rem}

\subsection{Application to monodromy representations and rigidity}

\begin{secnumber}\textbf{Local monodromy representations.}
	The restriction of $\Kl_{\cG}^{\rig}$ (resp. $\Ai_{\cG}^{\rig}$) at the unique wildly ramified point $\infty$ defines a differential module $\mathscr{E}$ with Frobenius structure on the Robba ring. 
	Let $F=\mathbb{F}_p(\!( t )\!)$ be a local field of equal characteristic, $I_F$ the inertia subgroup of $\Gal(\overline{F}/F)$, $I_F^+$ the wild inertia subgroup and $W_F$ the Weil group. 
	By the $p$-adic local monodromy theorem \cite{And02II,Ked04,Meb02}, we can associate to $\mathscr{E}$ a Weil--Deligne representation $(\rho,N)$ with coefficients in $\overline{K}$:
	\[
	\rho:W_F\to \cG(\overline{K}), \quad N\in \cfg. 
	\]

	The epipelagic representation, introduced by Reeder--Yu \cite{RY14}, is the local component of the automorphic function of $\Kl_{\cG}^{\rig}$ at infinity. 
	Hence, $(\rho,N)$ can be viewed as the local Langlands parameter associated to the epipelagic representations via the local Langlands correspondence. 

	Comparing the arithmetic local monodromy of $\mathscr{E}$ with the local monodromy of the underlying formal connection of $\mathscr{E}$ (\Cref{p:compare-monodromy-group}), we prove the following result. 
	In particular, this verifies the prediction of Reeder--Yu on the epipelagic Langlands parameter \cite[\S~7.1]{RY14}. 
\end{secnumber}

\begin{theorem}[\Cref{c:Swan-conductor}, \Cref{p:Airy-coh-rig}] \label{intro-t:local-monodromy}
	In the $\theta$-case (resp. Airy case), under the assumptions of \Cref{intro-t:Frobenius}, the local monodromy representation $(\rho,N)$ satisfies:

	\textnormal{(i)} 
	The nilpotent operator $N$ vanishes. 
	Let $\cT_X$ be the centralizer $C_{\cG}(X)$ of $X=X_1+X_{1-m}$ (resp. $X_1+X_{1-h}$) in $\cG$, which is a maximal torus. 
	The representation $\rho$ fits into a commutative diagram
	\[
	\xymatrix{
	1\ar[r] & I^+_F \ar[r] \ar[d] & I_F \ar[r] \ar[d] & I_F^t \ar[r] \ar[d] & 1 \\
	1\ar[r] & \cT_X \ar[r]& N_{\cG}(\cT_X) \ar[r] & \bW \ar[r] & 1
	}
	\]
	such that a generator of the tame inertia $I_F^t$ is sent to a regular elliptic element of the Weyl group $\bW$ of order $m$ (resp. a Coxeter element). 

	\textnormal{(ii)} We have $\cfg^{I_F}=0$ and $\Swan(\cfg)=\frac{\sharp \Phi}{m}$ (resp. $=\sharp \Phi\cdot \frac{1+h}{h}$), where $\sharp \Phi$ is the number of roots of $\cfg$. 
\end{theorem}

As explained in \cite[Proposition 5.2]{Yun16} and \cite[6.3]{JKY}, 
Theorem \ref{intro-t:local-monodromy} implies the cohomological rigidity:

\begin{coro}
	The $\cG$-overconvergent isocrystals $\Kl_{\cG}^{\rig}$, $\Ai_{\cG}^{\rig}$ are cohomologically rigid. 
\end{coro}

\begin{rem}
	(i) In the $\theta$-case with $m=h$, the representation $\rho$ is the \textit{simple wild parameter} introduced by Gross--Reeder \cite{GR10}. 
	For simple wild parameters, \Cref{intro-t:local-monodromy}(i) is proved in \cite[Proposition 9.4]{GR10} under the assumption that $p \nmid \sharp \bW$, and \Cref{intro-t:local-monodromy}(ii) is proved in \cite[Theorem 2, Corollary 2.15]{HNY} ($\ell$-adic setting) and \cite[Corollary 4.5.9]{XZ22} without any assumption on $p$. 

	(ii) However, \Cref{intro-t:local-monodromy}(i) fails for a simple wild parameter $\rho$ if $p\mid h$. 
	For instance, when $\cG=\SL_p$, one may expect the image $\rho(I_F^+)$ to be a non-abelian group of order $p^3$ and its tame quotient to be cyclic of order $p+1$. 
	This is known for $p=2,3$ \cite{And02,Qin24}. 

	(iii) Motivated by the theory of formal isoclinic connections \cite{JY23}, we introduce the notion of \textit{$p$-adic isoclinic connection} over the open unit disc in \S~\ref{ss:p-adic-isoclinic}. 
	For such connections, we describe the associated Weil--Deligne representations in \Cref{t:L-parameter}. 
	We expect that these Weil--Deligne representations correspond, under the local Langlands correspondence, to toral supercuspidal representations \cite{Adl98,Kal19}. 
	In our situation, we show that the differential module $\mathscr{E}$ at $\infty$ is $p$-adically isoclinic, and then deduce the above theorem from the general results of \Cref{t:L-parameter}. 
    
    (iv)
	The representation $\rho$ is an epipelagic parameter in the sense of Kaletha \cite{Kal15}. We expect our construction via the global Langlands correspondence is compatible with Kaletha's purely local construction and we will explore the relationship in the future. 
\end{rem}

\begin{secnumber}
	\textbf{Global monodromy of the Airy local system.} 
	Frenkel--Gross \cite[Corollary 10]{FG09} and Kamgarpour--Sage \cite[Theorem 12]{KS} computed the differential Galois group $G_{\diff}$ of the Bessel and Airy $\cG$-connections \eqref{intro-eq:Ai-conn} over $\overline{K}$ respectively. 
	The two results coincide, and we list them in the following table (up to central isogeny):
	\begin{equation}\label{eq:table monodromy}
	\begin{tabular}{L|L}
		\check{G} & G_{\diff} \\
		\hline
		A_{2n} & A_{2n} \\
		A_{2n-1}, C_{n} & C_{n} \\
		B_{n},D_{n+1} (n\ge 4) & B_{n} \\
		E_{7} & E_{7} \\
		E_{8} & E_{8} \\
		E_{6}, F_{4} & F_{4} \\
		B_{3}, D_{4}, G_{2} & G_{2}.
	\end{tabular}
	\end{equation}
	When $m=h$, the geometric monodromy group of $\Kl_{\cG}^{\rig}$ (and its $\ell$-adic companion) was computed in \cite{Katz88,HNY,XZ22}.

	If $G_{\geo}$ denotes the geometric monodromy group of the Airy $\cG$-overconvergent isocrystal $\Ai_{\cG}^{\rig}$ over $\overline{K}$, there exists a canonical embedding \cite[\S~4.5.1]{XZ22}:
\[ G_{\geo}\hookrightarrow G_{\diff}.\]
\end{secnumber}

\begin{theorem}[\Cref{t:monodromy-Ai}]
	Let $n$ be the minimal dimension of faithful representations among almost simple groups with Lie algebra $\cfg$. 
	If $p>2n+1$, then the above homomorphism is an isomorphism.
	\label{intro-t:global-mono}
\end{theorem}
\begin{rem}
	Via the companion, we can also recover the geometric monodromy group of the $\ell$-adic Airy sheaves $\Ai_{\cG}^{\ell}$. 
	In this way, We partially generalize the results of Katz and of Šuch for $\ell$-adic Airy sheaves of rank $n$ \cite{Katz87,Such}.
\end{rem}

The proof of Kamgarpour--Sage is based on a general principle: if a $\cG$-connection has the maximal slope $\frac{a}{h}$ with the Coxeter number $h$ and $(a,h)=1$, then its differential Galois group $G_{\diff}$ has a fundamental degree $h$. 
This result fails for arithmetic local system in positive characteristic. 
For instance, when $p=2$, the geometric monodromy group of the rank $2n+1$ Bessel $F$-isocrystal is $\SO_{2n+1}$ ($n\neq 3$) or $G_2$ ($n=3$) \cite[Theorem 4.5.2]{XZ22}, whose maximal fundamental degree is $2n$. However, its $p$-adic slope is $\frac{1}{2n+1}$. 
In our proof, we first show that $G_{\geo}$ is connected under the assumption $p>2n+1$. 
Since the local monodromy (\Cref{intro-t:local-monodromy}) is contained in $G_{\geo}$, we deduce that the Weyl group of $G_{\geo}$ has an element of order $h$ and hence $G_{\geo}$ has a fundamental degree $h$. 
This allows us to conclude \Cref{intro-t:global-mono}. 

\begin{secnumber}
	\textbf{Physical rigidity.} 
	Given an open embedding $j:U\to \P1$ over $\mathbb{F}_p$, a $\cG$-overconvergent isocrystal $\mathcal{E}$ on $U$ is called \textit{physically rigid} if it is uniquely determined by its local monodromy representations at the boundary $\P1-U$ (see \S~\ref{sss:phy-rigid}). 
	We prove that if there exists a physically rigid algebraic $\cG$-connection $E$ underlying $\mathcal{E}$, then $\mathcal{E}$ is also physically rigid as $\cG$-overconvergent isocrystals (see \Cref{t:physical-rigidity}). 

	In \cite{Yi24,CY24,Yi25}, the second author and T.H. Chen showed that the $\theta$-connection and Airy connection considered above are physically rigid. 
	In \cite{HJ24}, Hohl--Jakob prove the physical rigidity in a more general setting. 
	As an application, we can deduce the following result. 
\end{secnumber}

\begin{theorem}[\Cref{c:strong-phy-rigidity}]
	\label{intro-t:phy-rigid}	
	The $\cG$-overconvergent isocrystals $\Kl_{\cG}^{\rig}$, $\Ai_{\cG}^{\rig}$ are physically rigid. 
\end{theorem}

\begin{secnumber}
	\textbf{Companions and the passage between primes.} 
	For a prime $\ell\neq p$, we expect to recover the above $p$-adic rigidity result for $\ell$-adic local systems via the theory of companion. 
	For $\GL_n$-local systems on curves, the theory of crystalline companions is established by Abe \cite{Abe18}. 
	For every representation $V\in \Rep(\cG)$,  $\Kl^{\rig}_{\cG}(V)$ (resp. $\Ai^{\rig}_{\cG}(V)$) is the $p$-adic companion of $\Kl^{\ell}_{\cG}(V)$ (resp. $\Ai^{\ell}_{\cG}(V)$). 
	This allows us to deduce that $\Kl^{\ell}_{\cG}$ and $\Ai^{\ell}_{\cG}$ are cohomologically rigid by comparing the local monodromy representation at $\infty$ defined by the adjoint representation \Cref{p:companion-WD-rep}. 

	In this paper, we also use a notion of companion of $\cG$-local systems on curves for a connected semisimple group $\cG$ due to Drinfeld \cite{Dri18} (see \S~\ref{ss:companion}). 
	When $m=h$, $\Kl_{\cG}^{\rig}$ is the $p$-adic companion of the $\ell$-adic Kloosterman sheaf $\Kl_{\cG}^{\ell}$ in the above sense (\Cref{p:companion-Kl}). 
	By relating the local monodromy representations of companion local systems, we deduce the physical rigidity of $\Kl_{\cG}^{\ell}$ from \Cref{intro-t:phy-rigid}, as conjectured in \cite[Conjecture 7.1]{HNY}. 
\end{secnumber}

\begin{theorem}[\Cref{t:physical-rig-l}]
	Assume that $p$ does not divide the order of the Weyl group $\bW$. 
	Then $\Kl_{\cG}^{\ell}$ is physically rigid. 
	\label{intro-t:phy-rigid-l}
\end{theorem}

\subsection{Organisation of the paper}

	In \S~\ref{s:pde}, we review some properties of overconvergent $F$-isocrystals and $p$-adic differential equations needed in the sequel. 
	In \S~\ref{s:p-isoclinic}, we introduce the notion of $p$-adic isoclinic connection and study the associated Weil--Deligne representation. 
	In \S~\ref{s:Frobenius}, we construct the Frobenius structure on $\theta$ and Airy connections \Cref{intro-t:Frobenius}.
	Section~\ref{s:application} is devoted to applications to local and global monodromy representations of $\Kl_{\cG}^{\rig}$ and $\Ai_{\cG}^{\rig}$. 
	Finally, Section~\ref{s:rigidity} establishes the physical rigidity results in both the $p$-adic and the $\ell$-adic settings. 

\textbf{Acknowledgement.} 
	We would like to thank Tomoyuki Abe, Tsao-Hsien Chen, Weizhe Zheng, Xinwen Zhu for valuable discussions. 
	Daxin Xu is supported by National Key R\&D Program of China
2025YFA1018000, National Natural Science Foundation of China Grant no. 12288201 and CAS Project for Young Scientists in Basic Research, Grant No. YSBR-033.
	Lingfei Yi is supported by the National Key R\&D Program of China No. 2024YFA1014600.

\section{Review of $p$-adic differential equations} \label{s:pde}

	In this article, we fix a prime number $p$. 
	Let $k$ be a finite field, $\overline{k}$ an algebraic closure of $k$ and $\mathcal{O}_K$ a complete discrete valuation ring of mixed characteristic 
	with residue field $k$. 
	We set $K=\Frac(\mathcal{O}_K)$, equipped with the valuation $|\cdot|$ normalized by $|p|=p^{-1}$. 
	We fix an algebraic closure $\overline{K}$ of $K$.
	Let $\mathcal{O}_{\overline{K}}$ (resp. $\mathfrak{m}_{\overline{K}}$) be the ring of integers in $\overline{K}$ (resp. the maximal ideal of $\mathcal{O}_{\overline{K}}$). 
	Let $\mathbb{C}_p$ be the $p$-adic completion of $\overline{K}$. 

	Let $s$ be a positive integer and set $q=p^s$. 
	To discuss Frobenius structure, we assume that the $s$-th Frobenius endomorphism $k\xrightarrow{\sim} k,~ x\mapsto x^q$ lifts to an automorphism $\sigma:\mathcal{O}_K\xrightarrow{\sim}\mathcal{O}_K$. 
	Then, $\sigma$ can be extended to an automorphism $\sigma:\overline{K}\to \overline{K}$ and there exists a sequence of finite extensions $K_n$ of $K$ in $\overline{K}$ such that $\sigma(K_n)=K_n$ and that $\cup_{n} K_n=\overline{K}$.

\subsection{Overconvergent ($F$-)isocrystals}

\begin{secnumber} \label{sss:def-isod}
	Let $X$ be a $k$-scheme. We denote by $\FIsod(X/K)$ the category of overconvergent $F$-isocrystals on $X$ (relative to $K$ and $\sigma$) and refer to \cite{Ber96} for the definition. 
	We denote by $\Isod(X/K)$ the thick full subcategory of the category of overconvergent isocrystals on $X$ relative to $K$, generated by those that can be endowed with an $s'$-th Frobenius structure for some integer $s'$ divisible by $s$. 
	Note that this category is independent of the choice of $\sigma$. 

	We assume that there exists an element $\pi$ in $K$ satisfying $\pi^{p-1}=-p$. 
	If $x$ denotes a coordinate of $\A1$, the connection 
	$d -\pi dx$ \footnote{There is a typo in \cite[2.1.1]{XZ22}, $d+\pi dx$ should read $d-\pi dx$.}
	on $\mathcal{O}_{\A1_K}$ underlies an object of $\Iso^{\dagger}(\A1_k/K)$ and is called \textit{Dwork isocrystal}, denoted by $\mathscr{A}_{\pi}$. 
	A $p$-th Frobenius structure on $\mathscr{A}_{\pi}$ is given by the overconvergent function $\theta_{\pi}(x)=\exp(\pi(x-x^p))$. 

	There exists a unique nontrivial additive character $\psi:\mathbb{F}_p\to K^{\times}$ satisfying $\psi(1)=1+\pi \mod \pi^{2}$. 
	For each $x\in \mathbb{F}_{p}$, we denote by $[x]$ the Teichmüller lifting of $x$ in $\mathbb{Q}_p$. Then $\theta_{\pi}([x])=\psi(x)$ \cite[1.4]{Ber84}. 
	So the Frobenius trace function of $\mathscr{A}_{\pi}$ is equal to $\psi\circ\Tr_{k/\mathbb{F}_p}(-)$. 
	We also denote $\mathscr{A}_{\pi}$ by $\mathscr{A}_{\psi}$, as it plays a similar role of the Artin--Schreier sheaf associated to $\psi$ in the $\ell$-adic theory. 	
\end{secnumber}

\begin{secnumber} \label{sss:Robba-ring}
	We denote by $\mathcal{A}_K$ the $K$-algebra of analytic functions on the open unit disc $|t|<1$, i.e.
	\begin{equation} \label{eq:functions on open disc}
		\mathcal{A}_K=\{\sum_{n\ge 0} a_n t^n\in K\llbracket t\rrbracket;~ |a_n|\rho^n\to 0~ (n\to \infty)~ \forall \rho\in [0,1)~\}.
	\end{equation}
	We also denote by $\mathcal{R}_K$ the Robba ring of analytic functions on certain annuli: 
	\begin{equation}
		\mathcal{R}_K= \{ \sum_{n\in \mathbb{Z}} a_n t^n \in K\llbracket t,t^{-1}\rrbracket;~ |a_n|\rho^n \to 0 ~(n\to \infty), \forall \rho\in (\rho_0,1) ~\textnormal{ for some } \rho_0\in (0,1)\}. 	
		\label{eq:Robba-ring}
	\end{equation}

	For any algebraic extension $L$ of $K$, we set $\mathcal{A}_L=\mathcal{A}_K\otimes_K L$ and $\mathcal{R}_L=\mathcal{R}_K\otimes_K L$, and omit the subscript $\overline{K}$ when $L=\overline{K}$. 
	We set $\mathcal{A}_{K,t}=\mathcal{A}_K[\frac{1}{t}]$ and $\mathcal{A}_t=\mathcal{A}[\frac{1}{t}]$. 
	We have natural inclusions
	\begin{equation}
		\mathcal{A}_t\to \overline{K}(\!(t)\!),\quad \mathcal{A}_t\to \mathcal{R}.
		\label{eq:inclusion-rings}
	\end{equation}

	The rings $\mathcal{R}, \mathcal{A}_t$ are equipped with the differential operator $d$. 
	We denote by $\MC(\mathcal{R}_K)$ (resp. $\MC(\mathcal{R})$) the category of finite free differential modules over $\mathcal{R}_K$ relative to $K$ (resp. $\mathcal{R}$ relative to $\overline{K}$). 
	Every object of $\MC(\mathcal{R})$ comes from the extension of scalars of an object of $\MC(\mathcal{R}_L)$ for a finite extension $L$ of $K$.
	The category $\MC(\mathcal{R})$ is a Tannakian category over $\overline{K}$. 
\end{secnumber}

\begin{secnumber} \label{sss:MCF}
	There is a $\sigma$-linear endomorphism $\phi$ of $\mathcal{R}_K$, defined by $t\mapsto t^{q}$. 
	We say an object $M$ of $\MC(\mathcal{R}_K)$ can be equipped with a Frobenius structure if there exists an isomorphism $\varphi:\phi^*M\simeq M$. 
	The isomorphism $\varphi$ is called a \textit{Frobenius structure} on $M$ (with respect to $\sigma$). 

	By a theorem of Christol--Mebkhout \cite[\S~6]{CM01}, the full subcategory $\MCF(\mathcal{R})$ of $\MC(\mathcal{R})$, generated by sub-quotients of objects that can be equipped with a Frobenius structure, is a Tannakian subcategory and is stable by extension. 
	Moreover, by an argument of Dwork, every object of $\MCF(\mathcal{R})$ is solvable (see \ref{sss:slopes} for the definition). 

	The $p$-adic local monodromy theorem \cite{And02II,Ked04,Meb02} says that every object of $\MCF(\mathcal{R})$ is quasi-unipotent. 
	Let $F=k(\!( t )\!)$ be the local field, $I_F$ the inertia subgroup of $\Gal(\overline{F}/F)$ and $I_F^+$ the wild inertia subgroup. 
	In particular, the Tannakian group of $\MCF(\mathcal{R})$ is isomorphic to $I_F\times \Ga$, where $I_F$ is viewed as a profinite group \cite[Th\'eor\`eme 7.1.1]{And02}, i.e. we have a tensor equivalence:
	\begin{equation}
		\MCF(\mathcal{R})\simeq \Rep_{\overline{K}}(I_F\times \Ga).
		\label{eq:MCFR}
	\end{equation}

	Assume that $k=\mathbb{F}_q$ and $\sigma=\id_K:K\to K$.  
	Let $W_F$ be the Weil group of $F$. 
	Let $(\mathscr{E},\varphi)$ be a differential module over $\mathcal{R}_K$ of rank $n$ with a Frobenius structure. 
	By \cite[\S~3.2, 3.4]{Mar08}, we can associate to $(\mathscr{E},\varphi)$
	a Weil--Deligne representation $(\rho,N)$ with coefficients in $\overline{K}$ extending \eqref{eq:MCFR}:
	\[
	\rho: W_F \to \GL_n(\overline{K}), \quad N\in \mathfrak{gl}_n ~ \textnormal{nilpotent},
	\]
	such that $\rho(I_F)$ is finite and $\rho(\Frob)N \rho(\Frob)^{-1}=q N$. 
\end{secnumber}

\begin{secnumber}
	We recall the definition of subsidiary radii following \cite[\S~11]{pde}. 

	Let $M$ be a finite free differential module of rank $n$ over $K\langle \alpha/x,x/\beta\rangle$ for $0\le \alpha<\beta$. 
	For $\rho\in[\alpha,\beta]$, let $F_\rho$ be the completion of $K(x)$ under the $\rho$-Gauss norm $|\cdot|_{\rho}$ \cite[Definition 9.4.1]{pde}. Let $R_i(M,\rho)$ be the (extrinsic) subsidiary radii of $M_\rho=M\otimes F_{\rho}$ \cite[\S~11.3]{pde} listed in increasing order, so that $R_1(M,\rho)$ is the generic radius of convergence of $M_{\rho}$. 
	For $r\in [-\log \beta,-\log \alpha]$, we define
	\[f_i(M,r)=-\log R_i(M,e^{-r}),\]
	so that $f_i(M,r)\ge r$ for all $r$. 
	We set 
	\begin{equation} \label{eq:function-Fi}
	F_i(M,r)=f_1(M,r)+\cdots+f_i(M,r).
\end{equation}

	These functions satisfy the following properties \cite[Theorem 11.2.1]{pde}:
	\begin{itemize}
		\item (Linearity) For $i=1,\cdots,n$, $f_i(M,r),F_i(M,r)$ are continuous and piecewise linear. 
			Moreover, there are only finitely many different slopes. 

		\item (Convexity) For $i=1,\cdots,n$, $F_i(M,r)$ is convex. 
	\end{itemize}
\end{secnumber}

\begin{secnumber} \label{sss:slopes}
	Let $M$ be a finite free differential module over the Robba ring $\mathcal{R}$. 
	Recall that $M$ is \textit{solvable} if $\lim_{\rho\to 1^-} R_1(M,\rho)=1$. 
	Suppose $M$ is solvable, then there exist  non-negative rational numbers $\alpha_1\ge \dots \ge \alpha_n$ and some $\varepsilon\in (0,1)$ such that \cite[Theorem 4.2-1]{CM00}
	\begin{equation}
		R_i(M,\rho)=\rho^{1+\alpha_i}, \quad \forall 1\le i\le n,\quad \forall \rho\in (1-\varepsilon,1).
		\label{eq:def-p-slope}
	\end{equation}
	
	The rational numbers $\alpha_1\ge \alpha_2 \ge \dots \ge \alpha_n\ge 0$ are called \textit{$p$-adic slopes} of $M$ and $\alpha_1$ is called the \textit{maximal $p$-adic slope}. 
	The \textit{($p$-adic) irregularity} of $M$ is defined by the sum of $p$-adic slopes:
	\[
	\Irr(M)=\sum_{i=1}^n \alpha_i,
	\]
	which is an integer. Via the equivalence \eqref{eq:MCFR}, $\Irr(M)$ equals to the Swan conductor of the representation associated to $M$ \cite{Tsu98}. 
\end{secnumber}

\subsection{A theorem of Clark}

Consider a linear differential equation
\begin{equation}
	A_n(x) y^{(n)}+A_{n-1}(x)y^{(n-1)} +\cdots +A_0(x) y=0,
	\label{eq:diff-eq}
\end{equation}
whose coefficients lie in $\mathbb{C}_p\llbracket x \rrbracket$ with non-zero radius of convergence. 
We may write 
\[A_i(x)=x^i \sum_{j=0}^{\infty} a_{ij} x^j, \qquad 0\le i\le n,\]
with $a_{i0}\neq 0$ for some $i$. 
We set for $j\ge 0$, 
\[\Phi_j(s)=\sum_{i=0}^n a_{ij} \binom{s}{i}i!.\] 
Then $\Phi_0$ is the indicial polynomial of \eqref{eq:diff-eq} at $x=0$. 

\begin{theorem}[\cite{Clark} Theorem 3, \cite{Set97}]
	If the negative of each zero of the indicial polynomial $\Phi_0$ of \eqref{eq:diff-eq} is a $p$-adic non-Liouville number \cite[Definition 1]{Clark}, then each power series solution $f(x)$ converges for 
	\[ \ord x > c,\]
	where $c$ is a constant determined by the weights of zeros of $\Phi_0$ \cite[Definition 2]{Clark} and the $p$-adic orders of coefficients of $\Phi_j$. 
	\label{t:Clark}
\end{theorem}

We will apply the above theorem to a differential equation whose coefficients $A_i(x)$ lie in $(\overline{\mathbb{Q}}\cap K)[x]$. 
In this case, every zero of $\Phi_0(s)$ is a $p$-adic non-Liouville number. Hence we deduce a positive $p$-adic convergent radius for a power series solution $f(x)$ of \eqref{eq:diff-eq}.

\subsection{Rank one differential modules over the Robba ring, d'après Pulita \cite{Pulita}}
\begin{secnumber}
Let $P(X)$ be a Lubin-Tate series:
\[
P(X)=pX+\cdots+X^p+\cdots \in \mathbb{Z}_p[ [X]],
\]
such that $P(X)\equiv X^p \mod p$ and $P(X)\equiv p X \mod X^2$. 
We fix a sequence $\{\pi_j\}_{j\ge 0}$ of elements in $\overline{K}$ such that 
\[
P(\pi_{j+1})=\pi_j,~ P(\pi_0)=0,\quad \pi_0\neq 0.
\]
In view of the Newton polygon of $P$, we have 
\[|\pi_0|=p^{-\frac{1}{p-1}}:=\omega, \quad |\pi_i|=\omega^{\frac{1}{p^i}}.\]
	The sequence $(\pi_j)_{j\ge 0}$ forms a topological generator of the Tate module of the formal group defined by $P$. 
	For example if $P(X)=(X+1)^p-1$, we take $\pi_i=\zeta_{p^i}-1$ with a compatible system of primitive $p^i$-th roots of unity $\zeta_{p^i}$. 
	If $P(X)=X^p+pX$, we take $\pi_0$ to be Dwork's $\pi$ (\ref{sss:def-isod}). 
\end{secnumber}

\begin{secnumber}
	Let $B$ be a $\mathbb{Z}_{(p)}$-algebra. Let $\rW(B)$ be the ring of Witt vectors. For $n\ge 0$, we set $\rW_n(B)=\rW(B)/V^{n+1}\rW(B)$. 
We denote the $n$-th ghost component by 
\[
\omega_n:\rW_n(B)\to B, \quad \lambda=(\lambda_0,\lambda_1,\cdots,\lambda_n)\mapsto \lambda_0^{p^n}+p\lambda_1^{p^{n-1}}+\cdots+p^n \lambda_n. 
\]

	For each $m\ge 0$, there exists a unique element $[\pi_m]\in \rW(\mathbb{Z}_p(\pi_m))$ whose ghost coordinates are: \cite[Remark 2.3]{Pulita}
	\[
	\langle\omega_0([\pi_m]),\dots,\omega_m([\pi_m]),\dots\rangle=\langle\pi_m,\pi_{m-1},\cdots,\pi_0,0,\cdots\rangle.
	\]

	We denote the Artin-Hasse exponential by 
\[
E(T):=\exp\biggl( T+ \frac{T^p}{p}+\frac{T^{p^2}}{p^2}+\cdots\biggr) \in 1+T\mathbb{Z}_{(p)}[ [T]].
\]
For $\lambda=(\lambda_0,\lambda_1,\cdots)\in \rW(B)$, the Artin-Hasse exponential relative to $\lambda$ is 
\[
E(\lambda,T):=\prod_{j\ge 0} E(\lambda_j T^{p^j})=\exp\biggl( \omega_0T+\omega_1\frac{T^p}{p}+\cdots\biggr) \in 1+TB[ [T]],
\]
where $\langle\omega_0,\omega_1,\cdots\rangle$ is the ghost component of $\lambda$. 
\end{secnumber}
\begin{prop} [\cite{Pulita} Proposition 2.2]
	\textnormal{(i)} For $m\ge 0$, the following formal series converges exactly in the disk $|T|<1$:
	\[
	E([\pi_m],T)=\exp\biggl(\pi_m T+\pi_{m-1}\frac{T^p}{p}+\cdots+ \pi_0\frac{T^{p^m}}{p^m}\biggr).
	\]

	\textnormal{(ii)} 
	Let $d=np^m$ be a positive integer with $(n,p)=1$ and $\lambda=(\lambda_0,\cdots,\lambda_m)\in \rW_m(\mathbb{Z}_p[\pi_m])$. 
	The following series
	\[
	e_d(\lambda,T):=E([\pi_m]\lambda,T^n)=\exp\biggl(\pi_m\omega_0T^{n}+\cdots+\pi_0\omega_m\frac{T^{d}}{p^m}\biggr)
	\]
	converges for $|T|<1$, where $\omega_i=\omega_i(\lambda)$ is the $i$-th ghost component.
\end{prop}

\begin{definition}
	Let $d=np^m$ be a positive integer as above and $K_m=K(\pi_m)$. 
	For $\lambda=(\lambda_0,\cdots,\lambda_m)\in \rW_m(\mathcal{O}_K)$, we denote by 	
\[
L_d(\lambda):= d+ n\cdot \biggl(\frac{\pi_0\omega_m}{T^d}+\cdots+\frac{\pi_m\omega_0}{T^n}\biggr)\frac{dT}{T},
\]
the differential module over $\mathcal{R}_{K_m}$ associated to $e_d(\lambda,T^{-1})$ (i.e. $e_d(\lambda,T^{-1})$ is a solution).
\end{definition}

\begin{theorem}[\cite{Pulita} Theorems 2.2, 2.5] \label{p:Ld-solvable-Frobenius}
	\textnormal{(i)} The differential module $L_d(\lambda)$ is solvable.

	\textnormal{(ii)} 
	If $|\lambda_i|<1$ for all $i=0,1,\cdots,m$, then $L_d(\lambda)$ is trivial. 

If $|\lambda_0|,\cdots,|\lambda_{r-1}|<1$ and $|\lambda_r|= 1$ for some $r\le m$, then we have $\Irr(L_d(\lambda))=d/p^r$. 

	\textnormal{(iii)}
	The following function is overconvergent (i.e. convergent in some disc $|T|<1+\varepsilon$ with $\varepsilon>0$):
\[
\theta_d(\lambda,T):=\frac{e_d(\lambda,T^{-1})}{e_d(\sigma(\lambda),T^{-p})}. 
\]
It induces a Frobenius structure on $L_d(\lambda)$, with respect to a lift of Frobenius on $K_m$ extending $\sigma$.  
\end{theorem}
\begin{prop} \label{p:rank-one-exp}
	Let $d$ be an integer $\ge 1$. 
	Let 
	\[
	L:= d+\biggl( \frac{a_d}{T^d}+\frac{a_{d-1}}{T^{d-1}}+\cdots+ \frac{a_1}{T} \biggr)\frac{dT}{T}
	\]
	be a solvable differential module of rank one over $\mathcal{R}_{\overline{K}}$ with $a_d\neq 0$. 
	
	\textnormal{(i)} $L$ can be equipped with a Frobenius structure. 

	\textnormal{(ii)} We have $|a_d|\le \omega$ and $|a_i|<1$ for $1\le i\le d-1$. 

	\textnormal{(iii)} The equality $|a_d|=\omega$ holds if and only if $\Irr(L)=d$. 
\end{prop}
\begin{proof}
	(i) 
	After enlarging $K$, we may assume that each $a_i$ is contained in $K$. 
	Since the coefficient of $L$ on $\frac{dT}{T}$ is zero, the assertion follows from \cite[Theorem 4.9]{Pulita2}. 

	(ii) By a direct calculation, one can show that the generic radius \cite[Lemma 3.4]{Pulita2} 
	\begin{equation} \label{eq:R(L,rho)}
		R(L,\rho)=\omega|a_d|^{-1}\rho^{d+1}, \quad \textnormal{for $\rho$ sufficiently close to $0$}. 
\end{equation}
	Using the convexity of $F(L,r)=-\log R(L,e^{-r})$, we deduce that $|a_d|\le \omega$ \cite[Lemma 4.2]{Pulita2}. 

	We prove the assertion by induction on $d$. The case $d=1$ is already known. 
	Suppose the assertion is proved for integers $\le d-1$ and we will prove this for $d$. 
	
	Suppose $d=np^m$ with $(n,p)=1$. 
	Since $|a_{d}|\le \omega$, we may find an element $(\lambda_0,\cdots,\lambda_m)\in \rW_m(\mathcal{O}_{\overline{K}})$ such that $n\omega_m(\lambda) \pi_0=a_d$. 
	Since solvable modules are closed under tensor products, the differential module $L\otimes (L_d(\lambda))^{-1}$ is solvable and its coefficients lie in $\mathfrak{m}_{\overline{K}}$ by \Cref{p:Ld-solvable-Frobenius} and induction hypothesis. 
	Then we conclude the assertion. 

	(iii) If $\Irr(L)=d$, then $R(L,\rho)=\rho^{d+1}$ when $\rho$ is sufficiently close to $1$. 
	By the convexity of $F(L,r)$ and \eqref{eq:R(L,rho)}, we deduce that $R(L,\rho)=\rho^{d+1}$ for all $\rho\in (0,1)$ and therefore $|a_d|=\omega$. 

	If $|a_d|=\omega$, by the convexity of $F(L,r)$ and \eqref{eq:R(L,rho)}, we deduce that $R(L,\rho)=\rho^{d+1}$ for all $\rho\in (0,1)$. 
	Hence we have $\Irr(L)=d$. 
\end{proof}

\section{$P$-adic isoclinic connections} \label{s:p-isoclinic}

In this section, let $\cG$ be a connected split almost simple group of adjoint type over $K$, equipped with a split model over $\mathcal{O}_K$ and $\cg$ its Lie algebra. 
We fix a maximal torus $\cT$ of $\cG$ and a Borel subgroup $\cT\subset \cB \subset \cG$. We assume that $p$ is good for $\cG$.  

	The order of a regular elliptic element of the Weyl group $\bW$ of $\cG$ is called \textit{regular elliptic number} of $\cG$ \cite{Springer}. 
	Let $m$ be a regular elliptic number of $\cG$ which is prime to $p$, and $N$ a positive integer prime to $m$. We set $\nu=\frac{N}{m}$.

\subsection{$P$-adic isoclinic connection over the open unit disc}
\label{ss:p-adic-isoclinic}

\begin{secnumber} \label{sss:formal-isoclinic}
	Let $\check{\rho}$ be the half sum of positive coroots of $\cfg$. 
	We fix a primitive $m$-th root of unity $\zeta_m$. 
	Let $\theta=\Ad_{\check{\rho}(\zeta_m)}$ be an inner automorphism of $\cg$ that defines a grading:
	\begin{equation}
		\cg=\bigoplus_{i\in \mathbb{Z}/m\mathbb{Z}} \cg_i.
		\label{eq:principal-grading-m}
	\end{equation}

	Following \cite{JY23}, \textit{a formal isoclinic connection} is a $\cg$-valued connection:
	\[
	d+(A_rt^r+A_{r+1}t^{r+1}+\cdots)dt,\quad A_i\in \cg(\overline{K}), \quad r\in \mathbb{Z},
	\]
	that can be gauge transformed over $\cG(\overline{K(\!(t)\!)})$ into the following canonical form \cite{BV83}: 
	\begin{equation}
		d+(X_{-N}t^{-N/m}+\cdots+X_{-1}t^{-1/m})\frac{dt}{t},
		\label{eq:isoclinic-canonical}
	\end{equation}
	where $X:=X_{-N}\in \cg(\overline{K})$ is regular semisimple, $(N,m)=1$, $X_i\in \fz_{\cg}(X)=:\ct_{X}$ and $X_i\in \cg_i(\overline{K})$ for $-N+1\le i\le -1$. 

	Moreover, its differential Galois group is generated by the smallest torus in $\cG$ containing $X$ and a lift of a regular elliptic element of order $m$ in the Weyl group $\bW$ of $\cG$. 
	
	To work with the $p$-adic coefficients, we also need the following observation.
\end{secnumber}

	\begin{lemma}\label{l:isoclinic gauge}
		A formal isoclinic connection with slope $\frac{N}{m}$ can be gauge transformed over $\cG(\overline{K}(\!(t^{1/2m})\!))$
		into its canonical form.
	\end{lemma}
    \begin{proof}
    	Let $\nabla$ be a formal isoclinic connection with slope $\frac{N}{m}$
    	whose canonical form has a regular semisimple leading term $X\in\cg$.
    	Let $\{p_{-1},2\check{\rho},p_1\}$ be a principal $\mathfrak{sl}_2$-triple
    	of $\cg$ where $p_{-1}$ is a sum of basis of negative simple root subspaces. 
    	Let $r$ be the rank of $G$, $\{d_i\}_{i=1}^r$ the fundamental degrees of $G$ and $\{p_i\}_{i=1}^r$ the homogeneous basis of $\cfg^{p_1}$ with weight $d_i$. 
    	Denote $\ell_{i,j}=(d_i-1)N+m(d_i-1-j)$ for $1\leq i\leq r$ and any $j$.
    	By \cite[Proposition 23]{Yi25}, 
    	$\nabla$ can be gauge transformed by $\cG(\overline{K}(\!(t)\!))$
    	into an equation as follows:
    	\[
    	d+(p_{-1}+\sum_{i=1}^r\sum_{\ell_{i,j}\geq -N}v_{i,j}t^{-j-1}p_i)dt,
    	\]
    	where $v_{i,j}\in\mathbb{C}$ and 
    	$Y:=p_{-1}+\sum_{m\mid d_i}v_{i,d_i+\frac{d_iN}{m}-1}p_i\in\cG\cdot X$.
    	
    	Fix $s=t^{1/2m}$.
    	Applying gauge transform by $s^{2(N+m)\check{\rho}}$, 
    	we obtain $\nabla'$ as in the equation (44) in the \emph{loc. cit.},
    	which is a formal connection with most polar term $2mYs^{-2N-1}ds$.
    	Since $2mY$ is regular semisimple,
    	the computation in the proof of Lemma 20 of the \emph{loc. cit.}
    	shows that $\nabla'$ can be gauge transformed over $\cG(\overline{K}(\!(t^{1/2m})\!))$ into its canonical form.
    \end{proof}

    \begin{rem}\label{r:gauge by m}
    	The argument in the above actually shows that a formal isoclinic connection
    	with slope $\frac{N}{m}$ can be gauge transformed over $\cG(\overline{K}(\!(t^{1/m})\!))$ to its canonical form
    	as long as $\check{\rho}$ is a cocharacter of $\cG$.
    	This is the case for example when $\cG$ is adjoint.
    	Moreover, it can be verified that for those $\cG$ where $m$ can be odd,
    	$\check{\rho}$ is always a cocharacter.
    \end{rem}

In the following, we consider a $p$-adic analogue of the above notion. 

\begin{definition} \label{d:p-adic-isoclinic}
	A \textit{$p$-adic isoclinic connection} with ($p$-adic) slope $\nu=\frac{N}{m}$ is a solvable $\cg$-valued differential module $\mathscr{E}$ over $\mathcal{A}_t$ (\ref{sss:Robba-ring}, \ref{sss:slopes}) such that $[2m]^*\mathscr{E}$ can be $\cG(\mathcal{A}_{t^{1/2m}})$-gauge equivalent to a form as follows:
	\begin{equation}
		d+\biggl(\pi X_{-N}t^{-N/m}+\cdots + X_{-1}t^{-1/m}\biggr)\frac{dt}{t},
		\label{eq:p-adic-canonical-form}
	\end{equation}
	with $X:=X_{-N}\in \cg(\mathcal{O}_{\overline{K}})$, such that $X$ is both regular semisimple in $\cg(\overline{K})$ and in $\cg(\overline{k})$, $X_i\in \ct_{X}$ and $X_i\in \cg_i(\mathfrak{m}_{\overline{K}})$ for $-N+1\le i \le -1$. 
\end{definition}


In view of Remark \ref{r:gauge by m},
when $m$ is odd, we can take a gauge transform by $\cG(\mathcal{A}_{t^{1/m}})$.
Thus $p$ is always prime to the degree $2m$ (or $m$) of the covering we need to take.
Henceforth, we only work in the case where $p>2$.
When $p=2$, we can redo the following argument by replacing $t^{1/2m}$ by $t^{1/m}$.

By assumption, an eigenvalue of the action of $X$ on $\cg$ is either zero or a $p$-adic unit. 
Via the adjoint representation $\cG\to \GL(\cg)$, we obtain a decomposition of $[2m]^*\mathscr{E}(\cg)$ in the category of differential modules over $\mathcal{A}_{t^{1/2m}}$:
\begin{equation}
	[2m]^*\mathscr{E}(\cg)\simeq Y_0\bigoplus \oplus_{\alpha\in \Phi} Y_\alpha,
	\label{eq:decomposition-Y}
\end{equation}
where $Y_0$ is trivial of rank $r=\rank(\cG)$. 
Denote $s=t^{1/2m}$. 
Each $Y_\alpha$ is a rank one solvable differential module over $\mathcal{A}_s$:
\[
Y_{\alpha}= d+\biggl(\varepsilon_{N,\alpha}s^{-2N}+\cdots+\varepsilon_{1,\alpha}s^{-2} \biggr)\frac{ds}{s},
\]
with $|\varepsilon_{N,\alpha}|=\omega$ and $|\varepsilon_{i,\alpha}|<1$. 

\begin{prop} \label{p:slopes}
	Let $\mathscr{E}$ be a $p$-adic isoclinic connection. 

	\textnormal{(i)} $\mathscr{E}$ can be equipped with a $\cG$-valued Frobenius structure over $\mathcal{R}$. 

	\textnormal{(ii)} The maximal $p$-adic slope of $\mathscr{E}$ is $\nu$. Moreover, the multi-set of $p$-adic slopes of $\mathscr{E}(\cg)$, defined by the adjoint representation, consists of $\sharp \Phi$ copies of $\nu$ and $r$ copies of $0$. 
\end{prop}

\begin{proof}
	(i) After enlarging $K$ (preserved by $\sigma$), we may assume that $K$ contains all $n$-th roots of unity and that coefficients $X_i$ of $\mathscr{E}$ in \eqref{eq:p-adic-canonical-form} lie in $\cg(K)$. 
	Consider $\mathscr{E}(\cg)$ the $\GL(\cg)$-valued connection via the adjoint representation $\iota:\cG\to \GL(\cg)$. 
	By \Cref{p:rank-one-exp}(i), $[2m]^*\mathscr{E}(\cg)$ is equipped with a Frobenius structure over the Robba ring $\mathcal{R}$:
	\[
	\varphi(s)=\exp(\iota(\biggl(\frac{m\pi}{N} X_{-N}s^{-2N}+\cdots + mX_{-1}s^{-2}\biggr)))\cdot
	\exp(\iota(\biggl(\frac{m\pi}{N} \sigma(X_{-N})s^{-2qN}+\cdots + m\sigma(X_{-1})s^{-2q}\biggr)))^{-1}, 
	\]
	where $\sigma: \mathcal{O}_K\to \mathcal{O}_K$ is a lift of the $q$-th Frobenius. 
	By construction, $\varphi(s)$ lies in $\cG(\mathcal{R}_K)\subset \GL(\cg)(\mathcal{R}_K)$ and defines a $\cG$-valued Frobenius structure on $[2m]^*\mathscr{E}$. 
	When we replace $s$ by $\zeta_{2m} s$ for a $m$-th root of unity $\zeta_{2m}$, $\varphi(s)$ is conjugate to $\varphi(\zeta_{2m} s)$ by $\iota(\check{\rho}(\zeta_{2m}))$. 
	Hence, $\varphi$ descends to a $\cG$-valued Frobenius structure on $\mathscr{E}$ and the assertion follows. 

	(ii) By \Cref{p:rank-one-exp}(iii), each $Y_{\alpha}$ in decomposition \eqref{eq:decomposition-Y} has the $p$-adic slope $2N$. 
	Then the assertion follows. 
\end{proof}

\subsection{Comparison of local monodromy groups}
	Let $\mathscr{E}$ be a $p$-adic isoclinic connection. 
	Let $\mathcal{E}$ be the formal connection associated to $\mathscr{E}$ via the canonical map $\mathcal{A}_t\to \overline{K}(\!( t)\!)$ \eqref{eq:inclusion-rings}. 
	Then $\mathcal{E}$ is a formal isoclinic connection in the sense of \ref{sss:formal-isoclinic}. 

	Let $\langle \mathcal{E}\rangle$ be the Tannakian subcategory of the category $\MC(\overline{K}(\!(t)\!))$ of formal connections over $\overline{K}(\!(t)\!)$ relative to $\overline{K}$ generated by $\mathcal{E}$. 
	Let $\langle \mathscr{E}\rangle$ be the Tannakian subcategory of $\MC(\mathcal{R})$ generated by $\mathscr{E}$. 
We denote by $G_{\diff}(\mathcal{E})$ (resp. $G_{\geo}(\mathscr{E})$) the Tannakian group of $\langle \mathcal{E}\rangle$  (resp. $\langle \mathscr{E}\rangle$). 
	We establish a relationship between two Tannakian categories. 

\begin{prop} \label{p:compare-monodromy-group}
	There exists a natural tensor functor between the Tannakian categories:
	\begin{equation} \label{eq:dagger-functor-infty}
	\langle \mathcal{E}\rangle \to \langle \mathscr{E}\rangle,
	\end{equation}
	which induces an embedding between local monodromy groups:
	\[
	G_{\geo}(\mathscr{E})\to G_{\diff}(\mathcal{E}).
	\]
\end{prop}
\begin{proof}
	By construction, the category $\langle [2m]^* \mathscr{E}\rangle$ (resp. $\langle [2m]^* \mathcal{E}\rangle$) is semisimple and every irreducible object in this category is of rank one and has the form $d+(\sum_{j\ge 1} a_js^{-j})\frac{ds}{s}$. 
	For two rank one objects $L_i=d+(\sum_{j\ge 1} a_{i,j}s^{-j})\frac{ds}{s}$ of $\MC(\overline{K}(\!(s)\!))$, $L_1\simeq L_2$ if and only if $a_{1,j}=a_{2,j}$ for every $j$.  
	Hence, there exists a natural tensor functor
	\[
	\langle [2m]^*\mathcal{E}\rangle \to \langle [2m]^*\mathscr{E}\rangle. 
	\]
	
	In general, the category $\langle \mathcal{E}\rangle$ (resp. $\langle \mathscr{E}\rangle$) is equivalent to the category of descent data along the pullback $[2m]^*$. 
	Then we can define the functor \eqref{eq:dagger-functor-infty} by applying the above functor to the descent data. 
	In particular, $\mathcal{E}(V)$ is sent to $\mathscr{E}(V)$ via this functor for every representation $V$ of $\cG$. 

	By \cite[Proposition 2.21]{DM82}, the induced morphism $G_{\geo}(\mathscr{E})\to G_{\diff}(\mathcal{E})$ is a closed immersion. 
\end{proof}
\begin{secnumber}
	Let $F=k(\!(t)\!)$ be the local field, $\Gal_F$ the Galois group of $F$ and $I_F$ the inertia subgroup of $\Gal_F$. 
	By \Cref{p:slopes}(i) and the $p$-adic local monodromy theorem (see \ref{sss:MCF}), one can associate to $\mathscr{E}$ a representation $(\rho,N)$ with coefficients in $\overline{K}$ 
\[
\rho: I_F \to \cG(\overline{K}),\quad N\in \cfg(\overline{K})~ \textnormal{nilpotent},
\]
such that $\rho,N$ commute with each other and that $\rho(I_F)$ is finite.  
\end{secnumber}

\begin{theorem}
\label{t:L-parameter}
	\textnormal{(i)} 
	Let $\cT_X$ be the centralizer $C_{\cG}(X)$ of $X$ in $\cG$ over $\overline{K}$, which is a maximal torus. 
	The representation $\rho$ fits into a commutative diagram
	\[
	\xymatrix{
	1\ar[r] & I^+_F \ar[r] \ar[d] & I_F \ar[r] \ar[d] & I_F^t \ar[r] \ar[d] & 1 \\
	1\ar[r] & \cT_X \ar[r]& N_{\cG}(\cT_X) \ar[r] & \bW \ar[r] & 1
	}
	\]
	such that the wild inertia subgroup $I^+_F$ is sent to $\cT_X$ and a generator of the tame inertia $I_F^t$ is sent to a regular elliptic element $\sigma$ of the Weyl group $\bW$ of order $m$. 

	\textnormal{(ii)} The nilpotent operator $N$ is trivial. 

	\textnormal{(iii)} 
	The centralizer of $\rho(I_F^+)$ (resp. $\rho(I_F)$) in $\cG$ is the torus $\cT_X$ (resp. $\cT_X^{\sigma}$, which is finite). 

	\textnormal{(iv)} We have $\cfg^{I_F}=0$ and $\Swan(\cfg)=\sharp \Phi\cdot \nu$.
\end{theorem}


\begin{proof}
	(i) By a similar argument of \cite[Proposition 12]{CY24}, the wild inertia group of $G_{\diff}(\mathcal{E})$ is the smallest torus $S$ whose Lie algebra contains $X$, and a generator of its tame quotient is sent to a regular elliptic element of the Weyl group $\bW$ of order $m$. 
	Hence, the local differential Galois group $G_{\diff}(\mathcal{E})$ fits into the following diagram:
	\[
	\xymatrix{
	1\ar[r] & G_{\diff}^{\wild}(\mathcal{E}) \ar[r] \ar[d] & G_{\diff}(\mathcal{E}) \ar[r] \ar[d] & \mathbb{Z}/m\mathbb{Z} \ar[r] \ar[d] & 1 \\
	1\ar[r] & \cT_X \ar[r]& N_{\cG}(\cT_X) \ar[r] & \bW \ar[r] & 1
	}
	\]
	Then the assertion follows from \Cref{p:compare-monodromy-group}. 

	(ii) In view of decomposition \eqref{eq:decomposition-Y}, the monodromy operator $N$ is trivial. 

	(iii) Let $\cT_X[p]$ be the $p$-torsion subgroup of $\cT_X$. 
	Let $\zeta_p$ be a $p$-th root of unity in $\overline{K}$. 
	If we set $\bX=\Hom(\Gm,\cT_X)$, we have an isomorphism
	\[
	\mathbb{F}_p\otimes_{\mathbb{Z}}\bX \xrightarrow{\sim} \cT_X[p],\quad a\otimes \lambda \mapsto \lambda(\zeta_p^a).
	\]

	Since $\rho(I_F^+)$ is a finite abelian $p$-group, 
	$\rho$ factors through a quotient $I_F^+ \twoheadrightarrow k'^+ \to \cT_X[p]$ for a finite extension $k'$ of $k$. 
	Via the isomorphism $k' \otimes_{\mathbb{Z}} \bX \xrightarrow{\sim} \Hom_{\mathbb{F}_p}(k',\cT_X[p])$, $\rho$ corresponds to an element $\check{\lambda}\in k'\otimes_{\mathbb{Z}} \bX$. 
	In view of the action of $\rho(I_F^+)$ on $\cg$ \eqref{eq:decomposition-Y}, we deduce that for any root $\alpha:\cT_X\to \Gm \in \Phi$, we have $\alpha\circ \check{\lambda}\neq 0$. 
	On the other hand, $\rho(I_F^+)$ is contained in the maximal torus $\cT_X$. 
	Since $\rho(I_F^+)$ is an elementary abelian $p$-group, the centralizer of $\rho(I_F^+)$ is connected by \cite[Theorem 2.28(c)]{Steinberg} and is therefore equal to $\cT_X$. 

	Since the image of a tame generator $\sigma=\rho(1)\in \bW$ is elliptic, we deduce that $C_{\cG}(\rho(I_F))=\cT_X^\sigma$ is finite. 

(iv) By (iii), we have $\cfg_X^{I_F^+}=\cft_X$. Since $\sigma=\rho(1)\in \bW$ is elliptic, we deduce that $\cfg_X^{I_F}=\cft_X^\sigma=0$. 
	The $p$-adic irregularity can be calculated by \Cref{p:slopes}(ii).  
\end{proof}

\begin{coro} \label{c:unique-Frobenius}
	The Frobenius structure on $\mathscr{E}$ is unique up to an element of $\cT_X^{\sigma}$ in $\cG(\overline{K})$. 
\end{coro}
\begin{proof}
	Given two Frobenius structures $\varphi_1,\varphi_2$, the composition $u=\varphi_2\circ \varphi_1^{-1}$ commutes with $\rho(I_F)$. 
	Then the corollary follows from \Cref{t:L-parameter}(iii). 
\end{proof}

\begin{secnumber}
	Assume that $k=\mathbb{F}_q$, $\sigma=\id_K:K\to K$ and that there is a Frobenius structure $\varphi$ on $\mathscr{E}$. 
	By \ref{sss:MCF}, we can associate to $(\mathscr{E},\varphi)$
	a Weil--Deligne representation $(\rho,N)$ with coefficients in $\overline{K}$ 
	\[
	\rho: W_F \to \cG(\overline{K}), \quad N=0.
	\]
\end{secnumber}

\begin{coro} \label{c:WD-rep}
	\textnormal{(i)} The Weil representation $\rho:W_F\to \cG(\overline{K})$ is inertially discrete, i.e. the centralizer of $\rho(I_F)$ in $\cG$ is finite. 
	
	\textnormal{(ii)} The image $\rho(W_F)=D$ is finite and is contained in the normalizer $N_{\cG}(\cT_X)$. 
	It has ramification filtration of the form:
	\[
	D \ge D_0=\rho(I_F) \ge D_1=\rho(I_F^+), \quad D_2=1. 
	\]
	The only non-zero upper ramification break is $\nu=\frac{N}{m}$. 
\end{coro}
\begin{proof}
Assertion (i) follows from \Cref{t:L-parameter}(iii). 
	
(ii) 
By \cite[Lemma 3.1]{GR10}, we deduce that $\rho(W_F)$ is finite. 
The image of a lift of the Frobenius $\rho(\Fr)$ normalizes the image $\rho(I_F^+)$ and therefore normalizes the centralizer $C_{\cG}(\rho(I_F^+))=\cT_X$. 
Hence $\rho(W_F)$ is also contained in $N_{\cG}(\cT_X)$. 

It follows from \Cref{p:slopes}(ii) that the only non-zero ramification break is $\nu$. 
\end{proof}

\begin{rem}
	From the above discussion,
	we can see the Weil--Deligne representation associated to
	a $p$-adic isoclinic connection of slope $\nu$
	is a \emph{toral supercuspidal parameter of generic depth $\nu$}
	\cite[Definition 6.1.1]{Kal19}.
	In the \emph{loc. cit.}, the author constructed 
	supercuspidal $L$-packets for such parameters. 
	We will later discuss examples of such correspondence
	via global geometric Langlands correspondence.
	However, it is unclear to us whether $p$-adic isoclinic connections 
	give all the toral supercuspidal parameters.
\end{rem}

\section{Frobenius structures on certain rigid connections}
\label{s:Frobenius}
	In this section, let $G$ be a split connected almost simple group over a base $k$ (or $\mathcal{O}_K$, $K$), that we will specify. 
	We fix a Borel subgroup $B\subset G$ and a maximal torus $T\subset B$. 
	Let $U\subset B$ be the unipotent radical of $B$, and $U^{\op}\subset B^{\op}$ the opposite Borel and its unipotent radical. Let $T_{\ad}\subset B_{\ad}\subset G_{\ad}$ denote the quotients of $T\subset B\subset G$ by the center $Z(G)$ of $G$.
	We denote by $(\cG,\cB,\cT)$ the Langlands dual group of $G$ over $\overline{K}$, constructed by the geometric Satake equivalence.

\subsection{Generalized Kloosterman sheaves, d'après Yun \cite{Yun16}}

	\begin{secnumber} \label{sss:parahoric}
	We work over the base $k$. 
	Let $\bP$ be a parahoric subgroup of $G(k (\!(t)\!))$. 
	Let $\bP(1), \bP(2)$ be the first and second steps in the Moy--Prasad filtration of $\bP$, and let $L_\bP\subset \bP$ be the natural lift of the Levi quotient $L_{\bP}\simeq \bP/\bP(1)$. 
	Set $V_\bP=\bP(1)/\bP(2)$, which is equipped with an action of $L_{\bP}$. 
	
	The parahoric subgroup $\bP$ is called \textit{admissible} if there exists a closed $L_\bP$-orbit on the dual space $V_{\bP}^*$ with finite stabilizers \cite[Definition 2.5]{Yun16}. 
	Such an orbit is called \textit{stable}. 
	The set of admissible parahoric subgroups $\bP$ is bijective with the set of regular elliptic numbers $m=m(\bP)$ of $G$ \cite[\S~2.6]{Yun16}.

	Let $m$ be the regular elliptic number of $G$ associated to $\bP$ (which is also a regular elliptic number of $\cG$). 
	We assume that $p$ is prime to $m$ and that $p$ is not a torsion prime of $G$. 
	In the case where $m=h$ is the Coxeter number, we don't require any assumption on $p$. 
	
	Let $\mathcal{G}$ be the group scheme on $\P1$ such that $\mathcal{G}|_{D_0}\simeq \bP^{\op}$ the parahoric subgroup opposite to $\bP$, $\mathcal{G}|_{D_{\infty}}\simeq \bP(2)$ and $\mathcal{G}|_{\Gm}\simeq G\times \Gm$, where $D_x$ is the formal disc around $x\in\{0,\infty\}$. 
	Let $\Bun_{\mathcal{G}}$ be the moduli stack of $\mathcal{G}$-torsors on $\P1$. 
	There exists an open immersion 
	\begin{equation}
		j: U(\simeq V_{\bP}) \hookrightarrow \Bun_{\mathcal{G}},
		\label{eq:immersion-j}
	\end{equation}
	such that $U$ is the non-vanishing locus of a section of a line bundle on $\Bun_{\mathcal{G}}$ \cite[Lemma 3.1]{Yun14}. 
	In particular, $j$ is open and affine. 
\end{secnumber}

\begin{secnumber}\label{sss:Kl^rig}
	We take a non-trivial additive character $\psi:\mathbb{F}_{p}\to K^{\times}$ and denote by $\pi\in K$ the element satisfying $\pi^{p-1}=-p$ associated to $\psi$. Let $\mathscr{A}_{\psi}$ be the Dwork $F$-isocrystal on $\mathbb{A}^1_k$ (\ref{sss:def-isod}). 
	In the following, we use the theory of arithmetic $\mathscr{D}$-modules \cite{Ber02,AC18} (see \cite[\S~2.1]{XZ22} for a brief summary). 

	We fix a stable function $\phi:V_\bP\to \A1$ over $k$. 
	The canonical morphism 
	\[
	j_!(\phi^+ \mathscr{A}_\psi)\to j_+(\phi^+ \mathscr{A}_\psi)
	\]
	is an isomorphism. Then we set $\mathcal{A}_\phi=j_{!+}(\phi^+ \mathscr{A}_\psi)$, which is a holonomic module on $\Bun_{\mathcal{G}}$. 
	It is irreducible and is $(\bP(1),\phi^+\mathscr{A}_\psi)$-equivariant with respect to the action of $\bP(1)$ on $\Bun_{\mathcal{G}}$ at $\infty$. 
	
	By \cite[Theorem 3.8]{Yun16}, $\mathcal{A}_\phi$ is a Hecke eigen-module. 
	We denote the Hecke eigenvalue of $\mathcal{A}_\phi$ by
	\[
	\Kl_{\cG}^{\rig}(\phi):\Rep(\cG) \to \FIsod(\mathbb{G}_{m,k}/K).
	\]
\end{secnumber}

\begin{secnumber} \label{sss:Kl-dR}
	Let $\ell$ be a prime different from $p$. 
	There is a variant of Yun's construction using algebraic $\mathscr{D}$-modules (resp. $\ell$-adic sheaves) instead of arithmetic $\mathscr{D}$-modules to produce a $\cG$-connection on $\mathbb{G}_{m,K}$ (resp. an $\ell$-adic $\cG$-local system on $\mathbb{G}_{m,k}$), 
	as all the geometric objects used in the construction have analogues over $\OK$. Namely, 
	we replace the overconvergent $F$-isocrystal $\mathscr{A}_{\psi}$ on $\mathbb{A}^1_k$ by the \textit{exponential $\mathscr{D}$-module}
	$\EE_{\lambda}=K\langle x, \partial_x \rangle/(\partial_x+\lambda)$ with a parameter $\lambda\in K$ on $\mathbb{A}^{1}_{K}$ (resp. the Artin--Schreier sheaf $\AS_{\psi}$ for a fixed additive character $\psi:\mathbb{F}_p\to \overline{\mathbb{Q}}_{\ell}^{\times}$). 
	We fix a stable function $\phi:V_{\bP}\to \mathbb{A}^1$ over $K$ (resp. over $k$). 
	By applying the above construction, we obtain a de Rham (resp. $\ell$-adic) Hecke eigenvalue: 
	\begin{displaymath} 
		\Kl_{\cG}^{\dR}(\lambda\phi): \Rep(\cG)\to \Conn(\mathbb{G}_{m,K}) \qquad \textnormal{(resp. } 
		\Kl_{\cG}^{\ell}(\phi): \Rep(\cG)\to \LocS(\mathbb{G}_{m,k})). 
	\end{displaymath}
\end{secnumber}

\subsection{Comparison between $\Kl_{\cG}^{\dR}$ and $\Kl_{\cG}^{\rig}$} \label{ss:compare-dR-rig}
In this subsection, we work with schemes over $\mathcal{O}_K$. 
We say a linear function $\phi:V_{\bP}\to \A1$ over $\mathcal{O}_K$ is \textit{stable}, if its base changes to $k$ and to $K$ are both stable. 
	We take such a function $\phi$ and we denote abusively its base change to $k$ (resp. $K$) by $\phi$. 

	\begin{secnumber} \label{sss:dagger-connection}
	Let $U \hookrightarrow X=\mathbb{P}^1$ be an open immersion of schemes over $\mathcal{O}_K$ with the complement $S\to X$ flat over $\mathcal{O}_K$. 
	Let $\XX$ (resp. $\UU$) be the $p$-adic completion of $X$ (resp $U$), $\XX^{\rig}$ (resp. $\UU^{\rig}$) the associated rigid generic fiber and $\mathcal{X}$ (resp. $\mathcal{U}$) the analytification $X_K^{\an}$ (resp. $U_K^{\an}$). 
	Then the two rigid spaces $\XX^{\rig}$ and $\mathcal{X}$ are isomorphic, and $\UU^{\rig}$ is the tube $]U_k[_{\XX}$ of $U_{k}$ in $\overline{\mathcal{X}}$. 
	In particular, $\mathcal{U}$ is a strict neighborhood of $\UU^{\rig}$ in $\XX^{\rig}$. 
	We denote by $\Conn(U_{K})$ the category of coherent $\mathscr{O}_{U_K}$-modules with an integrable connection. 

	Let $(M,\nabla)$ be a coherent $\mathscr{O}_{U_K}$-module endowed with an integrable connection (relative to $K$). 
	We denote by $(M^{\an},\nabla^{\an})$ its pullback to $\mathcal{U}$ along $\varepsilon:\mathcal{U}\to U_K$. 
	For any strict neighborhood $V$ of $\UU^{\rig}$ in $\XX^{\rig}$, we refer to (\cite{Ber96} 2.1.1) for the definition of functor $j^{\dagger}$ from the category $\Ab(V)$ of abelian sheaves on $V$ to itself. 
	We associate to $M^{\an}$ a $j^{\dagger}\mathscr{O}_{\XX^{\rig}}$-module $M^{\dagger}=j^{\dagger}(M^{\an})$, endowed with the corresponding connection. 
	In this setting, we have the following tensor functors: 
	\begin{equation} \label{eq:dagger-functor}	
		\Conn(U_K) \xrightarrow{(-)^{\dagger}} 
		\Conn(j^{\dagger}\mathscr{O}_{\XX^{\rig}})
		\hookleftarrow \Iso^{\dagger}(U_k/K) 
	\end{equation}
	where $\Iso^{\dagger}(U_k/K)$ denotes the category of overconvergent isocrystals on $U_k$ relative to $K$ and the second functor is fully faithful. 
	However, the functor $(-)^{\dagger}$ is not fully faithful. 

	We may extend the functors \eqref{eq:dagger-functor} over $\overline{K}$ by extension of scalars. 
\end{secnumber}

\begin{theorem} \label{t:compare-Kl-rig-dr}
	There exists a canonical isomorphism of $\cG$-valued $j^{\dagger}\mathscr{O}_{\XX^{\rig}}$-modules with connection
		\[
		\iota:(\Kl_{\cG}^{\dR}(-\pi \phi_K))^{\dagger}\xrightarrow{\sim} \Kl_{\cG}^{\rig}(\phi_k).
		\]
	\end{theorem}

	One can prove the above theorem using the same argument as in \cite[\S~4.2]{XZ22}, where the Iwahoric case $\bP=\bI$ is proved. 
	In the proof of \textit{loc. cit}, a crucial ingredient is that the open embedding $j:V_{\bI}\to \Bun_{\mathcal{G}}$ is  the non-vanishing locus of a section of a line bundle on $\Bun_{\mathcal{G}}$. 
	The same geometric property holds in our parahoric setting.

\subsection{Comparison between $\Kl_{\cG}^{\dR}$ and theta connection}

	In this subsection, we work with schemes over $K$. 

\begin{secnumber}\label{sss:theta-conn}
	Let $\theta$ be an inner stable grading of $\cfg$ over $K$ corresponding to the regular elliptic number $m$ via \cite[Corollary 14]{RLYG12}. 
	More precisely, we may assume $\theta$ is given by the adjoint action of $\exp(x)$, with $x\in \mathbb{X}_*(\cT)\otimes_{\mathbb{Z}}\mathbb{R}$, such that $\check{\lambda}=mx\in\mathbb{X}_*(\cT) $. 
	We may moreover assume that $x$ lies in the closure of the fundamental alcove. 
	We denote the $\bZ$-grading defined by $\check{\lambda}$ on $\cfg$ by:
\begin{equation}
	\cfg=\bigoplus_l \cfg(l).
	\label{eq:grading-cg}
\end{equation}
Let $\cfg_1\subset \cfg$ be the degree one subspace with respect to $\theta$. 
Then we have:
\begin{equation}
	\cfg_1=\bigoplus_{l\equiv 1 \mod m} \cfg(l).
	\label{eq:decomposition-cg1}
\end{equation}

Let $X\in \cfg_1$ be an element. We write $X=\sum_l X_l$, with $X_l\in \cfg(l)$. 
Let $\beta$ be the highest root, $\alpha_0$ the affine root $1-\beta$. 
The (normalized) Kac coordinate $s_0:=\alpha_0(x)$ equals $0$ or $1$. 
By \cite[Lemma 2.1]{Chen17}, those $l$ with $X_l\neq 0$ satisfy $s_0-m\le l\le m$ and $m| l-1$. 

Following Yun and Chen \cite{Chen17}, the $\theta$-connection associated to $X$ is the $\cG$-connection on $\Gm$, defined by
\begin{equation} \label{eq:theta-connection}
	\Theta(X):= d+ \biggl(\sum_l X_l x^{\frac{1-l}{m}}\biggr) \frac{dx}{x}= d+ \biggl( X_1+X_{1-m} x\biggr)\frac{dx}{x}.
\end{equation}

Let $\cG_0 \subset \cG$ be the reductive subgroup fixed by $\theta$. 
Let $\cfg^{\st}$ be the locus of stable vectors, i.e. those whose $\cG_0$-conjugacy class is closed with finite stabilizer. 
On the other hand, we denote by $V_\bP^{*,\st}$ the locus of stable vectors under the action of $L_{\bP}$ (\ref{sss:parahoric}). 
\end{secnumber}
\begin{theorem}[\cite{CY24} Lemma 3, Theorem 5] \label{t:compare-Kl-theta-connection}
	Assume that the Kac coordinates of the associated point $x$ satisfies $s_0=1$, and the same is true for the order $m$ stable inner grading of $\fg$. 
	
	\textnormal{(i)} The natural isomorphism $\fg^*/\!\!/G\simeq \cfg/\!\!/\cG$ induces an isomorphism 
	\begin{equation}
		V_{\bP}^*/\!\!/L_{\bP} \xrightarrow{\sim} \cfg_1/\!\!/ \cG_0,
		\label{eq:isomorphism-VP-cg1}
	\end{equation}
	that restricts to a bijection between stable conjugacy classes
	\begin{equation}
		V_{\bP}^{*,\st}/L_{\bP}\xrightarrow{\sim} \cfg_1^{\st}/\cG_0.
		\label{eq:isomorphism-VP-cg1-st}
	\end{equation}

	\textnormal{(ii)} 
	Let $\phi\in V_{\bP}^{*,\st}$ be a stable linear function, $X$ a representative of the stable vector orbit
	given by the image of $\phi$ under \eqref{eq:isomorphism-VP-cg1-st}. 
	Then there exists an isomorphism of $\cG$-connections:
	\begin{equation}
		\Kl^{\dR}_{\cG}(\phi) \xrightarrow{\sim} \Theta(X).
	\label{eq:comparison-dR}
	\end{equation}
\end{theorem}
\begin{proof}
	By descent, it suffices to prove the theorem after base change from $K$ to $\overline{K}$. 

	Assertion (i) is proved in \cite[Lemma 3]{CY24}. 

	(ii) We first assume $G$ is simply connected. 
	Since the geometric objects in the construction of $\Kl^{\dR}_{\cG}(\phi)$ are defined over $K$, it suffices to prove the assertion after base change from $K$ to $\overline{K}$ by descent. 
	Then we prove the assertion by applying the Galois-to-automorphic direction of geometric Langlands correspondence as in \cite[Theorem 5]{CY24}. 

	We can extend this result to allow $G$ to be a general almost simple group as in \cite[\S~4.3.6]{XZ22}. 
	By the same argument in \textit{loc.cit}, up to isomorphisms, there exists a unique de Rham $\cG$-local system on $\mathbb{G}_{m,K}$, which induces $\Kl^{\dR}_{\cG_{\ad}}$, and has unipotent monodromy at $0$. 
	Then we deduce that $\Kl_{\cG}^{\dR}(\phi)$ and $\Theta(X)$ are isomorphic. 
\end{proof}

\subsection{Frobenius structure on $\theta$-connections} \label{ss:Frob-theta-connection}
In this subsection, we construct the Frobenius structure on the theta connection, by putting the above ingredients together. 
We keep the notation of \ref{ss:compare-dR-rig}. 

\begin{secnumber}
	We take a non-trivial additive character $\psi:\mathbb{F}_p\to K^{\times}$ and a stable linear function $\phi:V_{\bP}\to \mathbb{A}^1$ over $\mathcal{O}_K$ as in \ref{ss:compare-dR-rig}.
	We set $\lambda=-\pi$ corresponding to $\psi$ (\ref{sss:def-isod}). 
	Let $X_{-\pi}\in \cfg_1$ (resp. $X\in \cfg_1$) match $-\pi\phi_K$ (resp. $\phi_K$) under the isomorphism \eqref{eq:isomorphism-VP-cg1}. 
	Suppose $\Theta(X)$ is given as in \eqref{eq:theta-connection}. 
	By considering the values of invariant polynomials on $X_{-\pi}$ and $-\pi\phi_K$, we deduce that:
	\begin{equation}
		\Kl_{\cG}^{\dR}(-\pi\phi_K)\simeq \Theta(X_{-\pi})\simeq d+\biggl(X_1+(-\pi)^mX_{1-m}x\biggr)\frac{dx}{x}. 
		\label{eq:theta-connection-Frob}
	\end{equation}

	Let $\Theta^{\dagger}(X_{-\pi})$ be the composition of $\Theta(X_{-\pi}):\Rep(\cG)\to \Conn(X_K)$ with the $(-)^{\dagger}$ functor \eqref{eq:dagger-functor}. 
	By \Cref{t:compare-Kl-rig-dr},  a choice of the above isomorphism endows $\Theta^{\dagger}(X_{-\pi})$ with a Frobenius structure $\varphi$, i.e. a lifting of $\Theta^{\dagger}(X_{-\pi})$ as a functor $\Rep(\cG)\to \Isod(\mathbb{G}_{m,k}/K)$ together with an isomorphism of tensor functors:
	\[
	\varphi:F_{X_k}^*\circ \Theta^{\dagger}(X_{-\pi})\xrightarrow{\sim} \Theta^{\dagger}(X_{-\pi}) : \Rep(\cG)\to \Isod(\mathbb{G}_{m,k}/K),
	\]
	 where $F^{*}_{X_k}:\Isod(X_{k}/K)\to \Isod(X_{k}/K)$ denotes the $s$-th Frobenius pullback functor. 
\end{secnumber}

Finally, we end with the following lemma. 

\begin{lemma} \label{l:eigenvalues-X}
	Keep the above assumption. 
	Let $\Sigma$ be the set of eigenvalues of the adjoint action of $X$ on $\cfg$. 

	\textnormal{(i)} Each element of $\Sigma$ is either zero or a $p$-adic unit. 

	\textnormal{(ii)} There is a natural $\mu_m$-action on $\Sigma$. 
\end{lemma}

\begin{proof}
	(i) Since $\phi_K$ and $X$ are identified via the isomorphism $\mathfrak{t}^*/\!\!/ W \simeq \check{\mathfrak{t}}/\!\!/ W$, it suffices to prove the assertion for $\phi_K$. 
	Since both $\phi_k$ and $\phi_K$ are stable linear functions, 
	the reduction module $\mathfrak{m}_K$ of any non-zero eigenvalue of $\phi_K$ is still a non-zero eigenvalue of $\phi_k$. 
	Then the lemma follows. 

	(ii) Recall that $\theta$ denotes the automorphism of $\cfg$ defining the grading on $\cfg$.
	Then we have $\theta(X)=\zeta_m X$. 
	Let $\cft_X$ be the centralizer of $X$, which is a Cartan subalgebra of $\cfg$ and $\Phi=\Phi(\cfg,\cft_X)$ the set of roots with respect to $\cft_X$. 
	The action of $\theta$ preserves $\cft_X$ and induces an action on $\Phi$ by $\theta\cdot \alpha(Y)=\alpha(\theta Y)$ for $Y\in \cft_X$ and $\alpha\in \Phi$. 
	In particular, we have $\theta\cdot \alpha(X)=\zeta_m \alpha(X)$. 
	Then the assertion follows. 
\end{proof}

\subsection{Airy local systems} \label{ss:Ai-rig}
We consider another family of rigid local systems constructed in \cite{JKY}, generalizing the classical Airy equations.
Its de Rham variant is isoclinic at its unique singularity.

\begin{secnumber}\label{sss:Ai-rig}
	We recall the construction of \emph{Airy automorphic datum} in \cite[\S2.3]{JKY}. 
	Assume that $p$ is strictly larger than the Coxeter number $h$. 
	Let $\bP=\bI\subset G(k(\!(t)\!))$ be the Iwahori subgroup
	with Moy-Prasad filtration subgroups $\bI(r)$, $r\in\mathbb{N}$.
	Let $N\in\fg$ be a sum of basis of negative simple root subspaces
	and $E$ a basis of highest root subspace.
	Let $S=C_{G(k(\!(t)\!))}(tN+E)$ be an elliptic torus.
	Denote $S(r)=S\cap \bI(r)$.
	
	When the Coxeter number $h$ is even, we let $J=S(1)\bI(1+h/2)$.
	When $h$ is odd, which happens only when $G$ is of type $A_{2n}$,
	let $P$ be the parahoric subgroup defined by $\check{\rho}/2n$ (\ref{sss:formal-isoclinic})
	and let $J=S(1)P(1+n)$.
	In both cases, we have $\bI(1+h)\subset J\subset \bI(1)$.
	
	Let $\phi\in V_{\bI}^*=(\bI(1)/\bI(2))^*$ be the linear form induced from $tN+E\in\fg(k(\!(t)\!))$
	via the Killing form and the restriction to $\mathrm{Lie}(\bI(1))$.
	It is stable in the sense of \S~\ref{sss:parahoric}
	and induces a character 
	\[\bI(1+h)\rightarrow \bI(1+h)/\bI(2+h)\simeq I(1)/I(2)\rightarrow k.\]
	It trivially extends to a character on $\bI(1+h/2)$ (resp. $P(1+n)$) 
	when $h$ is even (resp. odd).
	Let $\chi$ be any extension of $\phi$ to $J$:
	\begin{equation}
		\chi: J\to \mathbb{A}^1.
		\label{eq:chi}
	\end{equation}
	
	Let $\mathcal{G}$ be the group scheme on $\mathbb{P}^1$ such that 
	$\mathcal{G}\mid_{\mathbb{A}^1}\simeq G\times\mathbb{A}^1$
	and $\mathcal{G}|_{D_\infty}\simeq I(2+h)$. 
	Let $\Bun_{\mathcal{G}}$ be the moduli stack of $\mathcal{G}$-bundles on $\mathbb{P}^1$. 
	Denote $L^-G=G[t^{-1}]$, 
	then we have 
	\[\Bun_{\mathcal{G}}\simeq L^-G\backslash G(\!(t)\!)/I(2+h).\]
	Let $j:O:=L^-G\backslash L^-G J/I(2+h)\hookrightarrow\Bun_{\mathcal{G}}$ be the relevant orbit \cite[Proposition 19]{JKY}.
	Set $\mathcal{A}_\chi=j_{!+}(\chi^+\mathscr{A}_\psi)$ for the Dwork $F$-isocrystal (resp. the Artin--Schreier sheaf) $\mathscr{A}_{\psi}$.
	By \cite[Lemma 21, Theorem 26]{JKY}, 
	$\mathcal{A}_\chi$ is a holonomic module that is a Hecke eigen-module. 
	We denote the associated Hecke eigenvalue by
	\[
	\Ai_{\cG}^{\rig}(\chi):\Rep(\cG) \to \FIsod(\mathbb{A}^1_k/K), \qquad \textnormal{(resp. } \Ai_{\cG}^{\ell}(\chi):\Rep(\cG) \to \LocS(\mathbb{A}^1_k)).
	\]
\end{secnumber}

\begin{secnumber} \label{sss:Ai-dR}
	Similarly to the generalized Kloosterman case,
	we can take a variant of the above construction 
	using algebraic $\mathscr{D}$-modules instead of arithmetic $\mathscr{D}$-modules to produce a $\cG$-connection on $\mathbb{G}_{m,K}$ 
	by replacing $\mathscr{A}_\psi$ with
	$\EE_{\lambda}=K\langle x, \partial_x \rangle/(\partial_x+\lambda)$ with a parameter $\lambda\in K-\{0\}$ on $\mathbb{A}^{1}_{K}$.
	In the definition of $\chi$, we replace $k$ with $K$. 
	We obtain a (de Rham) Hecke eigenvalue:
	\begin{displaymath} 
		\Ai_{\cG}^{\dR}(\lambda\chi): \Rep(\cG)\to \Conn(\mathbb{A}^1_K). 
	\end{displaymath}
	
	Let $p_{-1},\{p_i\}_{i=1}^r$ be as in the proof of Lemma \ref{l:isoclinic gauge},
	where $r$ is the rank of $\cG$, $d_i$'s are the fundamental degrees.
\end{secnumber}
	
	\begin{theorem}\label{t:Ai connection}
		The $\cG$-connection 
		$\Ai^{\dR}_{\cG}(\lambda\chi)$ 
		is isomorphic to a following form:
		\begin{equation} \label{eq:Ai-conn}
			\Ai(\underline{a},\lambda)=d+(p_{-1}+a\lambda^h x p_r+\sum_{i=1}^{r-1} a_i \lambda^{d_i} p_i)dx,\quad a\in K^\times,~ a_i\in K,~ 1\le i\le r-1.
	\end{equation}
		Its restriction $\Ai_{\cG}^{\dR}(\lambda\chi)\mid_{D_\infty^\times}$ at infinity is an isoclinic formal $\cG$-connection of slope $1+\frac{1}{h}$ \eqref{sss:formal-isoclinic}. 
		Moreover, the leading term $X=p_{-1}+a\lambda^h p_r$ of its canonical form 
		corresponds to 
		$\phi=\chi\mid_{I(1+h)/I(2+h)}\in(\bI(1+h)/\bI(2+h))^*\simeq V_{\bI}^*$
		under the correspondence (i) in Theorem \ref{t:compare-Kl-theta-connection}
		for $\bP=\bI$.
	\end{theorem}
    \begin{proof}
	    When $\cG=\cG_{\ad}$ is of adjoint type, the theorem follows from \cite[Theorem 36, Proposition 38, Theorem 41]{Yi25}. 

	    Next, we explain how to extend the result to allow $G$ to be a general almost simple group. 
	    We claim that, up to isomorphism, there exists a unique $\cG$-connection on $\mathbb{A}^1$, which induces $\Ai_{\cG_{\ad}}(\underline{a},\lambda)$ as $\cG_{\ad}$-connection. 
	    Indeed, any two such de Rham $\cG$-local system differ by a $\check{Z}$-connection on $\mathbb{A}^1$. As $\check{Z}$ is finite, the wild part of  the differential Galois group at $\infty$ of this $\check{Z}$-connection must be trivial, and therefore this connection is trivial. 
	    Then the theorem in this case follows. 
    \end{proof}

\subsection{Comparison between $\Ai_{\cG}^{\dR}$ and $\Ai_{\cG}^{\rig}$} \label{ss:compare-dR-rig-Airy}

In this subsection, we work over $\OK$. 
Let $j:U=\A1\to X=\mathbb{P}^1$ be the open immersion over $\mathcal{O}_K$. We assume that the functions $\phi,\chi$, fixed in \S~\ref{sss:Ai-rig}, are defined over $\OK$ such that $\phi_k$ (resp. $\phi_K$) is stable in the sense of \S~\ref{sss:parahoric}. 

We prove a similar result of \Cref{t:compare-Kl-rig-dr} for Airy connections. 

\begin{theorem} \label{t:compare-Ai-rig-dr}
	There exists a canonical isomorphism of $\cG$-valued $j^{\dagger}\mathscr{O}_{\XX^{\rig}}$-modules with connection
		\[
		\iota:(\Ai_{\cG}^{\dR}(-\pi \chi_K))^{\dagger}\xrightarrow{\sim} \Ai_{\cG}^{\rig}(\chi_k).
		\]
	\end{theorem}

	\begin{secnumber}
		As in \ref{ss:Frob-theta-connection}, if we take $\lambda=-\pi$, the above theorem shows that the $\cG$-connection $\Ai_{\cG}(\underline{a},-\pi)$ \eqref{eq:Ai-conn} can be equipped with a $\cG$-valued Frobenius structure and underlies a $\cG$-overconvergent $F$-isocrystal $\Ai_{\cG}^{\rig}(\chi_k)$. 
		Moreover, since $\phi_k$ and $\phi_K$ are both stable, we can deduce that the coefficient $a$ in \eqref{eq:Ai-conn} is a $p$-adic unit in this case. 
	\end{secnumber}

\begin{proof}
	The open immersion 
	$O':=L^-G\backslash L^-GI(1)/I(2+h)\to \Bun_{\mathcal{G}}$ is defined by the non-vanishing locus of a section of a line bundle and is therefore affine \cite[Proposition 19]{JKY}. 
	The relevant orbit $O=L^-G\backslash L^-G J/I(2+h)$ is a closed substack of $O'$. 
	The geometric structure is slightly different from the setting of the construction of $\Kl_{\cG}$, where the relevant orbit is an open substack of $\Bun_{\mathcal{G}}$. 
	In the following, we explain how to modify the argument of \cite[\S~4.2]{XZ22} to prove the theorem using the relative specialization map \eqref{eq:relative-specialization} in the appendix.
	On $O$ we have the rank-one objects
\[
	\chi^+E_{-\pi}\quad\textnormal{(algebraic on }O_K\text{)}\qquad\text{and}\qquad
\chi^+\mathscr A_\psi\quad\textnormal{(arithmetic on }O_k\text{)},
\]
which are associated in the sense of Definition \ref{d:associated}.

	Consider the Hecke stack over $U=\mathbb{A}^1$ and the natural projections:
	\[
	\Bun_{\mathcal{G}}\xleftarrow{\pr_1} \Hk_{\mathcal{G}}^U \xrightarrow{\pr_2} \Bun_{\mathcal{G}}\times U.
	\]
	Let $\star\in \Bun_{\mathcal{G}}$ be the base point corresponding to the trivial bundle $\mathcal{G}$. 
	The base change of the above diagram to $\star\times U$ can be written as:
	\[
	\Bun_{\mathcal{G}}\xleftarrow{p_1} \GR_{\mathcal{G}} \xrightarrow{p_2} U,
	\]
	where $\GR_{\mathcal{G}}$ is the Beilinson-Drinfeld Grassmannian of $\mathcal{G}$ with modifications on $U$. Note that $\GR_{\mathcal{G}}\simeq \Gr_G\times U$. 

	Let $V$ be a representation of $\cG$. We denote by $\GR_{\mathcal{G},V}\subset \GR_{\mathcal{G}}$ the support of the Satake sheaf of $V$.  
	Consider the following diagram:
	\[
	\xymatrix{
		& \widetilde{O} 
		\ar@{^{(}->}[r]^-i \ar[ld]_{p}
		& \GR_{\mathcal{G},V}^{\circ} 
		\ar@{^{(}->}[r] |{\circ} \ar[ld] 
		& \GR_{\mathcal{G},V} \ar[ld] \ar[rd] & \\
		O 
		\ar@{^{(}->}[r] 
		& O' 
		\ar@{^{(}->}[r] |{\circ} 
		& \Bun_{\mathcal{G}} 
		&& U }
	\]
	where the squares are Cartesian. 
	Let $f:\GR_{\mathcal{G},V}^{\circ}\to U$ be the composition in the above diagram. 
	By \Cref{p:proper-pushforward}, the following modules are associated on $\GR_{\mathcal{G},V}^{\circ}$: 
\[
	M_V:=i_+\bigl(p^+(\chi^+E_{-\pi})\bigr)\in \Coh(\mathscr D_{\GR_{\mathcal G,V}^\circ/K}),
\qquad
	\mathscr M_V:=i_+\bigl(p^+(\chi^+\mathscr A_\psi)\bigr)\in \Hol(\GR_{\mathcal G,V,k}^\circ),
\]
By construction of the Airy eigensheaves \(A_\chi\) and the definition of the Hecke eigenvalue functors, we have (exactly as in \cite[\S~4.2]{XZ22}) canonical isomorphisms
\begin{equation}\label{eq:Ai-as-pushforward}
\Ai^{\dR}_{\cG,V}(-\pi\chi_K)\simeq f_{K,+}(M_V),
\qquad
\Ai^{\rig}_{\cG,V}(\chi_k)\simeq f_{k,+}(\mathscr M_V).
\end{equation}

	Following the argument of \cite[\S~4.2]{XZ22}, we consider the following two cases:

	(i) Case $V$ is minuscule. In this case, $\widetilde{O}, \GR_{\mathcal G,V}^\circ$ are smooth over $\OK$ and  $f$ admits a good compactification (cf. \cite[4.2.3]{XZ22}). 
The relative specialisation map \eqref{eq:relative-specialization} is a morphism of coherent $j^\dagger\mathscr O_{\XX^{\rig}}$-modules with connection
	\[
	\iota_V:(f_{K,+}(M_V))^{\dagger}\to f_{k,+}(\mathscr{M}_V).
	\]
	By the same argument of \cite[4.2.5]{XZ22}, we can show that the above map is an isomorphism. 

	(ii) Case $V$ is associated to a quasi-minuscule coweight $\lambda$. 
	As in \cite[4.2.6]{XZ22}, we take an isomorphism $\GR_{\mathcal{G},V}\simeq U\times \Gr_{\le \lambda}$ and set $\GR_{\mathcal{G}, V}^{\circ\circ}=\GR_{\mathcal{G}, V}^{\circ}\cap (X\times \Gr_{\lambda})$ to be the smooth locus of $\GR_{\mathcal{G}, V}^{\circ}$.
	We denote by $j:\GR^{\circ\circ}_{\mathcal{G}, V}\to \GR^{\circ}_{\mathcal{G}, V}$ the open immersion and by 
	\begin{equation}
		\tau=p_2^{\circ}\circ j:\GR_{\mathcal{G}, V}^{\circ\circ}\to X
	\end{equation}
	the canonical morphism, which admits a good compactification $\widetilde{\Gr}_{\le \lambda}\times \P1 \to \P1$ in the sense of \ref{sss:specialization-map}. 
	Let $M_V^{\circ}$ and $\mathscr{M}_V^{\circ}$ be the restriction of $M_V$ and $\mathscr{M}_V$ on $\GR^{\circ\circ}_{\mathcal{G},V}$ respectively, which are associated. 
	Then we have $j_{+!}(M_V^{\circ})\simeq M_V$, $j_{+!}(\mathscr{M}_V^{\circ})\simeq \mathscr{M}_V$. 
	Following \cite[Lemma 4.2.7(i)]{XZ22}, we can show that:
	
	(a) The complex $\tau_{k,+}(\mathscr{M}_V^\circ)[1]$ (resp. $\tau_{K,+}(M_V^\circ)[1]$) is holonomic. 
 
	(b) Let $s$ be a point of $U(k)$. We choose a lifting in $U(\OK)$ and still denote it by $s$.
	Then the specialisation morphism on the fiber cohomology 
	\begin{displaymath} 
		\rH^{*}_{\dR}( (\GR_{\mathcal{G}, V,s}^{\circ\circ})_K,M_{V,s}^{\circ})\to \rH^{*}_{\rig}((\GR_{\mathcal{G}, V,s}^{\circ\circ})_k,\mathscr{M}_{V,s}^{\circ}) 
	\end{displaymath}
	induces an isomorphism
	\begin{equation} \label{sp quasi-min}
		\rH^{0}_{\dR}( (\GR_{\mathcal{G}, V,s}^{\circ})_{K}, j_{!+}(M_{V,s}^{\circ}))\xrightarrow{\sim} \rH_{\rig}^{0}( (\GR_{\mathcal{G}, V,s}^{\circ})_{k}, j_{!+}(\mathscr{M}_{V,s}^{\circ})).
	\end{equation}

	By (a), we have a diagram of $D_{X_K}$-modules:
	\[
		\xymatrix{
			\Gamma(U_{K},\Ai_{\cG,V}^{\dR}(-\pi\chi_K)) \ar[r] & 
			\Gamma(U_{K},\tau_{K,+}(M_V^\circ)) \ar[d] & \\
			\Gamma(U_{k},\Ai_{\cG,V}^{\rig}(\chi_k)) \ar[r] &
			\Gamma(U_{k},\tau_{k,+}(\mathscr{M}_V^\circ))
		}
	\]
	where the vertical arrow is the relative specialization morphism \eqref{eq:relative-specialization}. 
	By the argument of \cite[4.2.8]{XZ22}, we can conclude that the vertical map induces an isomorphism of $j^{\dagger}\mathscr{O}_{\XX^{\rig}}$-modules with connection
	\[
	(\Ai^{\dR}_{\cG,V}(-\pi\chi_K))^{\dagger}\xrightarrow{\sim} \Ai^{\rig}_{\cG,V}(\chi_k).
	\]

	Finally, we deduce the theorem from the minuscule and quasi-minuscule cases as in \cite[4.2.9]{XZ22}. 
\end{proof}
\begin{secnumber} \label{sss:trivial-functoriality-Ai}
	\textbf{Trivial functoriality of $\Ai_{\cG}$.}
	Let $\cG\to \cG'$ be a homomorphism of reductive groups inducing an isomorphism on their adjoint quotients $\cG_{\ad}\xrightarrow{\sim} \cG'_{\ad}$. 
	By \Cref{t:Ai connection}, $\Ai_{\cG'}(\lambda\chi)$ is isomorphic to the push-out of $\Ai_{\cG}(\lambda\chi)$ along $\cG\to \cG'$ in view of the explicit connection form \eqref{eq:Ai-conn}. 

	Then we can deduce the overconvergent case from the de Rham case by \Cref{t:compare-Ai-rig-dr}.
	Since the analytification functor $(-)^{\dagger}$ naturally commutes with the push-out of structure groups, the trivial functoriality holds for the overconvergent isocrystals $\Ai_{\cG}^{\rig}$ as well. 
\end{secnumber}

\section{Application to monodromy representations} \label{s:application}
\subsection{Review of companions} \label{ss:companion}
In this subsection, we assume that $k=\mathbb{F}_q$ and $\sigma=\id_K:K\to K$. 

\begin{secnumber} \label{sss:companion-GLn}
	Let $X$ be a smooth geometrically connected curve over $k$. 
	Let $\ell$ be a prime.  
	We denote by $\mathcal{T}_\ell(X)$ the Tannakian subcategory of the category of lisse $\overline{\mathbb{Q}}_{\ell}$-sheaves on $X$ if $\ell\neq p$ (resp. $\FIsod(X/K)\otimes_K \overline{K}$ if $\ell=p$) consisting of semisimple objects $\mathcal{E}$ such that the determinant of each irreducible component of $\mathcal{E}$ has finite order. 

	Let $\mathcal{E}$ be an object of $\mathcal{T}_{\ell}(X)$. By \cite{Laf02,Abe18}, we have:
	\begin{itemize}
		\item for each $x\in |X|$, the polynomial $\det(1-t\cdot F_x,\mathcal{E})$ belongs to $\bQQ[t]$;

		\item for every prime $\ell'$, there exists an \textit{$\ell'$-companion} $\mathcal{E}'$ of $\mathcal{T}_{\ell'}(X)$ such that
			\[
			\det(1-t\cdot F_x,\mathcal{E})=\det(1-t\cdot F_x,\mathcal{E}') \quad \textnormal{for all } x\in |X|.
			\]
	\end{itemize}

	Note that an $\ell'$-companion $\mathcal{E}'$ of $\mathcal{E}$ is unique up to isomorphism \cite[Proposition 1.1.2.1]{Lau87}. 

	In the above construction of Kloosterman/Airy $\cG$-local system $E=\Kl_{\cG}$ (resp. $\Ai_{\cG}$), the $p$-adic realization $E^{\rig}(V)$ is the $p$-adic companion of the $\ell$-adic realization $E^{\ell}(V)$ for every $V\in\Rep(\cG)$. 
\end{secnumber}

\begin{secnumber}
	Let $\overline{X}$ be a smooth compactification of $X$. 
	For a closed point $x\in \overline{X}$, we can associate a Weil--Deligne representation $\rho_{\mathcal{E},x}^{\WD}$ to an object $\mathcal{E}$ of $\mathcal{T}_\ell(X)$ at $x$ with coefficients in $\bQl$. 
	We fix an isomorphism $\bQl\simeq \mathbb{C}$ for every prime $\ell$. 
	We may consider $\rho_{\mathcal{E},x}^{\WD}$ and its Frobenius semisimplification $\rho_{\mathcal{E},x}^{\WD,\sss}$ as representations with coefficients in $\mathbb{C}$. 
	One can compare the Frobenius semisimplified Weil--Deligne representation associated to $\mathcal{E}$ via the companion. 
\end{secnumber}

\begin{prop}[\cite{Del73} Th\'eor\`eme 9.8, \cite{KZ25} Theorem 4.4.5] \label{p:companion-WD-rep}
	Let $\mathcal{E}$ be an object of $\mathcal{T}_{\ell}(X)$, $\ell'$ a prime and $\mathcal{E}'$ the $\ell'$-companion of $\mathcal{E}$. 
	Then for any closed point $x\in\overline{X}$, 
	we have an isomorphism up to conjugation in $\GL_{n,\mathbb{C}}$:
	\[
	\rho_{\mathcal{E},x}^{\WD,\sss} \sim_{\GL_{n,\mathbb{C}}} \rho_{\mathcal{E}',x}^{\WD,\sss}. 
	\]
\end{prop}

\begin{secnumber}
	Following Drinfeld \cite{Dri18}, we revise a notion of \textit{companion of $\cG$-local systems}, generalizing the definition from $\GL_n$ to a semisimple group $\cG$. 

	Let $E$ be an algebraically closed field. 
	We refer to \cite[\S~2.1.1]{Dri18} for the notion of pro-semisimple groups and pro-reductive groups. 
	Following \cite[\S~1.2.3]{Dri18}, $\Pred(E)$ denotes the groupoid whose objects are pro-reductive groups over $E$ and whose morphisms are as follows: a morphism $G_1\to G_2$ is an isomorphism of group schemes $G_1\xrightarrow{\sim} G_2$ defined up to composing with automorphisms of $G_2$ of the form $x\mapsto gxg^{-1}$, for $g$ in the neutral component $G_2^{\circ}$ of $G_2$. 
	Let $\Pss(E)\subset \Pred(E)$ be the full subcategory formed by pro-semisimple groups. 

	For any pro-algebraic group $G$ over $E$, we denote by $G^{\red}$ (resp. $G^{\sss}$) its pro-reductive (resp. pro-semisimple) quotient. 

	For any pro-reductive group $G$, let $[G]$ denote the GIT quotient of $G$ by the conjugation action of the neutral component $G^{\circ}$.
	We have the projection $G\to [G]$ and the canonical map $[G]\twoheadrightarrow \pi_0(G):=G/G^{\circ}$. 
\end{secnumber}

\begin{secnumber}
	Fix an algebraic closure $\bQQ$ of $\mathbb{Q}$. Let $\ell$ be a prime. 
	Drinfeld defined a pro-semisimple group $\HPi_{\ell}$ over $\bQl$ as the Tannakian group of the Tannakian subcategory $\mathcal{T}_{\ell}(X)$ (\ref{sss:companion-GLn}). 
	The embedding $\bQQ\to \bQl$ induces an equivalence of groupoids \cite[Proposition 2.2.5]{Dri18}:
	\begin{equation}
		\Pss(\bQQ)\xrightarrow{\sim} \Pss(\bQl).
		\label{eq:equi alg closed}
	\end{equation}
	We denote by $\HPi_{(\ell)}$ the object of $\Pss(\bQQ)$ associated to $\HPi_{\ell}$ by above equivalence. 
\end{secnumber}
\begin{secnumber}
	Let $\widetilde{X}$ be the universal cover of the category of finite \'etale covers of $X$ and $\Pi$ the automorphism of $\widetilde{X}/X$, which is isomorphic to the \'etale fundamental group $\pi_1(X,\overline{x})$ after choosing a base point $\overline{x}$. 

	The profinite group $\Pi$ admits a dense subset $\Pi_{\Fr}$ formed by powers of geometric Frobenius $F_{\widetilde{x}}^n$, where $\widetilde{x}$ runs through $|\widetilde{X}|$ and $n$ runs through $\mathbb{N}$. 

	We have an affine scheme $[\HPi_{(\ell)}]$ over $\bQQ$, a morphism $[\HPi_{(\ell)}]\to \pi_0(\HPi_{(\ell)})=\Pi$ and canonical morphisms \cite[(1.5), (7.12)]{Dri18}:
\begin{equation}
	\Pi_{\Fr}\to [\HPi_{(\ell)}](\bQQ) \twoheadrightarrow \Pi.
	\label{eq:Frob seq}
\end{equation}
	
	The main result \cite[Theorem 1.4.1]{Dri18} says that: 
	Given two primes $\ell, \ell'$, 
	there exists a unique isomorphism 
	\begin{equation} \label{eq:iso-companion}
	\HPi_{(\ell)}\xrightarrow{\sim} \HPi_{(\ell')}
	\end{equation}
	in the groupoid $\Pss(\bQQ)$ which sends the diagram \eqref{eq:Frob seq} to the corresponding diagram of $\HPi_{(\ell')}$. 
\end{secnumber}

\begin{secnumber}
	Given group schemes $G_1,G_2$ over $E$, let $\Hom_{\cor}(G_1,G_2)$ denote the quotient of the set $\Hom_E(G_1,G_2)$ by the conjugation action of $G_2^{\circ}(E)$. 
	Let $\widetilde{E}$ be an algebraically closed field containing $E$. 
	If $G_1$ is pro-reductive, the following canonical map is bijective \cite[Proposition 2.3.3]{Dri18}:
	\begin{equation}
		\Hom_{\cor}(G_1,G_2)\xrightarrow{\sim} \Hom_{\cor}(G_1\otimes_E \widetilde{E},G_2\otimes_E \widetilde{E}).
		\label{eq:Hom-base-change}
	\end{equation}
\end{secnumber}

\begin{definition} \label{d:G-companion}
	Let $\cG$ be a connected semisimple group over $\bQQ$, and $\varphi_{\ell}:\HPi_{\ell}\to \cG$ an $\ell$-adic $\cG$-local system. 
	For another prime $\ell'$, we define the $\ell'$-companion $\varphi_{\ell'}:\HPi_{\ell'}\to \cG$ of $\varphi_{\ell}$ as the image of $\varphi_{\ell}$ via the isomorphisms \eqref{eq:iso-companion} and \eqref{eq:Hom-base-change}:
	\[
	\Hom_{\cor}(\HPi_{\ell},\cG_{\bQl})\xleftarrow{\sim} \Hom_{\cor}(\HPi_{(\ell)},\cG)\xrightarrow{\sim} \Hom_{\cor}(\HPi_{(\ell')},\cG)\xrightarrow{\sim} \Hom_{\cor}(\HPi_{\ell'},\cG_{\overline{\mathbb{Q}}_{\ell'}}).
	\]
\end{definition}
\begin{prop} \label{p:companion-Frob-trace}
	With the notation of \Cref{d:G-companion}, for every representation $r:\cG\to \GL(V)$, $r\circ \varphi_{\ell'}$ is the $\ell'$-companion of $r\circ \varphi_{\ell}$ in the sense of \ref{sss:companion-GLn}.  
\end{prop}
\begin{proof}
	We consider the following commutative diagram:
	\[
	\xymatrix{
	\Pi_{\Fr} \ar[r] \ar@{=}[d] & [\HPi_{(\ell)}](\bQQ) \ar[r]^{\varphi_{(\ell)}} \ar[d]_{\sim} & [\cG](\bQQ) \ar[r]^-{r} &[\GL(V)](\bQQ) \\
	\Pi_{\Fr} \ar[r] & [\HPi_{(\ell')}](\bQQ) \ar[ru]_{\varphi_{(\ell')}} &&
	}
	\]
	Hence for every geometric Frobenius $F_{\widetilde{x}}$, it has the same image in $[\GL(V)](\bQQ)$ via $\varphi_\ell,\varphi_{\ell'}$. 
	This finishes the proof. 
\end{proof}

\begin{prop} \label{p:companion-Kl}
	Suppose $m=h$ is the Coxeter number. 
	Let $\ell$ be prime different from $p$ and $\phi$ a stable function of $\fg^*_1(k)$. 
	Then $\Kl_{\cG}^{\rig}(\phi)$ \eqref{sss:Kl^rig} is the $p$-adic companion of the $\ell$-adic Kloosterman sheaf $\Kl_{\cG}^{\ell}(\phi)$ in the sense of \Cref{d:G-companion}. 
\end{prop}
\begin{proof}
	Let $\mathcal{E}$ be the $\ell$-adic companion of $\Kl_{\cG}^{\rig}(\phi)$ in the sense of \Cref{d:G-companion}. 
	By construction and \Cref{p:companion-Frob-trace}, for every representation $V$ of $\cG$ and $a\in |\mathbb{G}_{m,k}|$, we have 
	\[
	\Tr(\Frob_a,\mathcal{E}_{\overline{a}})=\Tr(\Frob_a,\Kl^{\ell}_{\cG}(\phi)_{\overline{a}}). 
	\]

	Note that $\mathcal{E}$ and $\Kl^{\ell}_{\cG}(\phi)$ have the same geometric monodromy group as calculated in \cite[Theorem 4.5.2, Corollary 4.5.8]{XZ22}. 
	Then we can prove that two representations $\mathcal{E}$ and $\Kl^{\ell}_{\cG}(\phi)$ are conjugate in the geometric monodromy group by a similar argument of \cite[Theorem 5.1.4]{XZ22}. 
\end{proof}

\begin{rem}
	We expect that the same holds without assumption $m=h$ and also for $\Ai_{\cG}^{\rig}, \Ai_{\cG}^{\ell}$ as they satisfy the property in \Cref{p:companion-Frob-trace}. 
\end{rem}

\begin{rem} \label{r:companion-local-global} 
	Let $\cG$ be a connected semisimple group over $\overline{\bQ}$, $\cE$ an $\ell$-adic $\cG$-local system on a smooth curve $X$ over $k$ (we allow $\ell=p$). 
	Let $\ell'$ be another prime and $\cE'$ the $\ell'$-companion of $\cE$. 
	Let $\overline{X}$ be a smooth compactification of $X$ and $x$ a closed point of $\overline{X}-X$. 
	Inspired by \Cref{p:companion-WD-rep}, one may expect that there exists an isomorphism up to conjugation in $\cG$ between the associated Frobenius semisimplified Weil--Deligne representations at $x$:
	\begin{equation} \label{eq:companion-local-reps}
	\rho_{\mathcal{E},x}^{\WD,\sss} \sim_{\cG} \rho_{\mathcal{E}',x}^{\WD,\sss}. 
	\end{equation}

	For acceptable groups in the sense of Larsen \cite{Larsen}, such as $\Sp_{2n}, \SO_{2n+1}, G_2$, and $\SO_4$, this expectation can be unconditionally verified via the fact that ``elementwise conjugacy'' implies ``global conjugacy''. 	
\end{rem}

\subsection{Epipelagic Langlands parameters: equal characteristic case}
\begin{secnumber}
	We take again the notations of \ref{ss:Frob-theta-connection}.
	Recall that $X=X_1+X_{1-m}$ is regular semisimple in $\cg(K)$ and in $\cg(k)$. 
	We denote by $\cft_X$ the centralizer of $X$ in $\cfg$ and by $\Phi$ the set of roots of $\cfg$ with respect to $\cft_X$. 
	We set $r=\dim \cft_X$ and $\ell=r+\sharp \Phi$ the dimension of $\cfg$.

	Consider the $\theta$-connection \eqref{eq:theta-connection-Frob} 
	which can be equipped with a Frobenius structure:
\begin{equation} \label{eq:p-theta-connection}
	 d+\biggl(X_1+(-\pi)^mX_{1-m}x\biggr)\frac{dx}{x}.
 \end{equation}

	Let $t=x^{-1}$ be a coordinate around $\infty$. We may rewrite
	\begin{equation} \label{eq:theta-conn-infty}
	d-\biggl(X_1+(-\pi)^m X_{1-m}t^{-1}\biggr)\frac{dt}{t}.
	\end{equation}
	We denote the associated differential module over $\mathcal{A}_{K,t}$ by $\mathscr{E}$ and abusively consider $\mathscr{E}$ as a differential module over the Robba ring $\mathcal{R}_K$. 
	We denote the associated formal connection over $K(\!(t)\!)$ by $\mathcal{E}$.
	Denote the ring $\mathcal{A},\mathcal{A}_K$ for $s$ by $\mathcal{A}_s,\mathcal{A}_{K,s}$.
\end{secnumber}

\begin{theorem}
	\label{t:p-adic-theta}
	The connection $\mathscr{E}$ is a $p$-adic isoclinic connection of slope $\frac{1}{m}$ over $\mathcal{A}_t$ \eqref{d:p-adic-isoclinic}. 
	More precisely, $[2m]^*\mathscr{E}$ can be $\cG(\mathcal{A}_s)$-gauge equivalent to a form as follows:
	\begin{equation} \label{eq:canonical-theta-conn}
	d+\pi X t^{-1/m} \frac{dt}{t}. 
	\end{equation}
\end{theorem}
The proof will be given in \S\ref{ss:5.3}, \S\ref{ss:5.4}.

\begin{secnumber}
	Let $(\rho:W_F\to \cG(\overline{K}),N)$ be the Weil--Deligne representation associated to $\mathscr{E}$ (\ref{sss:MCF}). 
	It can be viewed as the L-parameter associated to the epipelagic representation under the local Langlands correspondence à la Genestier--Lafforgue. 

	Indeed, in the equal characteristic case with $\ell$-adic coefficients, Genestier--Lafforgue attach to an irreducible admissible representation of $G(F)$ a semisimple Langlands parameter with values in $\cG$ \cite{GL}. 
	Moreover, Genestier--Lafforgue also showed the local-global compatibility. 
	For $p$-adic coefficients, where an analogous Genestier--Lafforgue construction is not presently available, we regard our construction as a substitute inspired by the local-global compatibility and \Cref{r:companion-local-global}. 

	By \Cref{t:L-parameter}, we verify the prediction of Reeder--Yu on $\rho$ \cite[\S~7.1]{RY14} (cf. \cite[\S~2.9]{Yun16}):
\end{secnumber}

\begin{coro} \label{c:Swan-conductor}
	\textnormal{(i)} 
	The nilpotent operator $N$ vanishes. 
	Let $\cT_X$ be the centralizer $C_{\cG}(X)$ of $X$ in $\cG$, which is a maximal torus. 
	The representation $\rho$ fits into a commutative diagram
	\[
	\xymatrix{
	1\ar[r] & I^+_F \ar[r] \ar[d] & I_F \ar[r] \ar[d] & I_F^t \ar[r] \ar[d] & 1 \\
	1\ar[r] & \cT_X \ar[r]& N_{\cG}(\cT_X) \ar[r] & \bW \ar[r] & 1
	}
	\]
	such that a generator of the tame inertia $I_F^t$ is sent to a regular elliptic element of the Weyl group $\bW$ of order $m$. 

	\textnormal{(ii)} We have $\cfg^{I_F}=0$ and $\Swan(\cfg)=\frac{\sharp \Phi}{m}$.
\end{coro}

\begin{coro}
	The local system $\Kl_{\cG}^{\rig}(\phi)$ (resp. its $\ell$-adic variant $\Kl_{\cG}^{\ell}(\phi)$) is cohomologically rigid, i.e. the following cohomology groups vanish
	\[
	\rH^{*}(\mathbb{P}^1_k,j_{!+}(\Kl_{\cG}^{\bullet}(\phi)(\cfg)))=0,
	\]
	where $j:\mathbb{G}_{m,k}\to \mathbb{P}^1_k$ is the open immersion, $j_{!+}$ denotes the intermediate extension functor and $\Kl_{\cG}^{\bullet}(\phi)(\cfg)$ is the local system associated with the adjoint representation of $\cG$ for $\bullet\in \{\rig,\ell\}$.  
\end{coro}
\begin{proof}
	As explained in \cite[Proposition 5.2]{Yun16}, the assertion follows from the calculation of the Swan conductor $\Swan(\cfg)=\frac{\sharp \Phi}{m}$ and the fact that $\cfg^{I_F}=0$. 
	In the $p$-adic case, it is proved in \Cref{c:Swan-conductor}. 
	In the $\ell$-adic case, it follows from the $p$-adic result and \Cref{p:companion-WD-rep}. 
\end{proof}
\begin{coro} \label{c:finite-frob-structures}
	The differential module $\mathscr{E}$ over $\mathcal{R}$ (resp. the $\cG$-overconvergent isocrystal $\Kl_{\cG}^{\rig}(\phi)$) admits at most finitely many Frobenius structures.
\end{coro}
\begin{proof}
	We first prove the local case. 
	As explained in \cite[Lemma 5.1.1]{XZ22}, any two Frobenius structures $\phi_1, \phi_2$ on $\mathscr{E}$ differ by an element $u = \phi_1 \circ \phi_2^{-1}$ in the centralizer $C_{\cG}(\rho(I_F))$ by the Tannakian equivalence. 
	By \Cref{c:WD-rep}(i), the centralizer $C_{\cG}(\rho(I_F)) = \cT_X^\sigma$ is finite and the assertion for $\mathscr{E}$ follows. 

	Let $G_{\geo}$ be the geometric monodromy group of $\Kl_{\cG}^{\rig}(\phi)$. Since $\rho(I_F)\subset G_{\geo}$, we have $C_{\cG}(G_{\geo})\subset C_{\cG}(\rho(I_F))$ and they are both finite. 
	Then the assertion follows. 
\end{proof}
\subsection{$P$-adic slopes of $\mathscr{E}$}\label{ss:5.3}

\begin{secnumber} \label{sss:def-Q}
	We first review the gauge transformation between $[2m]^*\mathcal{E}$ and its canonical form \eqref{eq:canonical-theta-conn} in the sense of Babbitt--Varadarajan \cite{BV83}. 
	Then we will prove the $p$-adic convergence of this gauge transformation in \Cref{t:p-adic-convergence-Q} and then deduce \Cref{t:p-adic-theta}. 

	Recall that $\check{\lambda}=mx$ defines the grading \eqref{eq:grading-cg}, where $\check{\lambda}:\Gm\to \check{T}$. 
	Recall $s=t^{1/2m}$.
	Apply the gauge transformation of $\check{\lambda}(-\pi s^{-1})$ to $[2m]^*\mathcal{E}$ \eqref{eq:theta-conn-infty}, we obtain
	\begin{equation} \label{eq:pullback-theta-connection}
		Y:= d+2m\pi(X_1+X_{1-m})\frac{ds}{s^2}-\check{\lambda}\frac{ds}{s},	
	\end{equation}
	denoted by $Y$. 
	By Lemma \ref{l:isoclinic gauge},
	after enlarging $K$, there exists an element $Q\in \cG (K\llbracket s \rrbracket)$ such that $Y$ can be gauge transformed to be the canonical form $Z$ \eqref{eq:canonical-theta-conn} of $Y$ via $Q$:
	\[
	Z:= d+2m\pi(X_1+X_{1-m})\frac{ds}{s^2}. 
	\]
	In particular, $Q$ satisfies a differential equation:
	\begin{equation} \label{eq:diff-Q}
	QYQ^{-1}+sQ'Q^{-1}=Z.
\end{equation}

After enlarging $K$, we may assume $K$ contains all eigenvalues of the adjoint action of $X$ on $\cfg$. 
For $\alpha\in \Phi$, we denote by $\varepsilon_\alpha$ the eigenvalue of the action of $X$ on the root subspace $\cfg_\alpha\subset \cfg$, which is a $p$-adic unit by \Cref{l:eigenvalues-X}. 
Then, as a $\GL(\cfg)$-connection on $\Gm$, $Z$ admits a decomposition 
\begin{equation}
	Z=Z_0 \bigoplus \oplus_{\alpha\in \Phi} Z_\alpha,
	\label{eq:decomposition-Z-root}
\end{equation}
where $Z_0$ is the trivial connection and $Z_{\alpha}=\exp(-2m\pi\varepsilon_{\alpha}/s)$ is the exponential $\mathscr{D}$-module $d+ 2m\pi\varepsilon_\alpha \frac{ds}{s^2}$. 

In particular, the formal slope of $\mathcal{E}$ equals to $\frac{1}{m}$. 
\end{secnumber}

\begin{prop} \label{p:positive-convergent-radius}
	There exists a real number $\beta>0$ such that $Q$ lies in $\cG(K\langle s/\beta\rangle)$ (i.e. it has a $p$-adic convergent radius $\beta>0$). 
\end{prop}
\begin{proof}
	We consider $Q$ as a matrix of $\GL(\cfg \llbracket s \rrbracket)$ via the adjoint action. 
	By \eqref{eq:diff-Q}, any entry of $Q$ satisfies a linear differential equation as in \eqref{eq:diff-eq} with coefficient in $(\overline{\mathbb{Q}}\cap K)[s]$. 
	By \Cref{t:Clark}, there exists a real number $\beta>0$ such that $Q$ lies in $\GL(\cfg\langle s/\beta\rangle)$. 
\end{proof}

\begin{prop} \label{t:subsidiary-radii}
	Via the adjoint representation $\cfg$, the multi-set of $p$-adic slopes of $\mathscr{E}(\cfg)$ consists of $\sharp \Phi$ copies of $\frac{1}{m}$ and $r$ copies of $0$. 

\end{prop}

\begin{lemma} \label{l:subsidiary}
	Let $\beta$ be the real number as in \Cref{p:positive-convergent-radius}. 
	For $\rho\in (0,\beta)$, the (multi-)set of subsidiary radii of $Y$ (resp. $Z$) at $F_{\rho}$ consists of $\sharp \Phi$ copies of $\rho^2$ and $r$ copies of $\rho$. 
\end{lemma}

\begin{proof}
	By \Cref{p:positive-convergent-radius}, it suffices to prove the assertion for $Z$. 
	For each $Z_\alpha$ in the decomposition \eqref{eq:decomposition-Z-root}, its generic radius $F_1(d+2m\pi\varepsilon_\alpha\frac{ds}{s^2},r)$ is a linear function on $(0,\infty)$ passing the origin $(0,0)$ with slope $2$. 
	Then the assertion follows. 
\end{proof}

\begin{secnumber}\textit{Proof of \Cref{t:subsidiary-radii}}. \label{sss:proof-4.3.3} 
	It suffices to calculate the subsidiary radii functions $F_i(Y,r)$ of the connection $Y$. 
	Since the connection $Y$ admits a Frobenius structure, it is solvable at $1$ and its subsidiary radii $R_i(Y,\rho)$ have limit $1$ as $\rho\to 1^{-}$ \cite[Lemma 12.6.2]{pde}. 

	Let $\ell$ be the dimension of $\cfg$. 
	For each $1\le i\le \ell$, the function $F_i(Y,r)$ is piecewise linear and convex on $(0,+\infty)$.  
	By \Cref{l:subsidiary}, the restriction of $F_i(Y,r)$ on $(-\log \beta, +\infty)$ is a straight line passing the origin $(0,0)$ with the slope
	\[
		=	\left\{
			\begin{array}{ll}
				2i & i\le \sharp \Phi, \\
				i + \sharp \Phi & i> \sharp \Phi. 
			\end{array}
			\right.
	\]
	Then we deduce that the same property holds for $F_i(Y,r)$ on $(0,+\infty)$. 
	Then the multi-set of $p$-adic slopes of $Y$ coincides with the multi-set of formal slopes of $Y$, viewed as a formal connection, and the proposition follows.  \hfill\qed
\end{secnumber}

\subsection{$P$-adic convergence of the gauge transform $Q$}\label{ss:5.4}
We finish the proof of \Cref{t:p-adic-theta} in this subsection. 

	We consider $Y,Z$ as $\GL(\cfg)$-connections via the adjoint action. 
	We can reorganize the decomposition \eqref{eq:decomposition-Z-root} as follows:
	\[
	Z=\oplus_{b\in k} Z_{b},
	\]
	where $Z_0$ is the same as $Z_0$ in \eqref{eq:decomposition-Z-root} and for $b\in k^{\times}$, $Z_b$ denotes the direct sum of $Z_{\alpha}$ such that $\varepsilon_\alpha$ is congruent to $b$ modulo $\mathfrak{m}_K$. 

In the following, we prove a similar decomposition for $Y$ over $\mathcal{A}_{K,s}$. \Cref{t:p-adic-theta} follows from \Cref{t:p-adic-convergence-Q}(iv). 

\begin{prop}
	\label{t:p-adic-convergence-Q}	
	\textnormal{(i)} There exists a decomposition of differential modules over $\mathcal{A}_{K,s}$: 
\begin{equation} \label{eq:decomposition-Y-refined}
	Y=\oplus_{b\in k} Y_b,
\end{equation}
	such that the $p$-adic slope of $Y_b\otimes \exp(-m\pi [b]/s)$ is zero with the Teichmüller lift $[b]$ of $b$. 

	\textnormal{(ii)} Under the isomorphism $Y\xrightarrow{\sim} Z$ defined by $Q$ over $K\langle s/\beta\rangle$, $Y_b$ is isomorphic to $Z_b$. 

	\textnormal{(iii)} The isomorphism $Y_b\xrightarrow{\sim} Z_b$ defined over $K\langle s/\beta\rangle$ can be extended to $\mathcal{A}_{K,s}$. 

	\textnormal{(iv)} 
	The element $Q\in \cG(K \llbracket s \rrbracket)$ lies in $\cG(\mathcal{A}_{K,s})$.  
\end{prop}

\begin{rem}
	For the Bessel $F$-isocrystal of rank $2$, \Cref{t:p-adic-convergence-Q}(iv) was proved by Dwork via an explicit calculation \cite{Dw74}, \cite[Example 6.2.6]{Tsu98-slope}. 
\end{rem}

\begin{proof}
	(i-ii) Let $0<\gamma<\delta<1$. 
	By \Cref{t:subsidiary-radii} and \cite[Theorem 12.3.1]{pde}, the differential module $Y$ over $K\langle \gamma/s, s/\delta\rangle$ admits the slope decomposition of differential modules over $K\langle\gamma/s, s/\delta\rangle$:
	\[
	Y=Y^0\oplus Y^1,
	\]
	such that for $\rho\in [\gamma,\delta]$, the component $Y^i\otimes F_\rho$ has subsidiary radii $\rho^{i+1}$ for $i=0,1$. 
	We take $Y_0=Y^0$ in \eqref{eq:decomposition-Y-refined}. The above decomposition is compatible for different $\gamma,\delta$. 

	For $b\in k^{\times}$, we define $Y_b$ such that $Y_b\otimes \exp(-2m\pi[b]/s)$ is the ($p$-adic) slope zero part of the differential module $Y\otimes \exp(-2m\pi[b]/s)$. 
	By the slope decomposition, we obtain an inclusion of differential modules over $K\langle\gamma/s, s/\delta\rangle$:
	\begin{equation} \label{eq:inclusion-b}
	\oplus_{b\in k} Y_b \subset Y. 
	\end{equation}

	On the other hand, by \Cref{p:rank-one-exp}, $Z_b\otimes \exp(-2m\pi[b]/s)$ is the ($p$-adic) slope zero part of the differential module $Z\otimes \exp(-2m\pi[b]/s)$. 
By \Cref{p:positive-convergent-radius}, via the isomorphism between $Y$ and $Z$ induced by $Q$, $Y_b$ is isomorphic to $Z_b$ over $K\langle \gamma/s,s/\delta\rangle$ for $0<\gamma<\delta\le \beta$. 
	This shows assertion (ii). 
	In particular, the inclusion \eqref{eq:inclusion-b} is actually a decomposition. Assertion (i) follows. 

	(iii) Let $Z^1$ be the complement of $Z_0$ in \eqref{eq:decomposition-Z-root}. 
	As an invertible matrix of $\GL(\cfg \llbracket t_n \rrbracket)$, $Q$ can be decomposed as 
	\[
	Q = 
	\begin{pmatrix}
		Q_0 & 0 \\
		0 & Q_1
	\end{pmatrix}
	\]
	such that $Q_0$ (resp. $Q_1$) induces the isomorphism between $Y_0$ (resp. $Y^1$) and $Z_0$ (resp. $Z^1$). 

	Since $Z_0$ has the zero $p$-adic slope and is trivial, the same holds for $Y_0$ over $K\langle \gamma/s,s/\delta\rangle$ for $0<\gamma<\delta\le \beta$. 
	By \cite[Corollary 13.6.4]{pde}, we deduce that $Y_0$ is isomorphic to the trivial differential module via a matrix $P(s)$ over the open annulus with outer radius $1$ and inner radius $0$. 
	We deduce that $P(s)Q_0(s)^{-1}\in \GL_r(K)$, where $r$ denotes the rank of $Z_0$. 
	Hence, the matrix $Q_0(s)$ has a convergent radius $|s|<1$ and this finishes the proof of assertions (iii) for $Y_0\xrightarrow{\sim} Z_0$. 

	For $b\in k^{\times}$, the matrix $Q$ still defines an isomorphism between the following two connections:
	\[
	Y\otimes \exp(-2m\pi [b]/s) ,\qquad Z\otimes \exp(-2m\pi [b]/s).
	\]
	The submodule $Y_{b}\otimes \exp(-2m\pi [b]/s)$ (resp. $Z_{b}\otimes \exp(-2m\pi [b]/s)$) is the ($p$-adic) slope zero part of $Y\otimes \exp(-2m\pi [b]/s)$ (resp. $Z\otimes \exp(-2m\pi [b]/s)$). 
	By repeating the above argument to these slope zero parts, we prove assertions (iii) for the isomorphism $Y_{b}\xrightarrow{\sim} Z_{b}$. 

	(iv) From the above proof, we see that $Q$ lies in $\cG(\mathcal{A}_{K,s})$. 
\end{proof}

\subsection{Simple wild parameters}

In this subsection, we assume that $m=h$ is the Coxeter number, $\nu=\frac{1}{h}$ and that $\cG$ is almost simple of adjoint type. 
Suppose $k=\mathbb{F}_q$ is the finite field with $q$ elements. 
Let $\varphi:W_F\to \cG(\overline{K})$ be a continuous homomorphism whose image  consists of semi-simple elements. 
Following \cite[\S~6]{GR10}, $\varphi$ is a \textit{simple wild parameter} if it satisfies the following two conditions:
\[
\cg^{I_F}=0 \quad \textnormal{and} \quad \Swan(\cg)=\rank(\cG).
\]


\begin{theorem} \label{t:number-simple-wild}
	Assume $p$ does not divide the order of the Weyl group $\bW$ and $\cG$ is adjoint. 
	There are exactly $Z(q)\cdot(q-1)$ equivalence classes of simple wild parameters, where $Z(q)=\sharp Z_G(F)$. 
\end{theorem}

\begin{rem}
	The assertion is mentioned in \cite[Remark 9.5]{GR10}. 
	We present a detailed proof following \cite{GR10, RY14} for the application in \S~\ref{s:rigidity}. 
\end{rem}


\begin{secnumber} \label{sss:image-structure}
	We first review some properties of simple wild parameters following \cite{GR10}. 
	Let $\varphi:W_F\to \cG(\overline{K})$ be a simple wild parameter. 
	Under the assumption $p$ does not divide the order of $\bW$, we have the following properties \cite[Propositions 5.6, 9.4]{GR10}:

		(i) The image $\varphi(W_F)$ is contained in the normalizer $N(\cT)$ of a maximal torus $\cT$ in $\cG$ with the ramification filtration:
			\[
			D=\varphi(W_F) \ge D_0=\varphi(I_F) \ge D_1=\varphi(I_F^+),\quad D_2=1.
			\]

		(ii) Let $\sigma\in D_0$ be the image of a tame generator $s$ of $I_F$. 
		Since $\cG$ is adjoint type, we have $D_0=D_1 \rtimes \langle \sigma\rangle$, where $\sigma\in N(\cT)$ has order $h$ and is a lift of a Coxeter element in the Weyl group $\bW$. 

		(iii) The group $D_1$ has order $p^a$, where $a$ is the order of $p$ in $(\bZ/h\bZ)^{\times}$.
		Suppose $D=\Gal(L/F)$. 
		Then, $D_1$ is the unique simple $\mathbb{F}_p[\sigma]$-submodule of $\cT[p]$ containing the $\sigma$-eigenvalue 
		\begin{equation} \label{eq:zeta-h}
		\zeta_h=\frac{s(\varpi)}{\varpi},
	\end{equation}
		where $\varpi$ is a uniformizer of $E=L^{D_1}$ such that $\varpi^h\in E^{D_0}$. 

		(iv) 
		Let $d$ be the order of $q$ in $(\mathbb{Z}/h\mathbb{Z})^*$. Then $D/D_0$ is cyclic of order $f=dc$, where $c$ divides the exponent of $\pi_1(\cG)$. 
		Since $D/D_1$ is the image of the tame quotient of $W_F$, we have a splitting 
		\[D/D_1\simeq D_0/D_1\rtimes D/D_0.\] 
		Let $\Fr\in D/D_0$ be a generator such that $\Fr\sigma \Fr^{-1}=\sigma^q$ in $D/D_1$. 
		Since $D_1$ is a simple $\mathbb{F}_p[\sigma]$-submodule of $\cT[p]$, $D_1^{\sigma}=1$. Thus by \cite[Lemma 6.1]{GR10}, we have a splitting 
		\[D=D_1 \rtimes D/D_1.\] 

		(v) We consider the group $\Gamma$ defined by
\begin{equation} \label{eq:presentation-Gamma}
	\Gamma=\langle s,t| s^h=1, t^{dh}=1, tst^{-1}=s^q\rangle=\langle s\rangle \rtimes \langle t\rangle.
\end{equation}
		The center of $\Gamma$ is $\langle t^d\rangle$. 
		There exists a quotient $\Gamma\to D/D_1$ sending $s,t$ to $\sigma,\Fr$ respectively and the kernel of this quotient is contained in the center. 

		(vi) Via the natural map $D/D_1\to \bW=N(\cT)/\cT$, $\Fr$ projects to an element $\tau$ in $\bW$ such that $\tau\sigma\tau^{-1}=\sigma^q$. 
		The action of $\tau$ on $\cT^{\sigma}$ is the same as the $q$-power map. And the centralizer of $\varphi(W_F)$ in $\cG$ is the finite group $\cT^{\sigma,q=1}$, which has the cardinality $Z(q)$. 
\end{secnumber}

\begin{lemma}
	\textnormal{(i)} The image $\varphi(I_F)=D_0\subset \cG$ of simple wild parameters is uniquely determined up to conjugations of $\cG$. 
	Namely, for any other simple wild parameter $\varphi'$, the image of inertia subgroup $D'_0=\varphi'(I_F)$ is conjugate to $D_0$. 

	\textnormal{(ii)} 
		The action of $\Gamma$ on $D_1$ via the quotient $\Gamma\to D/D_1$ is independent of the choice of a simple wild parameter. 
	\label{l:D0-unique}
\end{lemma}
\begin{proof}
	(i) Let $\varphi'$ be another simple wild parameter. Up to conjugation, we may assume $D_1'=\varphi(I_F^+)$ is contained in $\cT$. 
	Now, $D_0'=D_1'\rtimes \langle \sigma'\rangle$, where $\sigma'\in N(\cT)$ is a lift of a Coxeter element. 
	Up to conjugation, we may assume that $\sigma$ and $\sigma'$ have the same image in the Weyl group $\bW$ and that they are the image of a tame generator $s$ via $\varphi,\varphi'$. 
	Since elements in $\sigma\cT$ are conjugate under $\cT$, we may assume $\sigma=\sigma'$. 
	Since $\zeta_h$ \eqref{eq:zeta-h} is uniquely determined by $s$, we deduce $D_0=D_0'$ by \ref{sss:image-structure}(iii). 
	Then the lemma follows. 

	(ii)
	The subgroup $\langle s\rangle$ of $\Gamma$ is isomorphic to the subgroup $H$ generated by a Coxeter element $\sigma$ in $\bW$ via $\Gamma\to D/D_1 \to \bW$, and $t$ is sent to an element of the normalizer $N_{\bW}(H)$. 
	By \cite[Corollary 4.3]{Springer}, the centralizer of $H$ in $\bW$ is itself. 
	By \cite[Proposition 4.7]{Springer}, we have isomorphisms
	\[
	(\mathbb{Z}/h\mathbb{Z})^*\simeq N_{\bW}(H)/H \simeq \Aut(H), \qquad a\!\! \mod h \mapsto (\sigma\mapsto \sigma^a).
	\]
	Via this isomorphism, $t$ is sent to $q \mod h$ in $N_{\bW}(H)/H$ and is independent of the choice of a simple wild parameter.  
	Then the assertion follows. 
\end{proof}

\begin{secnumber} \label{sss:Gal-extension}
	Let $K$ be the unramified extension of $F$ of degree $dh$ and $\fF$ the residue field of $K$, which is isomorphic to $\mathbb{F}_{q^{dh}}$. 
	Let $M/K$ be the tamely ramified extension of degree $h$. 
	We choose a uniformizer $\varpi$ of $M$ such that $\varpi^h$ is a uniformizer of $K$. 
	Then the Galois group $\Gal(M/F)$ is isomorphic to $\Gamma$ \eqref{eq:presentation-Gamma} with Galois action:
	\[
	s(\varpi)=\zeta_h \varpi, \quad t(\varpi)=\varpi,
	\]
	and $s$ (resp. $t$) is the identity (resp. $q$-th Frobenius) automorphism on $K$. 

	Let $N/M$ be the abelian extension with the norm group
	\[
	\Nm(N/M)= \langle\varpi\rangle \times \fF^{\times} \times U_M^2,
	\]
	so that $\Gal(N/M)\simeq U_M^1/U_M^2$. 
	Since $\Nm(N/M)$ is preserved by $\Gamma$, the extension $N/F$ is Galois. 
	
	The choice of $\varpi$ gives an $\mathbb{F}_p$-linear isomorphism:
	\[
	U_M^1/U_M^2\xrightarrow{\sim} \fF,\quad 1+\varpi a\mapsto a \mod \varpi.
	\]
	Via the above isomorphism, the natural action of $\Gamma$ on $U_M^1/U_M^2$ induces an action on $\fF$.

	Since $\Gal(M/K)=\langle s\rangle$ has trivial invariants in $U_M^1/U_M^2$, it follows from \cite[Lemma 6.1]{GR10} that there is a splitting:
	\[\Gal(N/F)\simeq (U_M^1/U_M^2)\rtimes \Gamma \xrightarrow{\varpi} \fF^+\rtimes \Gamma.\]
\end{secnumber}

\begin{secnumber}\label{sss:construction-psi}
	Let $\varphi:W_F\to \cG$ be a simple wild parameter and we take again the notaion of \ref{sss:image-structure}(iii). 
	Then $E$ is a subextension of $M$ over $F$ that induces the quotient (\ref{sss:image-structure}(v)): 
	\[\Gamma\simeq \Gal(M/F)\to D/D_1\simeq \Gal(E/F),\quad s\mapsto \sigma, ~t\mapsto \Fr. \]
	Let $\fF'$ be the residue field of $E$. Then the trace map
	\[
	\Tr:\fF\to \fF'
	\]
	is compatible with the natural actions of $\Gamma$. 

	As in \ref{sss:Gal-extension}, we define a purely wildly ramified extension $N'$ of $E$ with the norm group 
	\[
	\Nm(N'/E)= \langle \varpi\rangle\times \fF'^\times\times U_{E}^2.
	\]

	The simple wild parameter $\varphi:W_F\to D$ induces an isomorphism $\Gal(E/F)\simeq D/D_1$. 
	By local class field theory, $\varphi$ factors through $\Gal(N'/F)$ constructed above:
	\[
	\varphi:\Gal(N'/F)\simeq \fF'\rtimes \Gal(E/F)\to D_1\rtimes D/D_1.
	\]
	By restricting $\varphi$ to $\fF'$ and composing with $\Tr$, we obtain an $\mathbb{F}_p[\Gamma]$-linear map:
	\begin{equation} \label{eq:psi-Fp}
	\psi: \fF^+\to D_1.
\end{equation}
	
	Note that $D_1$ is isomorphic to $\mathbb{F}_p[\zeta_h]^+$ as abelian group. 
	Since $\psi$ is equivariant with respect to the action of $s$, $\psi$ is $\mathbb{F}_p[\zeta_h]$-linear. 
	In particular, we deduce that $\psi=\Tr_{\mathfrak{F}/\mathbb{F}_p(\zeta_h)}(a\cdot -)$ for some element $a\in \mathfrak{F}$. 
\end{secnumber}

\begin{prop} \label{p:map-Xi} 
	\textnormal{(i)} The restriction to the wild inertia subgroup induces a natural map 
	\begin{equation} \label{eq:parameters-Hom}
	\Xi:\{\textnormal{simple wild parameters}\}/ \sim \quad \to \Hom_{\mathbb{F}_p[\Gamma]}(\mathfrak{F}^+,D_1). 
	\end{equation}
	Its image consists of non-trivial homomorphisms and has cardinality $q-1$.

	\textnormal{(ii)} The map $\Xi$ induces a bijection
	\begin{equation} \label{eq:local-Xi}
	\{\varphi|_{I_F}:I_F\to \cG \textnormal{~with simple wild parameters~} \varphi\}/ \sim ~ \xrightarrow{\sim} ~ \Hom_{\mathbb{F}_p[\Gamma]}(\mathfrak{F}^+,D_1)-\{0\}. 
\end{equation}
\end{prop}

\begin{proof}
	(i) By \Cref{l:D0-unique}, the map is well-defined. 
	Next we show the surjectivity. A wild parameter $\varphi$ induces a surjection $\eta:\Gamma\to D/D_1$ and an $\mathbb{F}_p[\Gamma]$-equivariant map $\psi:\fF^+\to D_1$. 
	Let $\psi':\fF^+\to D_1$ be another non-trivial $\mathbb{F}_p[\Gamma]$-equivariant map. 
	Then the homomorphism
	\[
	\varphi'=\psi'\rtimes \eta: \fF\rtimes \Gamma\simeq \Gal(N/F)\to D_1\rtimes D/D_1\simeq D
	\]
	is still a simple wild parameter. 

	By discussion in \ref{sss:construction-psi}, any other $\mathbb{F}_p[\Gamma]$-equivariant map $\psi':\mathfrak{F}^+\to D_1$ is $\mathbb{F}_p[\zeta_h]$-linear and is given by 
	\begin{equation} \label{eq:different-psi}
	\psi'(-):=\psi(a \cdot -):\fF \to D_1,\quad \textnormal{for }a\in \mathfrak{F}^{\times}.
\end{equation}
	Moreover, by considering the $t$-equivariant property of $\psi'$, we deduce that $a\in \mathbb{F}_q^{\times}$. 
	This shows that the cardinality of $\Hom_{\mathbb{F}_p[\Gamma]}(\mathfrak{F}^+,D_1)$ equals $q$. 

	(ii) It suffices to show that the map is injective. 
	Let $\varphi,\varphi':I_F\to \cG$ be two homomorphisms in the left hand side. We may assume that they have the same image $D_0$ by \Cref{l:D0-unique} and then a tame generator is sent to $\sigma$ by $\varphi$ and $\varphi'$. 
	If $\varphi|_{I_F^+}=\varphi'|_{I_F^+}$ in the right hand side, 
	then we conclude that $\varphi=\varphi'$ as representations of $I_F$ to $\cG$. 
\end{proof}

\begin{secnumber} \label{sss:modification-Frob}
	\textit{Proof of \Cref{t:number-simple-wild}}. 
	Let $\psi:\mathfrak{F}^+\to D_1$ be a non-trivial $\Gamma$-equivariant homomorphism. It suffcies to show that $\sharp \Xi^{-1}(\psi)=Z(q)$. 

	Let $\varphi, \varphi':W_F\to \cG$ be two simple wild parameters in $\Xi^{-1}(\psi)$. 
	Then $\varphi|_{I_F}=\varphi'|_{I_F}$ as representations $I_F\to \cG$. 
	There exists an element $x\in C_{\cG}(\varphi(I_F))=\cT^{\sigma}$ such that $\varphi'(\Fr)=x\varphi(\Fr)$. 

	Suppose there exists an element $g\in \cG$ such that $\varphi'=\Ad_g(\varphi)$. 
	Then $g\in C_{\cG}(\varphi(I_F))=\cT^{\sigma}$. 
	Since the action of $\varphi(\Fr)$ on $\cT^{\sigma}$ is the $q$-th power map \cite[\S~9.5]{GR10}, we deduce that 
	\[
	x\varphi(\Fr)=g\varphi(\Fr)g^{-1}=gg^{-q}\varphi(\Fr) \Rightarrow x=g^{1-q}. 
	\]
	
	Hence the cardinality of $\Xi^{-1}(\psi)$ equals to the cardinality of $\cT^{\sigma}/\sim$, where the equivalent relationship $x\sim y$ is given by $x=g^{1-q}y$ for some $g\in \cT^{\sigma}$. 
	The latter equals to the number of orbits of the transitive action 
	\[
	\cT^{\sigma}\times \cT^{\sigma}\to \cT^{\sigma},\quad (g,x)\mapsto g^{1-q} x.
	\]
	The stabilizer of this action is the subgroup $\cT^{\sigma,q=1}$ of elements whose $q$-th power is trivial. 
	Hence the number of orbits equals to $\sharp \cT^{\sigma,q=1}=Z(q)$ \cite[\S~9.5]{GR10}. 
	This finishes the proof. 
\end{secnumber}
\begin{secnumber}

	Finally, we show that different simple wild parameters can be related by a scalar action. 
	We first prove a general statement. 

	Let $K = \cF(\!(t)\!)$ be a local field. 
	Let $K^{\mathrm{sep}}$ be a fixed separable closure of $K$, and let $G_K = \operatorname{Gal}(K^{\mathrm{sep}}/K)$. 
	Let $\rho:G_K\to \GL(V)$ be a representation factoring through the inertial subgroup $P\subset I_K$ with an abelian image. 

	An element $a\in \cF$ induces an automorphism $\sigma_a:K\to K,t\mapsto a^{-1}t$. 
	We fix an extension of $\sigma_a$ to an automorphism of the separable closure, denoted $\tilde{\sigma}_a \in \Aut(K^{\mathrm{sep}})$. 
	Then $\sigma_a$ induces a map $\sigma_{a,*}: G_K\to G_K$:
\[
\sigma_{a,*}(g) = \tilde{\sigma}_a^{-1} \circ g \circ \tilde{\sigma}_a. 
\]

	Let $\rho_a$ be the composition of $\rho$ with the automorphism $\sigma_{a,*}$. 
\end{secnumber}

\begin{prop} \label{p:scalar-action}
	Suppose $\rho$ factors through the $r$-th graded piece of the upper ramification filtration $P^{(r)} / P^{(r+1)}$. 
	Then we have $\rho_a(g)=\rho(a^r\cdot g)$ for $g\in P^{(r)}$ via $P^{(r)} / P^{(r+1)}\simeq \cF$.
\end{prop}
\begin{rem}\label{r:transitive action}
	Let $\varphi:G_F\to \cG$ be a simple wild parameter and $a\in \mathbb{F}_q^{\times}$. 
	By \eqref{eq:different-psi} and \Cref{p:scalar-action}, $\varphi_a$ is still a simple wild parameter and the bijection \eqref{eq:local-Xi} is compatible with the natural action of $a\in \mathbb{F}_q^{\times}$. 
	In particular, this action is transitive on both sides of \eqref{eq:local-Xi}. 
\end{rem}
\begin{proof}
	The local Artin reciprocity map $\theta_K \colon K^\times \to G_K^{\mathrm{ab}}$ is functorial with respect to continuous field isomorphisms. 
	By \cite[XIII \S~4 proposition 11]{Serre-LF}, we have:
\[
    \tilde{\sigma}_a^{-1} \circ \theta_K(u) \circ \tilde{\sigma}_a = \theta_K(\sigma_a^{-1}(u)) \quad \text{for any } u \in K^\times.
\]

	By local class field theory, $\theta_K$ maps the filtration of principal units $U_r = 1 + \mm_K^r = 1 + t^r \cF[\![t]\!]$ onto the upper ramification subgroups $P^{(r)}$. This induces canonical group isomorphisms on the successive quotients:
\[
\cF\xrightarrow{\sim} U_r / U_{r+1} \xrightarrow{\sim} P^{(r)} / P^{(r+1)}.
\]
The first isomorphism is given by $c\mapsto 1+c t^r$. 
The action of the inverse automorphism $\sigma_a^{-1}$ on the unit $u = 1 + c t^r \in U_r$ is given by $\sigma_a^{-1}(1+c t^r)=1+ c a^r t^r$. 

Let $g \in P^{(r)} / P^{(r+1)}$ be an element corresponding via $\theta_K$ to $u = 1 + c t^r \in U_r / U_{r+1}$. 
The pulled-back character $\rho_a$ evaluates $g$ as:
\[
    \rho_a(g) = \rho_a\big(\theta_K(1 + ct^r)\big) = \rho\big(\theta_K(\sigma_a^{-1}(1 + ct^r))\big) = \rho\big(\theta_K(1 + c a^r t^r)\big).
\]
This shows that $\rho_a(g)=\rho(a^r\cdot g)$. 
\end{proof}


\subsection{Epipelagic Langlands parameters: mixed characteristic case}

In this subsection, we first briefly review the Deligne--Kazhdan correspondence \cite{Del84,Gan15} and its conjectural relationship with the local Langlands correspondence of Genestier--Lafforgue \cite{GL} and Fargues--Scholze \cite{FS}. 
Then we explain how this conjecture is applicable to the epipelagic L-parameters in mixed characteristic. 

Let $F$ be a local field of characteristic $p$ with $\mathcal{O}_F$ its ring of integers, $\mathfrak{m}_F$ its maximal ideal, $\pi$ its uniformizer and the finite residue field $k=\mathcal{O}_F/\mathfrak{m}_F$. 
Let $F^{\sharp}$ be another local field of mixed characteristic with $\mathcal{O}_{F^{\sharp}}$ its ring of integers and $\mathfrak{m}_{F^{\sharp}}$ its maximal ideal. 

\begin{definition}
	Let $m$ be a positive integer. We say $F$ and $F^{\sharp}$ are \textit{$m$-close} if there is a ring isomorphism 
	$\mathcal{O}_F/\mathfrak{m}_F^m \xrightarrow{\sim} \mathcal{O}_{F^{\sharp}}/\mathfrak{m}_{F^{\sharp}}^m$.
\end{definition}

For example, the fields $\mathbb{F}_p(\!(t)\!)$ and $\mathbb{Q}_p(p^{1/m})$ are $m$-close. 

\begin{secnumber}
	Let $\Gal_F$ be the absolute Galois group of $F$, $I_F$ the inertia subgroup of $\Gal_F$, equipped with the higher ramification subgroups with upper numbering $\{I_F^r\}_{r\in \mathbb{Q}_{\ge 0}}$. 
	Suppose $F,F^{\sharp}$ are $m$-close. 
	Deligne \cite{Del84} showed that there exists an isomorphism 
	\begin{equation}
		\Del_m: \Gal_F/I_F^m \xrightarrow{\sim} \Gal_{F^{\sharp}}/I_{F^{\sharp}}^m,
		\label{eq:Del-m}
	\end{equation}
	which is unique up to inner automorphisms (see \cite[Equation 3.5.1]{Del84}). 
	In particular, $\Del_m$ induces a bijection between isomorphism classes of representations of $\Gal_F$ trivial on $I_F^m$ and those of $\Gal_{F^{\sharp}}$ trivial on $I_{F^{\sharp}}^m$. 
	The above properties hold when $\Gal_F$ is replaced by the Weil group $W_F$. 
\end{secnumber}

\begin{secnumber}
	Let $\sG$ be a split, connected reductive group defined over $\mathbb{Z}$, $G$ its base change to $F$ or $F^{\sharp}$ and $\sB$ a Borel subgroup of $\sG$. 
	Let $\ell$ be a prime different from $p$ and $\cG$ the Langlands dual group of $\sG$ over $\bQl$. 
	We denote by $\mathcal{G}^{\sss}(G,F)$ the set of equivalence classes of semisimple $\ell$-adic homomorphisms:
	\[
	\rho: W_F\to \cG.
	\]
	Let $\mathcal{G}^{\sss,m}(G,F)\subset \mathcal{G}^{\sss}(G,F)$ denote the subset of parameters $\phi$ of depth at most $m$, i.e. $\phi$ is trivial on the subgroup $I_F^m$. 

	Then Deligne's isomorphism $\Del_m$ induces a bijection:
	\begin{equation}
		\mathcal{G}^{\sss,m}(G,F)\xrightarrow{\sim} \mathcal{G}^{\sss,m}(G,F^{\sharp}). 
		\label{eq:Deligne-bij}
	\end{equation}
\end{secnumber}

\begin{secnumber}
	Let $\bI$ be the standard Iwahori subgroup of $G$, defined as the inverse image under $\sG(\mathcal{O}_F)\to \sG(k)$ of $\sB(k)$. 
	There is a smooth affine group scheme $I$ over $\mathcal{O}_F$ with generic fiber $G$ such that $I(\mathcal{O}_F)=\bI$. 
	For every positive integer $m$, we define $\bI_m:=\Ker(I(\mathcal{O}_F)\to I(\mathcal{O}_F/t^m))$. 
	Explicitly, we have $\bI_{m}=\{U_{\alpha,t^m},t^mT(\mathcal{O}_F),U_{-\alpha,t^{m+1}}|\alpha\in \Phi^+\}$. 
	Let $H(G(F),\bI_m)$ denote the Hecke algebra of $\bI_m$-biinvariant functions on $G(F)$ with coefficients in $\bQl$. 

	We denote by $\mathcal{A}(G,F)$ the set of equivalence classes of irreducible admissible representations of $G(F)$. 
	Let $\mathcal{A}^m(G,F)\subset \mathcal{A}(G,F)$  denote
the subset of equivalence classes of irreducible admissible representations generated by vectors fixed
under $\bI_m$. 
	Ganapathy refined a theorem of Kazhdan:
\end{secnumber}

\begin{theorem}[Ganapathy--Kazhdan, \cite{Gan15}]
	If $F,F^{\sharp}$ are $m$-close, then there is a natural isomorphism
	\[
	H(G(F),\bI_m)\xrightarrow{\sim} H(G(F^{\sharp}),\bI_m)
	\]
	of finitely-generated $\bQl$-algebras. Moreover, there is a bijection
	\begin{equation}
		\mathcal{A}^m(G,F)\simeq \mathcal{A}^m(G,F^{\sharp})
		\label{eq:Ganapathy-Kazhdan}
	\end{equation}
	with the property that, if $\pi\in \mathcal{A}^m(G,F)$ corresponds to $\pi^{\sharp}\in\mathcal{A}^m(G,F^{\sharp})$, then the invariant subspaces $\pi^{\bI_m}$ and $\pi^{\sharp, \bI_m}$ are isomorphic as modules with respect to the above isomorphism.  
\end{theorem}

\begin{secnumber}
	Over the local field $F$ of characteristic $p$ (resp. $F^{\sharp}$ of mixed characteristic), Genestier--Lafforgue (resp. Fargues--Scholze) constructed a semisimple parametrization of $\mathcal{A}(G,F^{?})$ for $?\in \{\empty,\sharp\}$:
	\[
	\mathcal{L}^{\sss}:\mathcal{A}(G,F^{?}) \to \mathcal{G}^{\sss}(G,F^{?}).
	\]
\end{secnumber}

\begin{conjecture}[\cite{GHS} Conjecture 11.7] \label{c:compatibility-DK-FS}
	For any positive integer $m$, the following diagram commutes:
	\[
	\xymatrix{
	\mathcal{A}^m(G,F) \ar[rr]^{\mathcal{L}^{\sss}}\ar[d]_{\eqref{eq:Ganapathy-Kazhdan}} && \mathcal{G}^{\sss,m}(G,F) \ar[d]^{\eqref{eq:Deligne-bij}}\\
	\mathcal{A}^m(G,F^{\sharp}) \ar[rr]^{\mathcal{L}^{\sss}} && \mathcal{G}^{\sss,m}(G,F^{\sharp}) 
	}
	\]
\end{conjecture}

Li-Huerta \cite{LH23} showed that, over local fields of positive characteristic, the constructions of Fargues--Scholze and of Genestier--Lafforgue are compatible. 
Thus, \Cref{c:compatibility-DK-FS} also asserts that the two Fargues--Scholze parametrizations are compatible with the Deligne-Kazhdan correspondence.
Recently, Li-Huerta \cite{LH24} proved such a compatibility in a different setting. 

\begin{secnumber}
Now, we go back to the setting of epipelagic representations (\S~\ref{sss:parahoric}). 
When the Kac coordinate $s_0=1$, it is easy to check that $\bI_1\subset \bP(2)$. 
By \cite[Theorem B.9]{Mis25}, Ganapathy--Kazhdan's bijection \eqref{eq:Ganapathy-Kazhdan} sends epipelagic representations over $F$ to epipelagic representations over $F^{\sharp}$ for large enough $m$. 

Under the expectation \Cref{r:companion-local-global} and \Cref{c:compatibility-DK-FS}, one can deduce the properties of \Cref{c:Swan-conductor} for the epipelagic Langlands parameters in mixed characteristics, defined by Fargues-Scholze, from the equal characteristics case.
\end{secnumber}

\subsection{Local monodromy of $\Ai^{\rig}_{\cG}$}
	We keep the notation of \S~\ref{ss:Ai-rig} and we show an analogue of \Cref{t:subsidiary-radii} for Airy local systems, which verifies \cite[Conjecture 29(i)]{JKY}. 
	We also prove \cite[Conjecture 29(ii)]{JKY} for a special Airy local system, which implies its cohomological rigidity. 
    
	\begin{prop} \label{p:Airy-Swan}
		The multi-set of the $p$-adic slopes of $\Ai_{\cG}^{\rig}(\chi)(\cfg)$ (resp. the slopes of $\Ai_{\cG}^{\ell}(\chi)(\cfg)$) 
		consists of $\sharp \Phi$ copies of $1+\frac{1}{h}$ and $r$ copies of $0$. 
		In particular, its $p$-adic irregularity (resp. Swan conductor) equals to $r(h+1)$. 
	\end{prop}
	\begin{proof}
		By \Cref{p:companion-WD-rep}, it suffices to prove the proposition for $\Ai_{\cG}^{\rig}(\chi)(\cfg)$. 
		Let $\mathcal{E}$ be the restriction of $\Ai(\underline{a},-\pi)$ \eqref{eq:Ai-conn} at infinity, which is a connection on the ring $\mathcal{A}_t$ where $t=x^{-1}$. 
		Recall that $a$ is a $p$-adic unit. 
		To simplify the notations, we may assume $a=1$ in \eqref{eq:Ai-conn}. 
		Recall $s=t^{1/2h}$. Apply the gauge transformation of $(-\pi^{-1}s^2)^{\check{\rho}}$ to $[2h]^*\mathcal{E}$, 
		we obtain 
		\begin{equation} \label{eq:Ai-Y}
		Y:= d+2h\pi(p_{-1}+p_r)s^{-2-2h}\frac{ds}{s}+\sum_{i=1}^{r-1} 2h\pi a_i p_is^{-2(h-d_i+1)}\frac{ds}{s}-\check{\rho}\frac{ds}{s},	
	\end{equation}
	denoted by $Y$. 
	Since $p_{-1}+p_r$ is regular semisimple, 
	by Lemma \ref{l:isoclinic gauge}, after enlarging $K$, 
	there exists an element $Q\in \cG (K\llbracket s \rrbracket)$ such that $Y$ can be gauge transformed to be the canonical form $Z$ of $Y$ via $Q$: 
	\begin{equation} \label{eq:Ai-Z}
	Z:= d+h\pi(p_{-1}+p_r)\frac{ds}{s^{2+h}}+\sum_{j\le h+1} Z_j \frac{ds}{s^j}, 
\end{equation}
	where $Z_j$ lies in the centralizer of $X=p_{-1}+p_r$ in $\cfg$. 

	By Clark's theorem \ref{t:Clark}, $Q$ has a positive $p$-adically convergent radius as in \Cref{p:positive-convergent-radius}. 
	Moreover, $Y$ has a Frobenius structure and is therefore solvable. 
	Since $p>h$, we can show that the subsidiary radii functions $F_i(Y,r)$ of $Y$ are linear functions on $r\in (0,1)$ as in the proof of \Cref{t:subsidiary-radii}, \Cref{l:subsidiary}. 
	Hence, the multi-set of $p$-adic slopes of $Y$ coincides with the multi-set of formal slopes of $Y$, viewed as a formal connection. 
	The latter consists of $\sharp \Phi$ copies of $1+\frac{1}{h}$ and $r$ copies of $0$, and the proposition follows.  
	\end{proof}
    
    Now we concentrate on a special class of Airy local systems.
    Recall the construction of $\Ai^{\rig}_{\cG}(\chi)$ in \ref{sss:Ai-rig}.
    The character $\chi:J\rightarrow\mathbb{A}^1$ extends $\phi$ onto $S(1)$.
    The extension is determined by an induced character of
    $S(1)/S(1+h/2)$ when $h$ is even, 
    or of $S(1)/S(1)\cap P(1+n)$ when $h=2n+1$ is odd.
    We call $\Ai^{\rig}_{\cG}(\chi)$ \emph{simple}
    if $\chi$ induces the trivial character on 
    $S(1)/S(1+h/2)$(resp. $S(1)/S(1)\cap P(1+n)$).
    
    \begin{lemma}\label{l:simple Ai}
    	A simple Airy local system $\Ai^{\rig}_{\cG}(\lambda\chi)$, 
    	$\lambda\in K-\{0\}$ is isomorphic to
    	\begin{equation} \label{eq:naive Ai-conn}
    		\Ai(\underline{a},\lambda)=d+(p_{-1}+a\lambda^h x p_r)dx,\quad a\in K^\times.
    	\end{equation}
    \end{lemma} 
    \begin{proof}
    	In view of Theorem \ref{t:Ai connection}.(i),
    	it suffices to show that $a_i=0$ in \eqref{eq:Ai-conn}.
    	By \cite[Theorem 36, Proposition 38]{Yi25},
    	$a_i=v_{i,2d_i-1}=\chi(S_{i,2d_i-1})$,
    	where $\overline{S_{i,2d_i-1}}$ is the image of a Segal-Sugawara operator
    	in the enveloping algebra of $\mathrm{Lie}(J/I(2+h))$ 
    	on which $\chi$ can evaluate.
    	
    	By the proof of Proposition 29 in the \emph{loc. cit.} 
    	where we let $(N,m)=(h+1,h)$,
    	$\overline{S_{i,2d_i-1}}$ can be written as a sum of tensor products
    	$X_{b_1}\otimes\cdots\otimes X_{b_{d_i}}$
    	where $X_{b_r}\in\mathrm{Lie}(J/I(2+h))$, $\sum_{r=1}^{d_i}b_r=hd_i$, 
    	and $b_1\leq\cdots\leq b_{d_i}\leq h+1$.
    	Also, using the splitting of $\mathrm{Lie}(I(1+h)/I(2+h))$
    	inside $\mathrm{Lie}(J/I(2+h))$ with respect to affine root subspaces,
    	$X_{b_r}$ has nonzero projection to $\mathrm{Lie}(I(1+h)/I(2+h))$
    	only if $b_r=1+h$.
    	Thus for $\chi(X_{b_1}\otimes\cdots\otimes X_{b_{d_i}})\neq 0$,
    	we must have $b_r=1+h$ for all $r$.
    	Then $\sum_{r=1}^{d_i}b_r=(1+h)d_i>hd_i$, impossible.
    	We obtain that $a_i=\chi(S_{i,2d_i-1})=0$.
    \end{proof}
    
    \begin{prop} \label{p:Airy-coh-rig}
    	Let $\Ai^{\rig}_{\cG}(-\pi\chi)$ be a simple Airy local system.
	Let $\mathscr{E}=\Ai^{\rig}_{\cG}(-\pi\chi)|_{\infty}$ and $(\rho:I_F\to \cG(\overline{K}),N)$ be the representation associated to $\mathscr{E}$ \eqref{eq:MCFR}. 

		\textnormal{(i)} $\mathscr{E}$ is a $p$-adic isoclinic connection of slope $1+\frac{1}{h}$ in the sense of \ref{d:p-adic-isoclinic}. 

		\textnormal{(ii)} 
	The nilpotent operator $N$ vanishes. 
	Let $\cT_X$ be the centralizer $C_{\cG}(X)$ of $X$ in $\cG$, which is a maximal torus. 
	The representation $\rho$ fits into a commutative diagram
	\[
	\xymatrix{
	1\ar[r] & I^+_F \ar[r] \ar[d] & I_F \ar[r] \ar[d] & I_F^t \ar[r] \ar[d] & 1 \\
	1\ar[r] & \cT_X \ar[r]& N_{\cG}(\cT_X) \ar[r] & \bW \ar[r] & 1
	}
	\]
	such that a generator of the tame inertia $I_F^t$ is sent to a Coxeter element of the Weyl group $\bW$. 
	In particular, we have $\cfg^{I_F}=0$.
	
	\textnormal{(iii)} The local system $\Ai^{\rig}_{\cG}(\chi)$ (resp. $\Ai_{\cG}^{\ell}(\chi)(\cfg)$) is cohomologically rigid.
	\end{prop}

	\begin{proof}
		(i) By Lemma \ref{l:simple Ai}, 
		the underlying connection $E_{\cG}(\underline{a},-\pi)$ of $\Ai^{\rig}_{\cG}(-\pi\chi)$ satisfies $a_i=0$ for $i=1,\dots,r-1$. Let $\mathcal{E}$ be the formal $\cG$-connection defined by $E_{\cG}(\underline{a},-\pi)$ at infinity. 
		As in the proof of \Cref{p:Airy-Swan}, $[2m]^*\mathcal{E}$ is gauge equivalent to \eqref{eq:Ai-Y}
		\[
		Y=d+h\pi(p_{-1}+p_r) s^{-2-2h}\frac{ds}{s}-\check{\rho}\frac{ds}{s}.
		\]
		Then $Y$ can be gauge transformed to its canonical form $Z$ \eqref{eq:Ai-Z} via $Q\in \cG (K\llbracket s \rrbracket)$:
	\[
	Z= d+h\pi(p_{-1}+p_r)s^{-2-2h}\frac{ds}{s}. 
	\]
	Since $p>h$, $Z$ has a Frobenius structure by \Cref{p:Ld-solvable-Frobenius} and \Cref{l:eigenvalues-X}. 
	Then by the same argument of \Cref{t:p-adic-convergence-Q}, we can show that $Q\in \cG(\mathcal{A}_{K,s})$. 
	Then assertion (i) follows. 

	Assertion (ii) follows from \Cref{t:L-parameter}. 

	Assertion (iii) follows from assertion (ii), \Cref{p:Airy-Swan}
	and the discussion in \cite[6.3]{JKY}.
	\end{proof}

	\begin{rem}
		We expect \Cref{p:Airy-coh-rig} holds for all Airy local systems considered in \ref{sss:Ai-rig}. 
		However, we don't have an explicit description of minor terms $Z_j$ in \eqref{eq:Ai-Z} and can't prove \Cref{p:Airy-coh-rig} in the general setting. 
	\end{rem}

\subsection{Global monodromy of $\Ai^{\rig}_{\cG}$}
	We keep the assumption of \Cref{p:Airy-coh-rig}. 
We apply the local monodromy result in the above subsection to study the global monodromy of Airy local system. 

\begin{secnumber} \label{sss:Airy-Gdiff}
	The underlying connection \eqref{eq:naive Ai-conn} of $\Ai^{\rig}_{\cG}(\chi)$ satisfies $a_i=0$ for $i=1,\dots,r-1$:
	\[
	\Ai(\underline{a},\lambda)=d+(p_{-1}+(-\pi)^h x p_r)dx. 
	\]

	Let $G_{\diff}$ be the differential Galois group of the above $\cG$-connection. 
	In \cite{KS}, the authors showed that $G_{\diff}$ is a connected reductive subgroup of $\cG$ of \textit{maximal degree}, i.e. it has a fundamental degree equal to the Coxeter number of $\cG$. 
	More precisely, if we denote by $\Sigma$ the outer automorphism group of $\cG$ and by $\Out(\check{\gg})$ the outer automorphism group of $\check{\gg}$, then $G_{\diff}$ is given by \cite[Theorem 12]{KS}:
	\begin{itemize}
		\item $G_{\diff}\xrightarrow{\sim} \cG^{\Sigma,\circ}$, if $\cG$ is not type $A_{2n}$ ($n\ge 2$) or $B_{3}$ or $D_{2n}$ ($n\ge 2$) with $\Sigma\neq \Out(\check{\gg})$.  
	        \item $G_{\diff}=\cG$, if $\cG$ is of type $A_{2n}$, 
	        \item $G_{\diff}\xrightarrow{\sim} G_2$, if $\cG$ is of type $B_{3}$ or of type $D_4$. 
		\item $G_{\diff}\xrightarrow{\sim}\textnormal{Spin}_{4n-1}$ if $\cG$ is of type $D_{2n}$ with $\Sigma\simeq \{1\}$ ($n\ge 3$). 
	\end{itemize}
	Note that the above result coincides with the differential Galois group of the Bessel $\cG$-connection introduced by Frenkel--Gross \cite{FG09}. 

	On the other hand, let $H$ be the geometric monodromy group of $\Ai^{\rig}_{\cG}(\chi)$ in $\cG$. 
	As explained in \cite[\S~4.5.1]{XZ22}(replacing Bessel $\cG$-connection by Airy $\cG$-connection), there exists a natural embedding:
	\begin{equation}
		H\to G_{\diff}.
		\label{eq:embedding-mono}
	\end{equation}
\end{secnumber}

\begin{secnumber} \label{sss-min-dim}
	Let $\cfg$ be the Lie algebra of $\cG$. We take $n$ to be the minimal dimension of faithful representations among almost simple groups with Lie algebra $\cfg$. 
	For each simple type (of rank $\ell$), the minimal dimension is:
\[
\begin{array}{c|c|l}
A_\ell\;(\mathfrak{sl}_{\ell+1}) & \ell+1 & \text{standard representation} \\
B_\ell\;(\mathfrak{so}_{2\ell+1}),\ \ell\ge 2 & 2\ell+1 & \text{vector representation (faithful on }\SO_{2\ell+1}\text{)} \\
C_\ell\;(\mathfrak{sp}_{2\ell}),\ \ell\ge 2 & 2\ell & \text{standard representation} \\
D_\ell\;(\mathfrak{so}_{2\ell}),\ \ell\ge 4 & 2\ell & \text{vector representation (faithful on }\SO_{2\ell}\text{)} \\
G_2 & 7 & \text{minimal representation} \\
F_4 & 26 & \text{minimal representation} \\
E_6 & 27 & \text{minimal representation} \\
E_7 & 56 & \text{minimal representation} \\
E_8 & 248 & \text{adjoint representation} \\
\end{array}
\]	
Note that $n$ is larger or equal to the Coxeter number $h$ and is larger than the rank $\ell$. 
\end{secnumber}

\begin{theorem}
	Keep the assumption of \Cref{p:Airy-coh-rig} and $p>2n+1$. 
	Then the natural morphism $H\to G_{\diff}$ \eqref{eq:embedding-mono} is an isomorphism. 
	\label{t:monodromy-Ai}
\end{theorem}

\begin{rem}
	The geometric monodromy of a family of rank $n$ Airy local systems was calculated by Katz and Šuch \cite{Katz87,Such}, including the above theorem in the $\GL_n$-case. 
	Our approach is inspired by Katz's paper. 
\end{rem}

\begin{secnumber} \label{sss:basic-facts}
	We assume that $\cG$ admits a faithful representation of minimal dimension $n$ as in \ref{sss-min-dim}. 
	We list some basic properties of the geometric monodromy group $H$:

	\begin{itemize}
		\item For $V\in \Rep(\cG)$, $\Ai^{\rig}_{\cG}(\chi)(V)$ is a pure local system and is geometrically semisimple. Hence, $H^{\circ}$ is a semisimple subgroup of $\cG$ \cite{Cr92}. 

		\item The group of connected components $\pi_0(H)$ is a finite quotient of $\pi_1(\mathbb{A}^1_{\overline{k}})$ \cite[(7.8)]{Dri18} and is therefore a \textit{quasi-$p$-group}, i.e. it has no non-trivial finite quotients of order prime to $p$ \cite[Exposé XIII, Corollaire 2.12]{SGAI}. 

		\item In view of the local monodromy at infinity in \Cref{p:Airy-coh-rig}, $C_{\cG}(H)$ is a subgroup of the finite group $\cG^{I_{\infty}}=\cT^{\sigma}$. 
	By the assumption on $p$, the latter has order prime to $p$. 
	\end{itemize}

	In the following, we show a connectedness criterion for a subgroup $H$ of $\cG$
\end{secnumber}

\begin{prop} \label{p:connected-H}
	Let $\cG$ be an almost simple group with a faithful representation $\cG\to \GL_n$ and $H\subset \cG$ an algebraic subgroup. 
	Let $p>2n+1$ be a prime. 
	Assume that
	(i) $\pi_0(H)$ is a quasi-$p$-group, (ii) $H^{\circ}$ is semisimple, (iii) $C_{\cG}(H)$ is finite of order prime to $p$. 
	Then $H$ is connected. 
\end{prop}

\begin{secnumber}
	We need Jordan's theorem, which says that there exists a function $c(n)$ on $n$ such that if $\Gamma$ is a finite subgroup of $\GL_n$, then there exists a normal abelian subgroup $\Gamma_0$ of $\Gamma$ such that the quotient $\Gamma/\Gamma_0$ has order $\le c(n)$. 
	Feit--Thompson's sharpening on Jordan's theorem \cite{FT} says that if $p>2n+1$, then the Sylow $p$-subgroup of $\Gamma$ is an abelian normal subgroup of $\Gamma$. 
\end{secnumber}

\begin{proof}[Proof of \Cref{p:connected-H}]
	(i) The conjugate action of $H$ on $H^{\circ}$ induces a homomorphism
	\[
	\pi_0(H) \to \Out(H^{\circ}). 
	\]
	All prime divisors of $|\Out(H^{\circ})|$ are $\le \max\{3, \rank(H)\}\le \max\{3,\rank(\cG)\}$. 
	Then the order of $\Out(H^{\circ})$ is prime to $p$. 
	Since $\pi_0(H)$ is a quasi-$p$-group, the above homomorphism is trivial. 
	
	Let $L=C_{\cG}(H^{\circ})$. 
	Every element $h\in H$ acts on $H^{\circ}$ by inner automorphism $\Ad_{s(h)}$ for some $s(h)\in H^{\circ}$. 
	Then $\ell(h)=hs(h)^{-1}$ is an element of $L$. 
	This construction defines an injective homomorphism
	\[
	\pi_0(H)\hookrightarrow \Delta:=L/Z(H^{\circ}), \quad h\mapsto \ell(h).  
	\]
	We write $K$ to be the image of $\pi_0(H)$ in $\Delta$. Note that $Z(H^{\circ})$ is a finite group whose order has prime divisors smaller than $p$. 

	(ii) Next we show that $K$ is abelian. 
	Let $\widehat{K}\subset L$ be the preimage of $K$ under $\pi:L\to \Delta$. Then $\widehat{K}$ is also finite. 
	We take a faithful representation $\cG\to \GL_n$ as in \ref{sss-min-dim}. 
	By Feit--Thompson's result, a Sylow $p$-subgroup of $\widehat{K}$ is abelian and normal. 
	Pushing down to $K$, the same holds for $K$. Since $K$ is a quasi-$p$-group, we deduce that $K$ is an abelian $p$-group. 

	(iii) We show that $C_{\cG}(H)=C_L(\widehat{K})$. 
	We first note that $H$ is generated by $H^{\circ}$ and $\widehat{K}$, i.e. $H=\widehat{K}\cdot H^{\circ}$. 
	For $g\in C_{\cG}(H)$, $g$ centralizes $H^\circ\subset H$, hence $g\in C_{\cG}(H^\circ)=L$. 
	Inside $L$, every element centralizes $H^\circ$, so $g\in L$ centralizes $H= \widehat K \cdot H^\circ$
if and only if $g$ centralizes $\widehat K$. Then we have $C_{\cG}(H)= C_L(\widehat K)$.

	(iv) Consider the preimage of the centralizer of $K$ in $\Delta$:
	\[
	M:=\pi^{-1}(C_{\Delta}(K))\subset L,\quad \pi: L\to \Delta.
	\]
	By definition of $M$, an element $\ell\in L$ lies in $M$ if and only if
$\pi(\ell)$ commutes with every element of $K$, i.e. for every $\hat k\in \widehat K$, $[\ell,\hat k]\in \ker(\pi)=Z(H^\circ)$. 
	Thus for $\ell\in M$ we may define a map
\[
\delta(\ell):K\to Z(H^\circ),\qquad
\delta(\ell)\bigl(\pi(\hat k)\bigr):=[\ell,\hat k].
\]
This map is independent of the choice of the lift of $k\in K$ in $\widehat{K}$ and is a homomorphism. 
Moreover, we obtain a homomorphism:
\[
\delta:M\to \Hom(K,Z(H^{\circ})),\quad \ell\mapsto [\ell,\hat k]. 
\]

Its kernel consists of those $\ell\in M$ commuting with every $\hat k\in\widehat K$, i.e.
$\ker(\delta)=C_L(\widehat K)$. 
Therefore $\delta$ induces an injection
\[
M/C_L(\widehat K)\hookrightarrow \Hom(K,Z(H^{\circ})).
\]

By assumption on $p$, the orders of $Z(H^{\circ})$ and of $C_L(\widehat{K})=C_{\cG}(H)$ are prime to $p$. Then $\Hom(K,Z(H^{\circ}))=1$ and we deduce that $|C_{\Delta}(K)|$ divides $|C_L(\widehat{K})|$ and is also prime to $p$. 
On the other hand, since $K$ is an abelian $p$-group, $K\subset C_{\Delta}(K)$. 
Then $K$ must be the trivial group.  
\end{proof}

\begin{proof}[Proof of \Cref{t:monodromy-Ai}]
	We first assume that $\cG$ admits a faithful representation of minimal dimension $n$ as in \ref{sss-min-dim}. 
	The general case would follow from this case and the trivial functoriality of $\Ai_{\cG}$ (see \ref{sss:trivial-functoriality-Ai}). 

	The geometric monodromy group $H$ of $\Ai^{\rig}_{\cG}(\chi)$ is connected and semisimple by \Cref{p:connected-H}.
	By a theorem of Serre and Borel, the image of the wild inertial at infinity $\rho(I_{F}^+)$ in $H$ is a finite $p$-group and is therefore contained in the normalizer $N_{H}(T_H)$ of a maximal torus $T_H$ of $H$. 
	Since $p>h$, $p$ does not divide the order of the Weyl group $N_{H}(T_H)/T_H$ and $\rho(I_F^+)$ is therefore contained in $T_H$. 
	We have inclusions in $\cG$:
	\[
		T_H\subset C_{H}(\rho(I_F^+))\subset C_{\cG}(\rho(I_F^+))=\cT_X.
	\]
	Then we deduce that $T_H=\cT_X\cap H$. 

	By \Cref{p:Airy-coh-rig}, the local monodromy at $\infty$ fits into the following diagram:
	\[
	\xymatrix{
	1\ar[r] & I^+_F \ar[r] \ar[d] & I_F \ar[r] \ar[d] & I_F^t \ar[r] \ar[d] & 1 \\
	1\ar[r] & T_H \ar[r]& N_{H}(T_H) \ar[r] & \bW_H \ar[r] & 1
	}
	\]
	such that a generator of the tame inertia $I_F^t$ is sent to an element of order $h$ in the Weyl group $\bW_H$. 
	
	Then we deduce that $h$ divides a fundamental degree of $H$. 
	Hence $H$ is a connected reductive subgroup of $\cG$ of maximal degree. 
	In view of the calculation of $G_{\diff}$ (\S~\ref{sss:Airy-Gdiff}) and the classification of reductive subgroups of maximal degree \cite[Theorem 3']{KS}, we conclude that the inclusion $H\to G_{\diff}$ is an isomorphism. 
\end{proof}

\section{Physical rigidity}
\label{s:rigidity}

In this section, we assume that $k=\mathbb{F}_q$ and $\sigma=\id_K:K\to K$. 
We show the physical rigidity of an overconvergent isocrystal on a curve, whose underlying connection is physically rigid. 
In particular, we deduce the physical rigidity of Bessel $F$-isocrystals (and that of Kloosterman sheaves) for almost simple groups as conjectured in \cite[Conjecture 7.1]{HNY}. 

\subsection{$P$-adic case}
\begin{secnumber} \label{sss:phy-rigid}
	Let $S$ be a proper closed subset of $\mathbb{P}^1_k$ and $U=\mathbb{P}^1_k-S$. 
	Let $\cG$ be a connected almost simple group split over $K$. 
	Let $\mathcal{E}$ be a $\cG$-overconvergent isocrystal on $U/\overline{K}$ (\ref{sss:def-isod}). 
	We say that $\mathcal{E}$ is \textit{physically rigid} if for any $\cG$-overconvergent isocrystal $\mathcal{E}'$ such that $\mathcal{E}|_x\simeq \mathcal{E}'|_x$ as differential modules over the Robba ring $\mathcal{R}$ for every closed point $x\in |S|$, then we have $\mathcal{E}\simeq \mathcal{E}'$. 
\end{secnumber}

\begin{theorem}
	Let $U \hookrightarrow X=\mathbb{P}^1$ be an open immersion of schemes over $\mathcal{O}_K$ with the complement $S\to X$ flat over $\mathcal{O}_K$. 
	Let $\mathcal{E}$ be a $\cG$-overconvergent isocrystal on $U_k/\overline{K}$. 
	Suppose that there exists a physically rigid $\cG$-connection $E$ on $U_{\overline{K}}$ such that $E^{\dagger}\simeq \mathcal{E}$ \eqref{eq:dagger-functor}. 
	Then $\mathcal{E}$ is physically rigid. 
	\label{t:physical-rigidity}
\end{theorem}
\begin{rem}
	Crew proved a similar result \cite[Theorem 2]{Cr17} for a rank $n$ overconvergent isocrystal which is the analytification of a rigid connection with regular singularities. 
	His proof reduces the problem to the cohomological rigidity (which is equivalent to the physical rigidity only in the $\GL_n$-case) and is based on comparing de Rham and rigid cohomology groups, which is only available in the regular singularity case. 
\end{rem}
\begin{secnumber}
	We review a local picture of the construction in \ref{sss:dagger-connection}. 
	Let $\MC(\mathcal{R}_K)$ (resp. $\MC(\mathcal{A}_{K,t})$, resp. $\MC(K (\!(t)\!) )$) be the category of free $\mathcal{R}_K$-modules of finite rank with connection (resp. free $\mathcal{A}_{K,t}$-modules of finite rank with connection of finite order poles, resp. free $K (\!(t)\!)$-modules with formal connection). 
	We have natural tensor functors, defined by extension of scalars \eqref{eq:inclusion-rings}:
	\begin{equation}
		\MC(\mathcal{A}_{K,t}) \to \MC(\mathcal{R}_K),\quad \MC(\mathcal{A}_{K,t})\to \MC(K (\!(t)\!)).
		\label{eq:MC-functors}
	\end{equation}
\end{secnumber}

\begin{proof}[Proof of \Cref{t:physical-rigidity}]
	After enlarging $K$, we may assume that $S$ is isomorphic to a finite disjoint union $\sqcup \Spec(\mathcal{O}_K)$. 

	Let $\mathcal{L}$ be a $\cG$-overconvergent isocrystal on $U_k$ whose local monodromy at each point $x\in S_k$ is isomorphic to that of $\mathcal{E}$. 
	After enlarging $K$, we may assume that both $\mathcal{L}, \mathcal{E}$ and the local isomorphisms between them are defined over $K$. 

	For $x\in S_K$, the restriction $E|_x$, considered as an object of $\MC(K (\!(t_x)\!))$ with a local coordinate $t_x$ at $x$, is algebraic and therefore comes from an object of $\MC(\mathcal{A}_{K,t_x})$, which we abusively denote by $E|_x$. 
	Moreover, the extension of scalars of $E|_x$ to $\mathcal{R}_K$ is isomorphic to $\mathcal{E}|_x$ and therefore $\mathcal{L}|_x$. 
    
    Let $\UU,\XX$ be the $p$-adic completion of $U,X$.
    Let $V$ be a strict neighborhood of $\UU^{\rig}$ in $\XX^{\rig}$, on which $\mathcal{L}$ can be realised as a $\cG$-bundle with connection. 
    For each $x\in S_K$, we have an isomorphism between the restriction of $\mathcal{L}$ and of $E|_x$ on the intersection of $V$ with the open unit ball around $x$ in $\XX^{\rig}$ (\ref{sss:dagger-connection}). 
    Then the data $(\mathcal{L},\{E|_x\}_{x\in S_K})$ can be glued to an analytic $\cG$-bundle $L^{\an}$ on $\XX^{\rig}$ and a connection with poles of finite order at $x\in S_K$. 
	By the GAGA theorem, we obtain an algebraic $\cG$-bundle with connection $L$ over $U_K$ such that $L^{\dagger}\simeq \mathcal{L}$. 
	Moreover, we have $L|_{x}\simeq E|_x$ as formal connections. 
	By the physical rigidity of $E$, we deduce that $L\simeq E$ in $\Conn(U_K)$ after enlarging $K$. 
	Then we obtain $\mathcal{L}\simeq \mathcal{E}$ via the functor $(-)^{\dagger}$.
\end{proof}

We will apply the above theorem to the following results of the second author: 
\begin{theorem}[\cite{Yi24} Theorem 1, \cite{CY24} Theorem 1, \cite{Yi25} Theorem 41]
	Under the assumption of \Cref{t:compare-Kl-theta-connection}, the $\cG$-connections $\Kl_{\cG}^{\dR}(\phi)$ \eqref{sss:Kl-dR}, $\Ai_{\cG}^{\dR}(\chi)$ \eqref{sss:Ai-rig} are physically rigid. 
	\label{t:rigidity-FG}
\end{theorem}
\begin{proof}
	When $\cG$ is of adjoint type, the theorem is proved in \textit{loc.cit}. 
	In general, as observed in \cite[4.3.6]{XZ22}, up to isomorphisms, there exists a unique de Rham $\cG$-local system on $\mathbb{G}_{m,\overline{K}}$ (resp. $\mathbb{A}^1_{\overline{K}}$), which induces $\Kl_{\cG}^{\dR}(\phi)$ (resp. $\Ai_{\cG}^{\dR}(\chi)$), and has unipotent monodromy at $0$. 
	This allows us to deduce the general case from the adjoint case. 
\end{proof}

By \Cref{t:physical-rigidity}, we deduce that:

\begin{coro} \label{c:strong-phy-rigidity}
	Under the assumption of \Cref{t:compare-Kl-theta-connection}, the $\cG$-overconvergent isocrystals $\Kl_{\cG}^{\rig}(\phi)$ \eqref{sss:Kl-dR}, $\Ai_{\cG}^{\rig}(\chi)$ \eqref{sss:Ai-rig} are physically rigid. 
\end{coro}


\subsection{Global-local relationship of $\Kl_{\cG}^{\rig}(\phi)$}

\begin{secnumber}
	In this subsection, we assume $G$ is almost simple split over $\mathcal{O}_K$, $m=h$ is the Coxeter number and that $p$ does not divide the order of the Weyl group. 

	For a stable function $\phi\in \mathfrak{g}_1^*(k)$, we can associate the overconvergent isocrystal $\Kl_{\cG}^{\rig}(\phi)$. 
	If $\widetilde{\phi}$ is a lift of $\phi$ to $\mathcal{O}_K$,
	then it has nonzero restriction to every root subspace in $\mathfrak{g}_1$, 
	so that $\widetilde{\phi}_K$ is also stable.
	The underlying connection of $\Kl_{\cG}^{\rig}(\phi)$ is isomorphic to $\Theta(X_{-\pi})$ such that $-\pi\widetilde{\phi}_K$ and $X_{-\pi}$ match under the isomorphism $\mathfrak{g}_1^*/\!\!/ G_0\simeq \cg_1/\!\!/\cG_0$ \eqref{eq:theta-connection-Frob}. 
\end{secnumber}
\begin{lemma}
	Let $\phi,\phi'\in \fg_1^*(k)$ be  two stable functions. 
	If $\phi=\phi'$ in the GIT quotient $(\fg_1^*/\!\!/G_0)(k)$, then $\Kl_{\cG}^{\rig}(\phi)$ (resp. $\Kl_{\cG}^{\ell}(\phi)$) is isomorphic to $\Kl_{\cG}^{\rig}(\phi')$ (resp. $\Kl_{\cG}^{\ell}(\phi')$) as $\cG$-overconvergent isocrystals on $\mathbb{G}_{m,k}/\overline{K}$ (resp. $\ell$-adic $\cG$-local systems on $\mathbb{G}_{m,\overline{k}}$). 
	\label{l:GIT-quotient-classes}
\end{lemma}
\begin{proof}
	If $\phi=\phi'$ in the GIT quotient, then we can take liftings $\widetilde{\phi},\widetilde{\phi}'$ over $\mathcal{O}_K$ with the same image in $(\fg_1^*/\!\!/G_0)(\mathcal{O}_K)$. 
	By \Cref{t:compare-Kl-rig-dr,t:compare-Kl-theta-connection}, $\Kl_{\cG}^{\rig}(\phi)$ and $\Kl_{\cG}^{\rig}(\phi')$ are isomorphic as $\cG$-overconvergent isocrystals. 

	The $\ell$-adic case follows from \Cref{p:companion-Kl} and the $p$-adic case. 	
\end{proof}
\begin{secnumber}
Via the $p$-adic local monodromy theorem, we can associate to $\Kl_{\cG}^{\rig}(\phi)|_{\infty}$ a simple wild parameter $\rho_{\phi}:W_F\to \cG$ by \Cref{t:L-parameter}. 
By composing with $\cG\to \cG_{\ad}$ and with $\Xi$ \eqref{eq:parameters-Hom}, we obtain maps of sets in view of the above lemma:
	\begin{eqnarray}
		\label{eq:global-local-Kl}	
		\Pi:\{\textnormal{overconvergent isocrystals}~ \Kl_{\cG}^{\rig}(\phi)| \textnormal{ stable } \phi\in \mathfrak{g}_1^*/\!\!/ G_0(k)\} &\to& \Hom_{\mathbb{F}_p[\Gamma]}(\mathfrak{F}^+,D_1)-\{0\},\\
		\Pi^{\ell}:\{\textnormal{$\ell$-adic $\cG$-local systems}~ \Kl_{\cG}^{\ell}(\phi)| \textnormal{ stable } \phi\in \mathfrak{g}_1^*/\!\!/ G_0(k)\} &\to& \Hom_{\mathbb{F}_p[\Gamma]}(\mathfrak{F}^+,D_1)-\{0\}. \nonumber
	\end{eqnarray}
\end{secnumber}

\begin{prop} \label{p:isomorphisms-class-phi}
	\textnormal{(i)} 
	The above map $\Pi$ (resp. $\Pi^{\ell}$) is a bijection.	
	
	\textnormal{(ii)}
	Let $\phi,\phi'\in \fg_1^*(k)$ be two stable functions. 	
	Then $\Kl_{\cG}^{\rig}(\phi)$ (resp. $\Kl_{\cG}^{\ell}(\phi)$) is isomorphic to $\Kl_{\cG}^{\rig}(\phi')$ (resp. $\Kl_{\cG}^{\ell}(\phi')$) as $\cG$-overconvergent isocrystals on $\mathbb{G}_{m,k}/\overline{K}$ (resp. $\ell$-adic $\cG$-local systems on $\mathbb{G}_{m,\overline{k}}$) if and only if $\phi=\phi'$ in the GIT quotient $(\fg_1^*/\!\!/G_0)(k)$. 
\end{prop}
\begin{proof}
	We prove the theorem in the $p$-adic case. 
	The $\ell$-adic case follows from \Cref{p:companion-Kl} and the $p$-adic case. 	
	
	(i) By the trivial functoriality \cite[\S~4.1.7]{XZ22}, we may assume $\cG$ is of adjoint type. 
	
	Let $a\in \mathbb{F}_q^{\times}$, $\widetilde{a}$ its Teichmüller lift in $K$ and $[a]:\Gm \to \Gm,~x\mapsto a\cdot x$. 

	Let $X=X_1+X_{1-h}\in \cfg(k)$ correspond to $\phi$ under the isomorphism \eqref{eq:isomorphism-VP-cg1}. 
	In view of the underlying connection \eqref{eq:theta-connection-Frob} of $\Kl_{\cG}^{\rig}(\phi)$, $[a]^*\Kl_{\cG}^{\rig}(\phi)$ admits an underlying connection:
	\[
	d+(X_1+ \widetilde{a}(-\pi)^h X_{1-h} x)\frac{dx}{x}.
	\]

	If we set $X_a=X_1+a X_{1-h}$ and $\phi_a\in (\fg_1^*/\!\!/G_0)(k)$ be the image of $X_a$ under \eqref{eq:isomorphism-VP-cg1}, 
	then we have $[a]^*\Kl_{\cG}^{\rig}(\phi)\simeq \Kl_{\cG}^{\rig}(\phi_a)$. 

	On the other hand, the corresponding action of $a\in \mathbb{F}_q^{\times}$ on $\Hom_{\mathbb{F}_p[\Gamma]}(\mathfrak{F}^+,D_1)-\{0\}$ is studied in \eqref{eq:different-psi} and is transitive. 
	Recall from Remark \ref{r:transitive action} that $\Pi$ is equivariant for this action.
	Thus $\Pi$ is a surjective map. 
	Since both sides have cardinality $q-1$, the assertion follows. 

	(ii) Suppose $\Kl_{\cG}^{\rig}(\phi)$ is isomorphic to $\Kl_{\cG}^{\rig}(\phi')$. 
	Via $\Pi$, the associated local monodromy representations are isomorphic. 
	Note that assertion (i) also shows 
	the domain of $\Pi$ is in natural bijection with
	stable elements of $\mathfrak{g}_1^*/\!\!/G_0(k)$.
	We deduce that $\phi=\phi'$ in the GIT quotient $(\fg_1^*/\!\!/G_0)(k)$. 
\end{proof}

\begin{coro} \label{p:bijection-simple-wild}
	Assume $p$ does not divide the order of the Weyl group $\bW$ and $\cG$ is of adjoint type.  

	\textnormal{(i)} The map $\Pi$ induces a bijection: 
	\[
	\{\textnormal{$p$-adic isoclinic connections of slope $\frac{1}{h}$}\}/ \sim \xrightarrow{\sim} 
	\{\textnormal{non-trivial $\mathbb{F}_p[\Gamma]$-linear maps}\quad \psi:\mathfrak{F}^+\to D_1\}.
	\]	
	
	\textnormal{(ii)} The above map induces a bijection between the sets of isomorphism classes:
	\[
	\{\textnormal{$p$-adic isoclinic connections of slope $\frac{1}{h}$ with Frobenius}\}/\sim \quad \xrightarrow{\sim} \quad \{\textnormal{simple wild parameters}\}/\sim.
	\]
\end{coro}
\begin{proof}
	(i) The equivalence of categories $\MCF(\mathcal{R}_{\overline{K}})\simeq \Rep(I_F\times \mathbb{G}_a)$ (\ref{sss:MCF}) sends $p$-adic isoclinic connections of slope $\frac{1}{h}$ to $\varphi|_{I_F}:I_F\to \cG(\overline{K})$ coming from a simple wild parameter $\varphi$. 
	By \Cref{p:isomorphisms-class-phi}(ii), the left hand side has the cardinality at least $q-1$. 
	Then the assertion follows from \Cref{p:map-Xi}. 

	(ii) In both sides, the Frobenius structures can be modified in the same way as in \ref{sss:modification-Frob}. Then we deduce the assertion from (i).  
\end{proof}

\subsection{$\ell$-adic case}

Suppose $m=h$ is the Coxeter number. 
Let $\ell$ be a prime different from $p$. We deduce the ($\ell$-adic) physical rigidity of $\Kl_{\cG}^{\ell}$ as conjectured in \cite[Conjecture 7.1]{HNY} from its $p$-adic variant (\Cref{t:physical-rigidity}) under suitable assumption:

\begin{theorem}
	Assume that $p$ does not divide the order of the Weyl group $\bW$.  
	Let $\phi$ be a stable function of $\fg^*_1(k)$ and $\cE$ an $\ell$-adic $\cG$-local system on $\mathbb{G}_{m,k}$ such that $\cE|_x\simeq \Kl_{\cG}^{\ell}(\phi)|_x$ as representations of the inertia subgroup $I_x$ for $x\in \{0,\infty\}$. 
	Then $\cE\simeq \Kl_{\cG}^{\ell}(\phi)$ over $\mathbb{G}_{m,\overline{k}}$. 
	\label{t:physical-rig-l}
\end{theorem}
\begin{proof}
	We first show that 
	there is a stable function $\phi'$ of $\fg^*_1(k)$ such that $\cE\simeq \Kl_{\cG}^{\ell}(\phi')$ over $\mathbb{G}_{m,\overline{k}}$.

	We take isomorphisms $\overline{\mathbb{Q}}_{\ell}\simeq \mathbb{C}\simeq \overline{K}$. 
	Let $\cE^{\dagger}$ (resp. $\Kl_{\cG}^{\rig}(\phi)$) be the $p$-adic companion of $\cE$ (resp. $\Kl_{\cG}^{\ell}(\phi)$ \Cref{p:companion-Kl}). 
	Since $\cE|_{\infty}$ is a simple wild representation and $\cE|_0$ is a tame representation with unipotent monodromy, the same holds for the Weil--Deligne representation associated to $\cE^{\dagger}|_x$ for $x\in \{0,\infty\}$. 
	By \Cref{p:bijection-simple-wild}, as $\cG_{\ad}$-differential modules over the Robba ring, $\cE^\dagger_{\cG_{\ad}}|_{\infty}$ is isomorphic to a $p$-adic isoclinic connection of slope $\nu=\frac{1}{h}$. 
	Then, we can apply \Cref{c:strong-phy-rigidity,p:bijection-simple-wild} to deduce that $\cE^{\dagger}\simeq \Kl_{\cG_{\ad}}^{\rig}(\phi')$ as $\cG_{\ad}$-valued overconvergent isocrystals for some stable function $\phi'\in \fg_1^*(k)$. 
	As $\cG$-overconvergent isocrystals, $\cE^{\dagger}$ is different from $\Kl_{\cG}^{\rig}(\phi')$ by a $Z_{\cG}$-overconvergent isocrystal $\chi$ on $\mathbb{A}^1_k$. 
	Since $Z_{\cG}$ is finite, $\chi$, equipped with a Frobenius structure, corresponds to a character of $\pi_1(\mathbb{A}^1_k)$ to $Z_{\cG}$. 
	Since $p$ does not divide the order of $Z_{\cG}$, 
	this character induces a tame character at $\infty$ and $\chi$ is therefore trivial as a $Z_{\cG}$-overconvergent isocrystal. 
	Hence, we deduce that $\cE^{\dagger}\simeq \Kl_{\cG}^{\rig}(\phi')$ as $\cG$-overconvergent isocrystals. 
	By \Cref{p:companion-Kl}, $\cE\simeq \Kl_{\cG}^{\ell}(\phi')$ as $\ell$-adic $\cG$-local systems. 

	The local monodromy representations of $\Kl_{\cG}^{\ell}(\phi), \Kl_{\cG}^{\ell}(\phi')$ are isomorphic. 
	By \Cref{p:isomorphisms-class-phi}(i), we deduce that $\Kl_{\cG}^{\ell}(\phi)\simeq \Kl_{\cG}^{\ell}(\phi')$. This finishes the proof. 
\end{proof}

\begin{rem}
	There is an alternative argument for $\ell$-adic physical rigidity 
	when $\cG=\SO_{2n+1}$ for $n\geq 4$ or $\Sp_{2n}$,
	using the functoriality of $\Kl_{\cG}^{\ell}(\phi)$
	proved in \cite{XZ22}.
	
	Let $\cE$ be an $\ell$-adic $\check{G}$-local system
	with the same local monodromies as $\Kl_{\cG}^{\ell}(\phi)$
	at $0,\infty$.
	Consider their push-out along 
	the standard representation $\mathrm{Std}$ of $\check{G}$.
	By \cite[Theorem 5.1.4, \S5.2.2, Proposition 5.2.10]{XZ22},
	$\Kl_{\cG}^{\ell}(\phi)(\mathrm{Std})$ is an irreducible hypergeometric sheaf,
	which is known to be physically rigid.
	Thus over $\overline{k}$,
	\[
	\cE_{\overline{k}}(\mathrm{Std})
	\simeq\Kl_{\cG}^{\ell}(\phi)_{\overline{k}}(\mathrm{Std}).
	\]
	Moreover, for the above $\check{G}$, 
	the geometric monodromy group of $\Kl_{\cG}^{\ell}(\phi)$ equals to $\check{G}$
	\cite[Theorem 3]{HNY} \cite[Theorem 1.2.7]{XZ22}.
	The above isomorphism says that the geometric monodromy representations 
	of $\cE$ and $\Kl_{\cG}^{\ell}(\phi)$
	are conjugate by some element $g\in\mathrm{GL}(\mathrm{Std})$.
	In particular, $g$ normalizes the geometric monodromy group
	$\check{G}\subset\mathrm{GL}(\mathrm{Std})$.
	Since $\mathrm{SO}_{2n+1}$ and $\mathrm{Sp}_{2n}$
	are self-normalizing up to center inside $\mathrm{GL}(\mathrm{Std})$,
	we can replace $g$ by an element of $\check{G}$,
	which gives the desired isomorphism $\cE_{\overline{k}}\simeq\Kl_{\cG}^{\ell}(\phi)_{\overline{k}}$.
\end{rem}

\appendix

\section{Comparison between algebraic and arithmetic $\mathscr{D}$-modules}

\subsection{From algebraic $\mathscr{D}$-modules to arithmetic $\mathscr{D}$-modules}
\begin{secnumber}
	We follow the construction of Huyghe--Schmidt in \cite[\S~3]{HS}. 
	Let $X$ be a smooth $\OK$-scheme and $\XX$ the $p$-adic completion of $X$. There exist canonical morphisms of ringed spaces:
\begin{displaymath}
	\XX\xrightarrow{\alpha} X \xleftarrow{j} X_K.
\end{displaymath}
	
	We can define a functor 
\begin{equation} \label{eq:functor-}
	\overline{(-)}: \Qcoh(X_K)\to \Mod(\overline{\mathscr{O}}_{X_K}),\quad \mathscr{E}\mapsto \overline{\mathscr{E}}:=\alpha^{-1}(j_*(\mathscr{E})).
\end{equation}

Note that we have $\mathscr{O}_{X}[\frac{1}{p}]\simeq j_{*}(\mathscr{O}_{X_K})$ and $\mathscr{D}_{X}^{(m)}[\frac{1}{p}]\simeq j_{*}(\mathscr{D}_{X_K})$, where $\mathscr{D}_X^{(m)}$ is 
the sheaf of (arithmetic) differential operators of finite level $m\in \mathbb{Z}_{\ge 0}$ on $X$ \cite{Ber96II}. 
Let $\mathscr{D}^{\dagger}_{\XX}$ be the sheaf of overconvergent arithmetic differential operators on $\XX$. 
Then we obtain injective flat homomorphisms \cite[Lemma 3.2.3]{HS}:
\[
	\overline{\mathscr{O}}_{X_K}\to \mathscr{O}_{\XX}[\frac{1}{p}],\qquad 
	\overline{\mathscr{D}}_{X_K}\to \mathscr{D}^{\dagger}_{\XX,\mathbb{Q}}.
\]
By applying \eqref{eq:functor-} and the tensor product functors with the above homomorphisms, we obtain exact functors \cite[Lemma 3.2.4]{HS}:
\begin{equation} \label{eq:hat-functor-D-mods}
	\Coh(\mathscr{O}_{X_K})\to \Coh(\mathscr{O}_{\XX}[\frac{1}{p}]),\quad 
	\widehat{(-)}: \Coh(\mathscr{D}_{X_K})\to \Coh(\mathscr{D}_{\XX,\mathbb{Q}}^{\dagger}).
\end{equation}
\end{secnumber}

\begin{prop}[\cite{HS} 3.3.2,  3.2.6] \label{p:pushforward compare}
	Let $f:X\to Y$ be an $\OK$-morphism of smooth $\OK$-schemes, $\widehat{f}:\XX\to \YY$ its $p$-adic completion and $\mathscr{M}\in \rD^{b}_{\coh}(\mathscr{D}_{X_K})$. 

	\textnormal{(i)} There exists a canonical morphism in $\rD^+(\mathscr{D}_{\YY,\bQ}^{\dagger})$:
	\begin{equation}
		\mathscr{D}_{\YY,\bQ}^{\dagger}\otimes_{\overline{\mathscr{D}}_{Y_K}} \overline{f_{K,+}(\mathscr{M})}\to
		\widehat{f}_+ (\mathscr{D}_{\XX,\bQ}^{\dagger}\otimes_{\overline{\mathscr{D}}_{X_K}}\overline{\mathscr{M}})
	\end{equation}

	\textnormal{(ii)} If $f$ is moreover projective, then the above map is an isomorphism in $\rD_{\coh}^b(\mathscr{D}_{\YY,\bQ}^{\dagger})$.
\end{prop}

\begin{rem}
	In \cite[3.2.6]{HS}, the authors prove assertion (ii) assuming that $X,Y$ are smooth projective and $f$ is proper. However, the same argument works under the above assumption. 
\end{rem}

\subsection{Comparison between overholonomic arithmetic $\mathscr{D}$-modules and algebraic $\mathscr{D}$-modules}

	For a variety $Z$ over $k$, we denote by $\Hol(Z)$ the abelian category of overholonomic arithmetic $\mathscr{D}$-modules, which can be equipped with a Frobenius structure (these objects are also called $F$-able overholonomic $\mathscr{D}$-modules in \cite{CD25}). 
	Recently, Caro-D'Addezio generalized Kedlaya's fully faithful result for $F$-isocrystals to arithmetic $\mathscr{D}$-modules \cite[Theorem 8.2.4]{CD25}. 
	The following proposition is a variant of their result. 

\begin{prop} \label{p:fully-faithful-Ddagger}
	Let $Y$ be a smooth quasi-projective scheme over $\mathcal{O}_K$. 
	There exists a fully faithful functor 
	\begin{equation} \label{eq:Hol-to-Ddagger-mods}
	\Hol(Y_k)\to \Coh(\mathscr{D}_{\YY,\mathbb{Q}}^{\dagger}).
	\end{equation}
\end{prop}

\begin{proof}
	Let $Y\hookrightarrow P$ be a closed embedding into a smooth $\mathcal{O}_K$-scheme and $P\to \overline{P}$ an open embedding into a smooth projective $\mathcal{O}_K$-scheme. 
	Let $\overline{Y}_k$ be the closure of $Y_k$ in $\overline{P}_k$ and $Z=\overline{Y}_k-Y_k$. 
	With the notation of \cite[\S~8.2]{CD25}, 
	there is an equivalence
	\[
		\Hol(Y_k)\simeq \rH_F(\overline{Y}_k,\overline{\PP},Z).
	\]
	By \cite[Theorem 8.2.4]{CD25}, the restriction to $\PP$ defines a fully faithful functor:
	\[
	\rH_F(\overline{Y}_k,\overline{\PP},Z) \to \rH_F(Y_k,\PP).
	\]
	By the Berthelot--Kashiwara equivalence, the category $\rH_F(Y_k,\PP)$ is equivalent to a full subcategory of $\Coh(\Ddd_{\YY,\bQ})$. 
	Composing these functors gives \eqref{eq:Hol-to-Ddagger-mods} and it is fully faithful.
	
	Finally, we show that \eqref{eq:Hol-to-Ddagger-mods} is independent (up to canonical isomorphism) of the choice of $Y\hookrightarrow P\hookrightarrow \overline{P}$.
	Let $Y\hookrightarrow P'\hookrightarrow \overline{P}'$ be another choice and set
	\[
		\overline{P}'':=\overline{P}\times_{\OK}\overline{P}',\qquad P'':=P\times_{\OK}P'\subset \overline{P}''.
	\]
	Then $\overline{P}''$ is smooth projective over $\OK$ and $P''$ is a smooth open subscheme.
	The diagonal map induced by the two immersions gives a closed immersion $Y\hookrightarrow P''$.
	Let $\overline{Y}''_k$ be the closure of $Y_k$ in $\overline{P}''_k$ and set $Z'':=\overline{Y}''_k\setminus Y_k$.
	The projections $\overline{P}''\to \overline{P}$ and $\overline{P}''\to \overline{P}'$ induce canonical equivalences of frame categories
	\[
		\rH_F(\overline{Y}_k,\overline{\PP},Z)\ \xleftarrow{\ \sim\ }\ \rH_F(\overline{Y}''_k,\overline{\PP}'',Z'')\ \xrightarrow{\ \sim\ }\ \rH_F(\overline{Y}'_k,\overline{\PP}',Z'),
	\]
	compatible with restriction to $\PP,\PP',\PP''$ and with the Berthelot--Kashiwara equivalences on the closed subscheme $Y$.
	Hence the resulting functors $\Hol(Y_k)\to \Coh(\Ddd_{\YY,\bQ})$ are canonically isomorphic.
\end{proof}

\begin{definition} \label{d:associated}
	Let $Y$ be a smooth quasi-projective scheme over $\mathcal{O}_K$. 
	Let $E$ be an algebraic $\mathscr{D}$-module in $\Coh(\mathscr{D}_{Y_K})$ and $\mathcal{E}$ an arithmetic $\mathscr{D}$-module in $\Hol(Y_k)$. We say $(E,\mathcal{E})$ are \textit{associated} if, via the functors \eqref{eq:hat-functor-D-mods} and \eqref{eq:Hol-to-Ddagger-mods}, their images in $\Coh(\mathscr{D}^{\dagger}_{\YY,\mathbb{Q}})$ are isomorphic. 
\end{definition}

\begin{secnumber}
	Suppose $X$ is a smooth $\OK$-scheme with a compactification $j:X\to \overline{X}$ into a smooth projective $\OK$-scheme $\overline{X}$ such that the boundary $Z=\overline{X}-X$ is a transversal divisor over $\OK$. 
	Let $\widehat{j}:\XX\to \overline{\XX}$ be the open immersion. 
	We have the natural homomorphisms
\begin{displaymath}
	\widehat{j}_*(\overline{\mathscr{O}}_{X_K})\to \mathscr{O}_{\overline{\XX},\mathbb{Q}}^{\dagger}(Z_k) \to \widehat{j}_*(\mathscr{O}_{\XX}), \qquad 
	\widehat{j}_*\overline{\mathscr{D}}_{X_K}\to \mathscr{D}_{\overline{\XX},\mathbb{Q}}^{\dagger}(Z_k) \to \widehat{j}_*(\mathscr{D}_{\XX}^{\dagger}).
\end{displaymath}
	
	By \cite[Theorem 4.3.10(ii)]{Ber96II}, $\mathscr{D}_{\overline{\XX},\mathbb{Q}}^{\dagger}(Z_k) \to \widehat{j}_*(\mathscr{D}_{\XX}^{\dagger})$ is faithfully flat and we have an exact functor
\begin{equation} \label{dagger functor D-mods}
	(-)^{\dagger}: \Coh(\mathscr{D}_{X_K})\to \Coh(\mathscr{D}_{\overline{\XX},\mathbb{Q}}^{\dagger}(Z_k)),\quad M\mapsto \mathscr{D}_{\overline{\XX},\mathbb{Q}}^{\dagger}(Z_k) \otimes_{\widehat{j}_*\overline{\mathscr{D}}_{X_K}} \widehat{j}_*\overline{M}.
\end{equation}
	
	The category $\Hol(X_k)$ can be realized as a full subcategory of $\Coh(\mathscr{D}_{\overline{\XX},\mathbb{Q}}^{\dagger}(Z_k))$. 
	In the above setting, by \Cref{p:fully-faithful-Ddagger}, a pair $(E,\mathcal{E})$ is associated in the sense of \Cref{d:associated} if and only if $E^{\dagger}\simeq \mathcal{E}$ in the category $\Coh(\mathscr{D}_{\overline{\XX},\mathbb{Q}}^{\dagger}(Z_k))$. 
\end{secnumber}

\begin{secnumber}\label{sss:specialization-map}
	Let $f:X\to S$ be an $\OK$-morphism of quasi-projective smooth $\OK$-schemes. 
	Suppose that $f$ admits a \textit{good compactification}, i.e. $f$ can be extended to an $\OK$-morphism $\overline{f}:\overline{X}\to \overline{S}$ of smooth projective $\OK$-schemes such that $\overline{X}-X$ and $\overline{S}-S$ are transversal divisors over $\OK$. 
	Let $(E,\mathcal{E})$ be a pair of associated algebraic (resp. arithmetic) $\mathscr{D}$-modules on $X_K$ (resp. $X_k$). 

	By $\mathscr{D}$-affinity (resp. $\mathscr{D}^{\dagger}$-affinity), the global section functor induces an equivalence of categories $\Coh(\mathscr{D}_{X_K})\simeq \Coh(D_{X_K})$ (resp. $\Coh(\mathscr{D}_{\overline{\XX},\mathbb{Q}}^{\dagger}(Z_k))\simeq \Coh(D_{\overline{\XX},\mathbb{Q}}^{\dagger}(Z_k))$), where $D$ denotes the global section of the sheaf $\mathscr{D}$. 
	We abusively denote $E$ (resp. $\mathcal{E}$) by its global section. 
	Since $\mathcal{E}\simeq D^{\dagger}_{\overline{\XX},\bQ}(\infty) \otimes_{D_{X_K}} E$, we have a natural morphism of complexes of $D_{S_K}$-modules:
	\[
	D_{S_K\leftarrow X_K} \otimes_{D_{X_K}}^L E \to D^{\dagger}_{\overline{\mathfrak{S}}\leftarrow \overline{\XX},\bQ}(\infty) \otimes_{D^{\dagger}_{\overline{\XX},\infty}(\infty)}^L\mathcal{E}. 
	\]
	By taking the inverse of the global section functors, we obtain the (relative) specialization map of $\Coh(\Ddd_{\overline{\mathfrak{S}}}(\infty))$:
	\begin{equation}
		(\rR^i f_{K,+}E)^{\dagger}\to \rR^i f_{k,+} \mathcal{E}. 
		\label{eq:relative-specialization}
	\end{equation}

	The above construction generalizes the relative specialization map in \cite[2.4.7]{XZ22}. 
\end{secnumber}

Finally, we show that being associated is preserved by proper pushforward. 

\begin{prop} \label{p:proper-pushforward}
	Let $f:X\to Y$ be a proper morphism of quasi-projective smooth $\mathcal{O}_K$-schemes. 
	Suppose that $E,\mathcal{E}$ are associated over $X$. 
	Then $\rR^if_{k,+}(\mathcal{E})$ and $\rR^if_{K,+}(E)$ are associated over $Y$. 
\end{prop}

\begin{proof}
	We first note that $f$ is projective. Indeed, we take an immersion $X\to \Pn_{\OK}$ and compose it with the graph $\Gamma_f:X\to X\times_{\OK}Y$:
	\[
	X \to \Pn_Y \to Y.
	\]
	The above composition is proper and $\Pn_Y\to Y$ is separated. 
	Then we deduce that $X\to \Pn_Y$ is proper and is therefore a closed immersion. 

	We consider the following diagram: 
	\[
	\xymatrix{
	\Hol(X_k) \ar[r] \ar[d]_{\rR^i f_{k,+}} & \Coh(\Ddd_{\XX,\bQ}) \ar[d]_{\rR^i \widehat{f}_+} & \Coh(\cD_{X_K}) \ar[d]_{\rR^i f_{K,+}} \ar[l] \\
	\Hol(Y_k) \ar[r] & \Coh(\Ddd_{\YY,\bQ}) & \Coh(\cD_{Y_K}) \ar[l]
	}
	\]
	The right square is commutative by \Cref{p:pushforward compare}(ii). 
	And it suffices to show that the left square is commutative.

	By the above factorization $f=\pi\circ i$ and functoriality of pushforward, it is enough to treat the case where
	(i) $f$ is a closed immersion and (ii) $f:\Pn_Y\to Y$ is the projection.

	(i) If $f$ is a closed immersion, the assertion follows from the Berthelot--Kashiwara equivalence
	and the construction of the fully faithful functor \eqref{eq:Hol-to-Ddagger-mods} in the proof of \Cref{p:fully-faithful-Ddagger}:
both arithmetic and algebraic pushforwards are described by the same transfer bimodule after passing to $\Ddd$,
hence the left square commutes.

	(ii) For $f=\pi:\Pn_Y\to Y$, choose (as in the proof of \Cref{p:fully-faithful-Ddagger}) a closed immersion $Y\hookrightarrow P$ into a smooth $\OK$-scheme
and a smooth proper compactification $P\hookrightarrow \overline{P}$.
Then $\Pn_P\to P$ is a smooth projective morphism and the arithmetic pushforward $f_{k,+}$ is computed by the transfer bimodule $\Ddd_{\overline{\PP}\leftarrow \Pn_{\overline{\PP}}}$.
Upon restriction from $\overline{P}$ to $P$, the functor $f_{k,+}$ is compatible with the realization of
$\Hol(\Pn_{Y,k})$ inside $\Coh(\Ddd_{\Pn_{\PP},\bQ})$ and with $\widehat{\pi}_+$ on $\Ddd_{\PP,\bQ}$-modules. 
This proves commutativity of the left square in this case.
\end{proof}

\end{document}